\documentclass[a4paper, 11pt]{amsart}

\usepackage[margin=50pt]{geometry}
\usepackage{lipsum}

\usepackage{CJKutf8}

\usepackage{mathpazo}
\usepackage{textcomp}
\usepackage{nicefrac} 
\usepackage{xfrac}  
\usepackage{stmaryrd}

\usepackage{enumitem}
\usepackage{mathtools}
\usepackage[utf8]{inputenc}
\usepackage[all,matrix,arrow,curve]{xy}
\usepackage{amssymb}
\usepackage{amsfonts}
\usepackage{amsthm}
\usepackage{adjustbox}

\usepackage{float}
\usepackage{tabularx}

\usepackage{bbding}

\usepackage{tikz-cd}
  \usepackage{tikz}
  \usepackage{multicol}
  \usepackage[linguistics]{forest}
        \usepackage{tikz-qtree}
        
        \forestset{sn edges/.style={for tree={parent anchor= north, child anchor=south}}}

\forestset{
fairly nice empty nodes/.style={
delay={where content={}
{shape=coordinate, }{}},
for tree={s sep=4mm}
}
}
\forestset{
nice empty nodes/.style={
    for tree={calign=fixed edge angles},
    delay={where content={}{shape=coordinate, for children={anchor=south} }{}}
},
}

\def\be{\begin{equation}}
\def\ee{\end{equation}}
\def\ba{\begin{eqnarray}}
\def\ea{\end{eqnarray}}
\def\id{\mathrm{id}}

\def\CE{{\mathcal E}}

\def\CA{\mathcal A}

\def\r{\mathfrak{r}}

\def \K{\mathbb K}
\def \otimesk {\otimes_{\mathbb K}}
\def\CO{\mathcal{O}}
\def\CI{\mathcal{I}}
\def\CJ{\mathcal{J}}

\def\CC{\mathcal{C}}

\def\CT{\mathcal{T}ree}

\def\FM{\mathfrak M}

\def \CEb{\mathcal E_{\bullet}}
\def \FMb {\mathfrak M_{\bullet}}

\newcommand{\card}{\mathrm{card}}

\newcommand{\SdegFM}{
{S}^{\geq 2}({\mathcal T}ree[\FM_{\bullet}])
}

\newcommand{\Sdeg}{
{S}^{\geq}
}
\newcommand{\TdegFM}{
{{\mathcal T}ree}^{\geq 2}[\FM_{\bullet}]
}
\newcommand{\Tdeg}{
{{\mathcal T}ree}^{ \geq 2}
}

\newcommand{\tdeg}{
{Tree}^{ \geq 2}
}

\newcommand{\p}{\mathrm p^{ \geq 2}}

\def \ptr {{\mathrm{p}}^1_{|}}

\def \pv {{\mathrm{p}}^1_{\vee}}
\newcommand{\wtree}{\scalebox{0.15}{\, \, \begin{forest}
for tree = {grow' = 90}, nice empty nodes, for tree={ inner sep=0 pt, s sep= 0 pt, fit=band, 
},
[[, tier =1] [, tier =1] [, tier= 1]]
\path[fill=black] (.parent anchor) circle[radius=2pt];
\end{forest}}
}
\newcommand{\fourtree}{\scalebox{0.19}{\, \, \begin{forest}
for tree = {grow' = 90}, nice empty nodes, for tree={ inner sep=0 pt, s sep= 0 pt, fit=band, 
},
[[, tier =1] [, tier =1] [, tier= 1][, tier= 1]]
\path[fill=black] (.parent anchor) circle[radius=2pt];
\end{forest}}
}
\def \vtree {\vee}
\newcommand{\ltree}{\scalebox{0.10}{\,\, \begin{forest}
for tree = {grow' = 90}, nice empty nodes,
[
 [
 [, tier =1] 
 [, tier = 1]
 ]
 [, tier =1]
]
\path[fill=black] (.parent anchor) circle[radius=2pt]
                (!1.child anchor) circle[radius=2pt];
\end{forest}
}
}

\newcommand{\rtree}{\scalebox{0.10}{\,\, \begin{forest}
for tree = {grow' = 90}, nice empty nodes,
[
 [, tier =1]
 [
 [, tier =1] 
 [, tier = 1]
 ]
]
\path[fill=black] (.parent anchor) circle[radius=2pt]
                (!2.child anchor) circle[radius=2pt];
\end{forest}
}
}

\newtheorem{theorem}{Theorem}[section]
\newtheorem*{theorem*}{Theorem}
\newtheorem{proposition}[theorem]{Proposition}

\newtheorem{remark}[theorem]{Remark}
\newtheorem{example}[theorem]{Example}
\newtheorem{definition}[theorem]{Definition}
\newtheorem{corollary}[theorem]{Corollary}
\newtheorem{lemma}[theorem]{Lemma}

\newtheorem{convention}[theorem]{Convention}

\newtheorem{construction}[theorem]{Construction}
\newtheorem{question}[theorem]{Question}

\title{Koszul-Tate resolutions and decorated trees}

\author[A. Hancharuk]{Aliaksandr Hancharuk}
\address{Aliaksandr Hancharuk: Department of Mathematics, Jilin University, Changchun 130012, Jilin, China}
\email{hancharukATjlu.edu.cn}

\author[C. Laurent-Gengoux]{Camille~Laurent-Gengoux}
\address{Camille~Laurent-Gengoux: 
Universit\'e de Lorraine, CNRS, IECL, F-57000 Metz, France }
\email{camille.laurent-gengouxATuniv-lorraine.fr}

\author[T. Strobl]{Thomas Strobl}

\address{Thomas Strobl: Institut Camille Jordan, 
Universit\'e Claude Bernard Lyon 1, 69622 Villeurbanne,  France}
\email{stroblATmath.univ-lyon1.fr}

\date{\today}

\begin{document}
\tikzset{mystyle/.style={xshift = - 2.4ex, yshift = - 1.5ex, align=left}}

\tikzset{mystyle2/.style={xshift = - 1.7ex, yshift = - 2.ex, align=left}}

\tikzset{mystyle3/.style={xshift = + 2.4ex, yshift = - 1.5ex, align=left}}

\begin{abstract}\noindent Given a commutative algebra $\CO$, a proper ideal $\CI$, and a resolution of  $\mathcal O/ \mathcal I$ by projective $\mathcal O $-modules, we construct an explicit Koszul-Tate resolution. We call it the arborescent Koszul-Tate resolution since it is indexed by decorated trees.
When the $ \CO$-module resolution has finite length, only  finitely many operations are needed in our constructions---this is to be compared with the classical Tate algorithm, which requires infinitely many such computations if $ \CI$ is not a complete intersection.
  As a by-product of our construction, the initial projective $\CO $-module resolution becomes equipped with an explicit $A_\infty$-algebra. 
 \end{abstract}

\maketitle

\tableofcontents

\section{Introduction}
\label{sec:introduction}
\noindent Let $\mathcal O $ be a commutative algebra and $\mathcal I \subset \mathcal O $ an ideal. Then the quotient $\CO/\CI $ is
\begin{enumerate}
\item[i.] an $\mathcal O $-module (if $ \mathcal O$ is unital, it is even a \emph{cyclic module}, i.e. a module with one generator) and 
\item[ii.]  a commutative algebra.
\end{enumerate}
\noindent
Correspondingly, there are two different meanings of "a resolution of $\mathcal O/\mathcal I $".
 
\begin{enumerate}
\item[i.] If we forget the algebra structure and only keep in mind the $\mathcal O $-module structure, then, as for any $\mathcal O $-module, it is natural to look for a resolution $(\FM_\bullet,d) $ of $ \CO/\CI$ by free or projective  $ \mathcal O$-modules, starting with $ \CO$ itself:
\begin{equation}
    \label{eq:MdO}
\begin{tikzcd}[row sep=2.5em, column sep=2.5em]
  \cdots \arrow[r, "d"] & \FM_{k} \arrow[r, "d"] & \cdots \arrow[r, "d"] & \FM_1 \arrow[r, "d"] & \CO \arrow[r, ] & \CO/\CI \arrow[r, ]& 0.
\end{tikzcd}
\end{equation}
\noindent More precisely, a chain complex as in \eqref{eq:MdO} such that
\begin{enumerate}
    \item every $\FM_k$ is a projective (or free) $\mathcal O $-module and  $d $ is $\mathcal O $-linear, and
     \item the complex has no homology in positive degrees while $H_0(\FM_{\bullet}, d) = \CO/d(\FM_1)=\CO/\CI$, 
    %
\end{enumerate}
is called a projective (or free) \emph{$\mathcal O $-module  resolution} of the quotient $ \CO/\CI$.

\item[] Note that, at least a priori,  $\CO \oplus \FM_\bullet :=\mathcal O \oplus \FM_1 \oplus \FM_2 \oplus \cdots $ is an $\mathcal O $-module, but not a ring or an algebra.

\item[ii.] If we do not want to forget the algebra structure on $\CO/\CI$, then one may look for a free/projective resolution as in \eqref{eq:MdO} such that $\CO \oplus \FM_\bullet  $ admits a graded commutative product compatible with the differential $d$ making it a differential graded commutative algebra.
It is known  that such a commutative product may not exist on a \emph{given} $\CO$-module resolution
$(\FM_\bullet,d)$: See, for example, Appendix 7 in \cite{zbMATH03468959}, which provides a negative answer to a conjecture in \cite{zbMATH03580664}, or see  \cite{zbMATH03699062} and \cite{Katthan2019} for further counter-examples. 

\item[] 
 A particular resolution that also respects the algebra structure, extending the one on $\CO$ in a particular way is what is called a (free or projective) \emph{Koszul-Tate resolution} \cite{Tate} of $\mathcal O/\mathcal I$.
By definition the latter is a pair $(\mathcal A,\delta)$ such that 
 \begin{enumerate}
  \item  $\mathcal A$ is a graded commutative $\mathcal O $-algebra isomorphic to the symmetric algebra generated by some free/projective $\CO$-module $\CE_\bullet = (\CE_i)_{i \geq 1}  $.
     \item $\delta$ is a degree $-1$ derivation of $\mathcal A$ squaring to zero whose homology is zero in all degrees except for in degree $0$, where it is $\CO/\CI $.
 \end{enumerate}
Let us give some explanations on these two conditions.
Condition (a)
means that: 
      $$\begin{array}{rcl}
           {\mathcal A} &:=& S ( \oplus_{i=1}^\infty \CE_i ) 
           \end{array}
           $$
           where the symbol $S$ stands for the graded symmetric $\mathcal O $-algebra generated by the graded $\CO$-module $\CE_\bullet$  and $\odot$ stands for its product (see Convention \ref{conv:tensors} below). In lowest degrees:
           $$
           \begin{array}{rcl}
           {\mathcal A} &=& \underbrace{\CO}_{=\mathcal A_0} \oplus \underbrace{\CE_1}_{=\mathcal A_1}  \oplus  \, \underbrace{\CE_2 \oplus \CE_1\odot \CE_1 }_{=\mathcal A_2} \oplus  \, \underbrace{\CE_3 \oplus \CE_1\odot \CE_2 \oplus \CE_1\odot \CE_1 \odot \CE_1 }_{=\mathcal A_3}   \, \oplus \cdots \: .
      \end{array}$$
    Condition (b)  
  means that  the complex 
  \[
\begin{tikzcd}[row sep=2.5em, column sep=2.5em]
  \cdots \arrow[r, "\delta"] & \CA_{3} \arrow[r, "\delta"] & \CA_2 \arrow[r, "\delta"] & \CA_1 \arrow[r, "\delta"] & \CO \arrow[r, "0"]& 0 
\end{tikzcd}
\]
is an $\CO$-module resolution of $ \CO/\CI$ by free/projective $ \CO$-modules:
      $$ H_i (\mathcal A,\delta )=0 \hbox{ for $ i \geq 1$ and }  H_0 (\mathcal A,\delta ) = \mathcal O/\mathcal I  .$$
      The last condition means that necessarily $\delta ( \mathcal A_1) = \mathcal I  $.
      \end{enumerate}



\vspace{1mm} 
 
\noindent
The existence of free resolutions of any $\CO $-module, in the sense of both item i.\ and item ii.\ above, is a classical result of commutative algebra. 
In the first case and for the $\CO $-module $\CO/\CI $ this works as follows:   
    \begin{construction} $ $ 
    \label{const:free.resolution}
        \normalfont
    \begin{enumerate}
        \item[(a)] We choose generators $(E_j)_{j \in J} $ of the $ \CO$-module $\mathcal I $, labelled by a (not necessarily finite) set $J$. 
        \item[(b)] We define $ \FM_1$ to be the free $\CO $-module generated by elements in $J$, i.e. sums of the form 
        $$ \sum_{j \in J} f_j [j]  $$ 
        where only finitely many of the elements $f_j \in \CO$ are not equal to zero. 
        We define $d \colon \FM_1 \to \CO $ to be the $\mathcal O $-linear map
        \begin{equation}\label{eq:dM1}d \, \colon \,  \sum_{j \in J} f_j [j] \mapsto  \sum_{j \in J} f_j  E_j .\end{equation}
        \item[(c)] We choose generators $ (F_k)_{k \in K} $ of the kernel of $d \colon \FM_1 \to \CO $, then define $\FM_2 $ to be the free $\mathcal O $-module generated by $K$, 
  and we define $d \colon \FM_2 \to \FM_1$ by
        $$d \, \colon \,  \sum_{k \in K} g_k [k] \mapsto  \sum_{k \in K} g_k  F_k .$$
        The procedure continues by recursion. 
    \end{enumerate}   
    \end{construction}

    \vspace{1mm} 
 
\noindent
    Moreover, there are wide classes of algebras $\CO$  for which this construction can be chosen to be "small".
 For instance, if $\CO $ is a Noetherian algebra, then each one of the free $ \CO$-module  $\FM_k $ can be chosen to be finitely generated. Moreover,  some Noetherian algebras $\mathcal O$ have the property that for any ideal $\CI$ the quotient $\CO/\CI$ 
 admit a free $\mathcal O$-module resolution $(\FM,d)$ of finite length, i.e. satisfy that there exists an integer $N$ such that $\FM_k=0 $ for $k \geq N+1$. Noetherian algebras admitting such resolutions of finite length are called \emph{Syzygy algebras}.  
 A classical example of a Syzygy algebra is the algebra  $ \mathbb K[X_1, \dots, X_{N}] $ of polynomials in $N$ variables, with $\mathbb K=\mathbb R $ or $\mathbb C $. Also, germs at $0$ of real analytic functions on $\mathbb R^{N} $, or holomorphic functions on $\mathbb C^{N} $, are Syzygy algebras \cite{Tougeron1976}. There are also many instances of Noetherian algebras which are not Syzygy algebras: for instance the algebra of functions on a singular affine variety in $\mathbb C^n$ is in general not a Syzygy algebra. For practical purposes, especially for Syzygy algebras, one may find explicit $\CO$-module resolutions using the Macaulay software \cite{Macaulay}.
 
    \vspace{1mm} 
 
\noindent
     Koszul-Tate resolutions can be constructed by the  \emph{Tate algorithm}, that we now briefly describe. 
     \begin{construction} Tate algorithm:
         \label{const:tate.algorithm}
         \normalfont
    \begin{enumerate}
     \item[  
     (a)] The module $\mathcal E_1$ can be constructed exactly as the module $\FM_1$ in Construction \ref{const:free.resolution}: $\mathcal E_1 : = \FM_1$ is a free module equipped with an $\CO $-linear map $ d_1: \FM_1 \longrightarrow \CO$ such that $d_1(\FM_1)=\CI$. 
        \item [(b)] The differential $d_1\colon {\mathcal E}_1 \to\CO$ extends to  $S( {\mathcal E}_1)$  as a graded derivation of degree $-1 $ squaring to zero that we denote by $ \delta_1$.  By construction, $H_0( S( {\mathcal E}_1),\delta_1) =\CO/\CI $. Of course,  in general $H_1( S( {\mathcal E}_1),\delta_1) \neq 0$ but it is an $\CO $-module. Hence, the free $\CO $-module  $ {\CE}_2$ generated by generators of $H_1( S( {\mathcal E}_1),\delta_1) $ comes equipped with an $\CO$-linear map $d_2\colon   {\mathcal E}_2  \to S( {\mathcal E}_1)_1 $  whose image is $H_1(S( {\mathcal E}_1), \delta_1)$. 
        The maps $ \delta_1,d_2$ then extend to a degree $-1$ derivation $\delta_2$ of $ S( {\mathcal E}_1 \oplus  {\mathcal E}_2)$  that squares to zero. Now, by construction, 
         $$ H_0( S( {\mathcal E}_1 \oplus  {\mathcal E}_2), \delta_2) = \CO/\CI \hbox{ and } H_1( S( {\mathcal E}_1 \oplus  {\mathcal E}_2), \delta_2)=0  .$$
        \item[(c)] The procedure is repeated in order to construct a sequence  $( {\mathcal E}_{k})_{k \geq 1}$ of free $\CO $-modules and a sequence $d_k  \colon  {\mathcal E}_{k} \to S( {\mathcal E}_{1}\oplus \dots \oplus  {\mathcal E}_{k-1})_{k-1}$ of $\CO$-linear maps such that for every $k > 1$ the image of $d_k$ generates $ H_{k-1}( S( {\mathcal E}_{1}\oplus \dots \oplus  {\mathcal E}_{k-1}), \delta_{k-1}) $. Here $\delta_{k-1}$ stands for the degree $-1$ derivation whose restriction to $ {\mathcal E}_i$ is $ d_i$ for all $i \leq k-1$.
        It is easily proven by recursion that for all $k \geq 1$, $d_k , \delta_{k-1}$ extend to a degree $-1$ derivation such that $\delta_k^2=0 $  and
        $$ H_0\left( S({\mathcal E}_{1}\oplus \dots \oplus  {\mathcal E}_{k} ), \delta_k \right) = \CO/\CI \hbox{ and  }  H_i\left( S({\mathcal E}_{1}\oplus \dots \oplus  {\mathcal E}_{k}), \delta_k \right)=0 \hbox{ for all $i \leq k$ .} $$ 
    \end{enumerate}  
     \end{construction}
\noindent    This implies that the pair 
     $$ (S( {\mathcal E}_{1}\oplus  {\mathcal E}_{2} \oplus \cdots ) , \delta),$$ 
     with $\delta$  being a degree $-1$ derivation whose restriction to $\mathcal E_i $ is $d_i$ for all $i \geq 1 $, is a Koszul-Tate resolution of $\CO/\CI$.
   
 \vspace{1mm}

    \vspace{1mm} 
 
\noindent
    Although both algorithms seem to be similar at first sight, the Tate algorithm used to produce a free Koszul-Tate resolution is in general significantly more complicated than the algorithm constructing an $\CO$-module resolution. Generically, even for a Syzygy algebra $\CO $, there exist infinitely many $\mathcal E_k$ appearing in the Tate algorithm. This means that one has to compute infinitely many homologies and needs to make equally many  choices of generating sets for those homologies.  Let us remark in parentheses that there is also the following more subtle difference: If, when computing an $ \CO$-module resolution, one finds, for example,  $\FM_{16}=0$, then one can stop the computation there. In contrast, for the Tate algorithm, if one arrives at the possible choice $\mathcal E_{16}=0 $, one still can be forced to put $\mathcal E_{17} \neq 0 $; so there is no obvious manner to say "we can stop there". 
    
    \vspace{1mm}

    \noindent 
    In conclusion, \emph{even for Syzygy algebras whose $\CO$-module resolutions can be found by making finitely many operations, the Tate algorithm obliges one to compute infinitely many homologies 
and to do infinitely many choices.}  As a result, for most ideals $\CI$ of an algebra $ \CO$, no Koszul-Tate resolution of $\CO/\CI $ is explicitly known, although it is possible to find an explicit $\CO$-module resolution. 

\vspace{1mm}

\noindent
The purpose of the present article is mainly to answer the following natural question:

\begin{question}
\label{question:howto}
Is there a way to construct a Koszul-Tate resolution by performing only finitely many homology computations (and having therefore only finitely many choices to take)? 
\end{question}

\noindent
We will provide a positive answer to this question 
in the case when an $ \CO$-module resolution of finite length and of finite ranks exists, so, e.g., always if  $\mathcal O$ is a Syzygy algebra.

\vspace{1mm} \noindent In addition, with our construction, there are not only finitely many computations and choices to be made, but we also know when we can stop the construction: at the $N$-th step where $N$ is the length of $(\FM_\bullet,d)$.

\noindent

    \vspace{1mm}

\noindent
We now explain how to answer Question \ref{question:howto}. We start with given a projective $\CO$-module resolution $(\FM_\bullet,d) $ as in \eqref{eq:MdO}. Then we consider the free  $\CO $-module $\CT[\FM_\bullet] $  generated by rooted trees  with leaves decorated by homogeneous elements of $ \FM_\bullet $, e.g. elements of the form: 
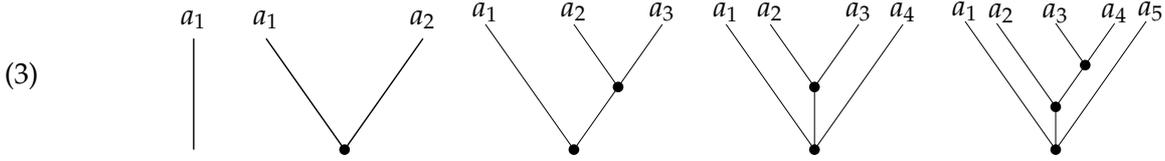
\begin{equation}
    \label{eq:treesdecorated}
    \begin{array}{ccccc}  
    \scalebox{1.33}{
      \begin{forest}
for tree = {grow' = 90}, nice empty nodes, for tree={ inner sep=0 pt, s sep= 0 pt, fit=band, 
},
[[\scalebox{0.75}{$a_1$}, tier =1]];
\end{forest} 
}
&
\scalebox{1.33}{
 \begin{forest}
for tree = {grow' = 90}, nice empty nodes, for tree={ inner sep=0 pt, s sep= 0 pt, fit=band, 
},
[[\scalebox{0.75}{$a_1$}, tier =1] [\scalebox{0.75}{$a_2$}, tier =1]]
\path[fill=black] (.parent anchor) circle[radius=1.5pt];
\end{forest} 
}
&
    \scalebox{0.75}  {\begin{forest}
for tree = {grow' = 90}, nice empty nodes, for tree={ inner sep=0 pt, s sep= 0 pt, fit=band, 
},
[
    [\scalebox{1.33}{$a_1$}, tier =1]
    [
        [\scalebox{1.33}{$a_2$}, tier =1] 
        [\scalebox{1.33}{$a_3$}, tier =1]
    ]
]
\path[fill=black] (.parent anchor) circle[radius=2.67pt]
(!2.child anchor) circle[radius=2.67pt];
\end{forest}
 }     
&\scalebox{0.75}{
      \begin{forest}
for tree = {grow' = 90}, nice empty nodes, for tree={ inner sep=0 pt, s sep= 0 pt, fit=band, 
},
[
    [\scalebox{1.33}{$a_1$}, tier =1]
    [
        [\scalebox{1.33}{$a_2$}, tier =1] 
        [\scalebox{1.33}{$a_3$}, tier =1]
    ]
    [\scalebox{1.33}{$a_4$}, tier =1]
]
\path[fill=black] (.parent anchor) circle[radius=2.67pt]
(!2.child anchor) circle[radius=2.67pt];
\end{forest}
}
&     
\scalebox{0.5}{
        \begin{forest}
for tree = {grow' = 90}, nice empty nodes, for tree={ inner sep=0 pt, s sep= 0 pt, fit=band, 
},
[
    [\scalebox{2}{$a_1$}, tier =1]
    [
        [\scalebox{2}{$\,\,a_2$}, tier =1] 
        [[\scalebox{2}{$a_3$}, tier =1][\scalebox{2}{$a_4$},tier = 1]]
    ]
     [\scalebox{2}{$\,\,a_5$}, tier =1]
]
\path[fill=black] (.parent anchor) circle[radius=4pt]
(!2.child anchor) circle[radius=4pt]
(!22.child anchor) circle[radius=4pt];
\end{forest}    
}
\end{array}
\end{equation}
with $a_i \in \FM_{\bullet}$ of fixed degree for $i =1, \dots, 5$.
 Such trees have a natural $\CO$-module structure. It becomes graded by declaring the degree of a decorated tree to be the sum of the degrees of the decorations in $\FM_\bullet$ plus the number of vertices. 
 The main point of the present article is to equip the graded symmetric algebra $S(\CT[\FM_\bullet]) $  of $\CT[\FM_\bullet] $, thought of as a forest of such trees,
 with an explicit differential $\delta $ that makes it a Koszul-Tate  resolution of $\CO/\CI $. We call it the \emph{arborescent Koszul-Tate resolution of $\CO/\CI $} (attached to the free or projective $\CO$-module resolution $(\FM_\bullet,d)$). When constructed, this will answer positively Question \ref{question:howto}. 

\vspace{0.1cm}
 \normalfont
\noindent Moreover, this differential has several components, which are of the following types:
 \begin{enumerate}
    \item The first component consists of acting on the decorations attached to the leaves, i.e.\ on the corresponding elements in $\FM_\bullet$.  Here $\delta$ acts in most cases simply as $d$---up to appropriate signs determined from the tree. 
    \item The second component operates on the trees only: 
    Either by removing the root, resulting in a contribution in terms of  a forest, or by 
    removing inner edges (adding appropriate signs). 
     \item  The last and the most intriguing component is the following one: For a tree $t$ with $n$ leaves and $k$ inner vertices it is determined by an $n$-linear map $\psi_t \colon  \otimes^n  \FM_\bullet \to  \FM_\bullet$ of degree $k$. 
 \end{enumerate}
\noindent    
    We will call the maps $\psi_t$, indexed by trees $t$,  \emph{arborescent operations}. Their construction is the only one which 
requires choices. Moreover and most importantly, for degree reasons, if the resolution $(\FM_\bullet,d)$ has finite length, there exist only finitely many non-trivial arborescent operations, resulting in only  a finite number of  computations and  choices.
 
\vspace{0.1cm}
\noindent
It is a natural question to ask what is the structure of these "mysterious" arborescent operations.
A complete answer probably involves operads, but we have an answer which is sufficient to explain their importance. Since any Koszul-Tate resolution of $\CO/\CI $ is in particular a projective  $\CO $-module resolution of $\CO/\CI$, and since  any two projective  $\CO $-module resolutions of $\CO/\CI$ are homotopy equivalent,
$(\FM_\bullet,d)$ is homotopy equivalent to any Koszul-Tate resolution. 
As a consequence, the homotopy transfer theorem \cite{Kontsevich:2000yf} applies
and yields an $A_\infty$-algebra structure on the complex 
\begin{equation}
\label{eq:complexAinfty}
\begin{tikzcd}[row sep=2.5em, column sep=2.5em]
  \cdots \arrow[r, "d"] & \FM_{k} \arrow[r, "d"] & \cdots \arrow[r, "d"] & \FM_1 \arrow[r, "d"] & \CO \arrow[r, ] & 0.
\end{tikzcd}
\end{equation}
\vspace{1mm} 
 
\noindent
This description is quite abstract. However, Theorem \ref{thm:ainfty} below gives explicit formulas for the $A_\infty$-algebra products in terms of the arborescent operations. We show that these products are even those of a $ C_\infty$-algebra, i.e. the graded commutative equivalent of $ A_\infty$-algebras. We give two proofs of these facts: one  using the homotopy transfer theorem and an \emph{ab initio} proof using only the tools developed here.
 \vspace{1mm} 

\vspace{0.1cm}
\noindent
Koszul-Tate resolutions play an essential role in the Batalin-(Fradkin)-Vilkoviskiy formalism:  It was created in the physics literature in \cite{BFV1, BFV2} to handle the omni-present gauge theories, then formalized, e.g., in \cite{HT, stasheff_poisson,Kazdhan-Felder,MuellerLennart}, where Koszul-Tate resolutions are often taken as a given starting point. 
Koszul-Tate resolutions are also used for PDE's with symmetries, see, e.g., \cite{Grigoriev_2019}.
Finally, they also appear, under different names, in deformation theory of algebraic varieties, see, e.g., \cite{Fantechi}.

\vspace{1mm}
\noindent
In the literature there exists an explicit construction of an algebra resolution by Chuang and King \cite{Chuang-King}, which even works for non-commutative $\CO$.
However,
their algebra is not a graded symmetric algebra---and thus not a Koszul-Tate resolution.

\vspace{1mm} \noindent 
In this article, we provide answers to several side questions which are of interest in their own: For instance, we define the notion of a \emph{minimal} Koszul-Tate resolution, which in good cases reduces to a Koszul-Tate resolution with a minimal number of generators. We also introduce a  \emph{reduced complex} associated uniquely to any $\CO$ and $\CI$ which has the following property: 
 Its homology counts the minimal number of generators needed in a Koszul-Tate resolution at each degree. While arborescent Koszul-Tate resolutions are not minimal, they are very useful in determining lower bounds to the number of generators of any Koszul-Tate resolution. 
 
 \vspace{1mm} 
 
\noindent
 Section \ref{sec:arborescent} is devoted to the construction of the arborescent Koszul-Tate resolution, with a particular emphasis on the arborescent operations. Section \ref{sec:Cinfty} explains how some components of an arborescent Koszul-Tate resolution encode an $A_\infty$- and $C_\infty$- algebra on the module resolution 
 Examples of Koszul-Tate  resolutions are given in Section \ref{sec:examples}.
 In Section \ref{sec:complexity} we study  the complexity of the construction of arborescent Koszul-Tate resolutions and compare it with the Tate algorithm. 
 In appendix \ref{app:reduced}, we give a short and complete proof of the up-to-homotopy uniqueness of the Koszul-Tate resolution, and we define its reduced complex mentioned above.
 In Appendix \ref{app:minimal} we introduce the notion of minimality for Koszul-Tate resolutions.


\vspace{5mm} 
 
\noindent
We conclude the introduction by a few conventions:

\begin{convention}
In this article, $ (\FM_\bullet,d)$ denotes an $\CO$-module resolution  of $\CO/\CI $  as in \eqref{eq:MdO}.
Also, $\FM_\bullet$ and $\CO\oplus\FM_\bullet$ denote the graded vector spaces $\FM_1 \oplus \FM_2 \oplus \cdots  $ and  $\CO \oplus \FM_1 \oplus \FM_2 \oplus \cdots  $, respectively, without any reference to the differential $d$.
\end{convention} 


\begin{convention}
\label{conv:tensors}
Given a graded $\CO $-modules $\mathcal E_\bullet$ and $\mathcal F_\bullet$, the notations   $ \mathcal E_\bullet \otimes \mathcal F_\bullet$ and $\mathcal E_\bullet \odot \mathcal F_\bullet $ stand for the graded tensor and graded symmetric tensor product  over the algebra $\CO $, respectively. The tensor and symmetric tensor products over the base field $ \mathbb K$ are denoted by $ \otimes_\mathbb K$ and $\odot_\mathbb K $, respectively. 
For $\mathcal E_\bullet$ a graded $\CO$-module,  $$ S(\mathcal E)= \oplus_{k=0}^\infty S^k(\mathcal E) \vspace{-6mm}$$ 
stands by default for the symmetric algebra over $\CO $. Here, $S^0(\CE):=\CO$ and 
 $S^k(\mathcal E) := \overbrace{\mathcal E_\bullet \odot \cdots \odot \mathcal E_\bullet}^{ \hbox{\normalfont $k$ times}} $. Finally, $S^{\geq 2}(\CE) := \oplus_{k \geq 2} S^{k}(\mathcal E)$.
\end{convention}

\begin{convention}
Let us define what we call an \emph{oriented planar tree $t \in {\mathrm{Tree}}$}. 
It is a connected finite planar graph with oriented edges. Its vertices $A \in \mathrm{Ver}(t)$ are of three possible types: internal vertices, denoted by ${\mathrm{InnVer}}(t) $, have one in-going and at least two out-going edges. The external vertices have either no in-going edge, in which case it is called the \emph{root} of $t$, or no out-going edge, in which case this vertex is called a \emph{leaf} of $t$. Given a vertex $A$ which is not a leaf, we call the vertices connected to $A $ by one out-going  edge the \emph{children} of $ A$. The vertices related to $A$ by a chain of out-going  edges are called the \emph{descendants} of $A$. 

\vspace{1mm} \noindent 
We also assume that the leaves are totally ordered. This induces a total order on the outgoing edges of each inner vertex as well as of the root. 
We say that a tree is \emph{binary} if every vertex $A$ which is not a leaf has precisely two children. Then the first out-going edge is called \emph{left}, and the second one \emph{right}. The whole subtree starting from the left edge of $A$ is said to be the \emph{left subtree}: its root is the child to the left of $A$.  
The \emph{right tree} is defined correspondingly. Notice that the right and left subtrees of a leaf are the empty set.
\end{convention}
 
\section{Arborescent Koszul-Tate resolutions associated to 
a module resolution $(\FM_\bullet,d)$}

\label{sec:arborescent}

\noindent In this section, we present an explicit and constructive way to obtain a Koszul-Tate resolution of $\CO/\CI $ out of an $\CO$-module resolution $(\FM_\bullet,d) $ of $ \CO/\CI$ as in \eqref{eq:MdO}. 

\vspace{1mm} 
 
\noindent
In Section \ref{sec:space}, we describe the free graded commutative algebra $S \left(\CT[\FM_\bullet]  \right)$ generated by the $\CO$-module $\CT[\FM_\bullet] $ of trees with leaves decorated by elements of $\FM_\bullet$. 
In Section \ref{sec:differential}, we explain how to equip the graded commutative algebra
$S \left(\CT[\FM_\bullet]  \right)$
with a differential $\delta_\psi $, that depends on some recursively constructed multi-linear maps $\psi$ on $\FM_\bullet $, the arborescent operations mentioned above.
Theorem \ref{thm:isKT} shows that the homology of $\delta_\psi $ is zero, hence proving that we indeed obtain a Koszul-Tate resolution in this way, an arborescent Koszul-Tate resolution of $\CO/\CI$.

\subsection{Arborescent Koszul-Tate resolutions: the graded algebra of decorated trees}

\label{sec:space}

Throughout this section, $\FM_\bullet = (\FM_i)_{i \geq 1} $ stands for an arbitrary sequence   of $\CO $-modules.
We define the  $\CO$-module ${\mathcal T}ree[\FM_\bullet] $ of decorated trees:
\begin{enumerate}
    \item [(i)] First, let us consider the set of \emph{planar rooted trees}, i.e. the set of rooted trees with ordered leaves, satisfying the property that each vertex has at least two children. The set is enlarged by adding a trivial tree. Typical examples are listed below:
\vspace{2mm}    
    
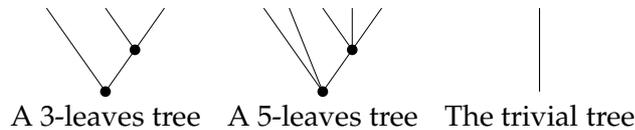
    \begin{center}
\begin{tabular}{ccc} 
    \scalebox{0.5}{  \begin{forest}
for tree = {grow' = 90}, nice empty nodes, for tree={ inner sep=0 pt, s sep= 0 pt, fit=band, 
},
[
    [, tier =1]
    [
        [, tier =1] 
        [, tier =1]
    ]
]
\path[fill=black] (.parent anchor) circle[radius=4pt]
(!2.child anchor) circle[radius=4pt];
\end{forest}
}
&
   \scalebox{0.5}{   \begin{forest}
for tree = {grow' = 90}, nice empty nodes, for tree={ inner sep=0 pt, s sep= 15 pt, fit=band, 
},
[
    [, tier =1]
      [, tier =1]
    [
        [, tier =1] 
        [, tier =1]
        [, tier =1]
    ]
]
\path[fill=black] (.parent anchor) circle[radius=4pt]
(!3.child anchor) circle[radius=4pt];
\end{forest}
}
&          \begin{forest}
for tree = {grow' = 90}, nice empty nodes, for tree={ inner sep=0 pt, s sep= 0 pt, fit=band, 
},
[[, tier =1]]
;
\end{forest} \\
A $3$-leaves tree & A $5$-leaves tree & The trivial tree
\end{tabular}
\end{center}

     We denote  the free $\mathbb K$-vector space  generated by rooted planar trees by $Tree$, those having $n$ leaves by $Tree^{n} $, those with $k$ vertices by $Tree_k$, while $ {Tree}_k^{n}:= Tree^{n} \cap Tree_k$ denotes their intersection.

    We choose the convention that a vertex of a tree is  either an inner vertex or a root, i.e. leaves are not vertices: for instance, the rightmost tree of \eqref{eq:treesdecorated}  has $3$ vertices. By convention also, the trivial tree has $0$ vertices.
     Notice that the vector space generated by the trivial tree, ${Tree}_0^{1}$, is isomorphic to $\mathbb K $.
    
    \item [(ii)] The free $\CO$-module of \emph{ordered decorated trees} is defined as
        \begin{equation}
            Tree[\FM_\bullet] = \bigoplus^{\infty}_{n=1}{Tree}^{n}\otimes_{\mathbb K} \underbrace{\FM_{\bullet}\otimes ... \otimes \FM_{\bullet}}_{n \text{ times}} \, .
        \end{equation}
        The $\CO $-module of ordered decorated trees with $n$ leaves is denoted by ${Tree}^{n}[\FM_\bullet]$, the one  with $k$ vertices by  ${Tree}_k[\FM_\bullet] $ and their intersection by 
         ${Tree}_k^{n}[\FM_\bullet]$.
         For all $ t \in Tree^{n}$ and $ a_1, \dots, a_n \in \FM_{\bullet} $, the element $t \otimes_\mathbb K a_1 \otimes \cdots \otimes a_n$ in ${Tree}^{n}[\FM_\bullet] $ is denoted as
         $$ t \,  [a_1, \dots, a_n]    .$$
          
        We invite the reader to think that $ a_1, \dots, a_n$ are attached to the leaves of $t$, as in \eqref{eq:treesdecorated}. The $\CO $-module structure is given  by
         \begin{align*} F \cdot t \,  [a_1, \dots, a_n] = t  \, [F a_1, \dots, a_n] = \cdots =t \, [a_1, \dots, F a_n] \\ \qquad \forall  F \in \CO \: , \; t \in Tree^{n}\: , \;  a_1, \dots, a_n \in \FM_{\bullet} .\end{align*}
         
         \begin{remark}
\label{rmk:notation}
\normalfont
Elements in $ Tree^{1}_0 \otimes_{\mathbb K} \FM_\bullet$ are denoted as $  |\, [a]$ with $a \in \FM_\bullet$, the vertical bar $|$ referring to the trivial tree.
\end{remark}
    \item[(iii)]  Let us introduce a grading on $Tree$  by declaring   elements in $Tree^{n}_k$ to be homogeneous of degree $k=\hbox{ $\#$ $\{$Vertices of $t \}$}$ and on $Tree[\FM_\bullet]$ by
     $\mathrm{deg} \left( t \, [ a_1 , \dots , a_n]\right) := \mathrm{deg}(t) + \sum_{i=1}^n \mathrm{deg}(a_i)$.
    Upon abbreviating the notation for the degree by $ |\cdot |$, this becomes 
    \begin{equation} \label{treedegree}
         |  t \, [ a_1 , \dots , a_n]| = |a_1| + \dots + |a_n| + \hbox{ $\#$ $\{$Vertices of $t \}$ } .
           \end{equation}
Thus $Tree[\FM_\bullet] = \oplus^{\infty}_{i =1} Tree[\FM_\bullet]_{i} $, where $i$ denotes the degree. In lowest degrees, we have 
    \begin{align*}
    Tree[\FM_\bullet]_1= |[\FM_1] \simeq  \FM_1 \, , \, Tree[\FM_\bullet]_2= |[\FM_2] \simeq  \FM_2 \, ,
    \\
    Tree[\FM_\bullet]_3 =
\begin{gathered}
\scalebox{0.5}{
\begin{forest}
for tree = {grow' = 90}, nice empty nodes,
[
 [, tier =1] 
]
\end{forest}}
\otimes_{\mathbb K}(\FM_3)\enspace \bigoplus
\scalebox{0.5}{
\begin{forest}
for tree = {grow' = 90}, nice empty nodes,
[
 [, tier =1]
 [, tier =1] 
]
\path[fill=black] (.parent anchor) circle[radius=4pt];
\end{forest}}
\otimes_{\mathbb K} (\FM_1 \otimes \FM_1)
\end{gathered}\, ,
\end{align*}
and $Tree[\FM_\bullet]_4 $ is equal to
\begin{equation*}
\hspace{1cm}
\scalebox{0.5}{
\begin{forest}
for tree = {grow' = 90}, nice empty nodes,
[
 [, tier =1] 
]
\end{forest}}
\otimes_{\mathbb K}(\FM_4)\enspace \bigoplus
\scalebox{0.5}{
     \begin{forest}
for tree = {grow' = 90}, nice empty nodes,
[
 [, tier =1]
 [, tier =1] 
 [, tier =1]
] 
\path[fill=black] (.parent anchor) circle[radius=4pt];
\end{forest}}\!\! \otimes_{\mathbb K}(\FM_1 \otimes \FM_1 \otimes \FM_1) \enspace  \bigoplus 
\scalebox{0.5}{
\begin{forest}
for tree = {grow' = 90}, nice empty nodes,
[
 [, tier =1]
 [, tier =1] 
]
\path[fill=black] (.parent anchor) circle[radius=4pt];
\end{forest}}
\!\! \otimes_{\mathbb K} (\FM_2 \otimes \FM_1) \enspace \bigoplus 
\scalebox{0.5}{
\begin{forest}
for tree = {grow' = 90}, nice empty nodes,
[
 [, tier =1]
 [, tier =1] 
] 
\path[fill=black] (.parent anchor) circle[radius=4pt];
\end{forest}}
\!\! \otimes_{\mathbb K}  (\FM_1 \otimes \FM_2).
\end{equation*}
In general,
\begin{equation}
\label{tree_decomposition}
    Tree[\FM_\bullet]_{i}  = \, Tree_0^{1}\otimes_{\mathbb K}\FM_i \, \, \oplus \, \, 
    \tdeg[\FM_\bullet]_{i} 
\end{equation}
where $\tdeg[\FM_\bullet]$ is the $\CO $-module generated by decorated trees with at least two leaves (this only excludes trivial trees). 

\item[(iv)]
Now we are ready to define the $\CO$-module $\CT[\FM_\bullet]$, underlying the to-be-constructed Koszul-Tate resolution on $S(\CT[\FM_\bullet])$. 

\vspace{1mm}

\noindent 
For this purpose, we first define permutations acting on some given $t[a_1, \dots , a_n] \in Tree_k^{n}[\FM_{\bullet}] $ parametrized by the vertices of the underlying tree $t$.  Let $v$ be the number of children of the chosen vertex $V$. These children are themselves ordered decorated trees $$t_1[b_1, \dots , b_{\card(t_1)}] \, , \, \dots\, , \, t_v [b_{\card(t_1)+\cdots \card(t_{v-1})+1}, \dots , b_{m}] \,  $$ of degrees $\theta_1, \dots,\theta_v $, respectively. Here $\card(t_i)$ denotes the number of leaves of the corresponding tree $t_i$, and $m$ is the number of the leaves that descend from $V$. 

\vspace{1mm}

\noindent 
A permutation $\sigma \in S_{v} $ acts by exchanging the position of the children, keeping the rest of the tree untouched. The outcome is a tree that we denote as follows $\sigma \cdot t[a_1, \dots , a_n] = t_{V,\sigma}[a_{\sigma(1)}, \dots , a_{\sigma(n)}] $. Let $\epsilon(t_{V, \sigma}) $ the Koszul sign of the permutation $\sigma $ with respect to the degrees $ \theta_1, \dots, \theta_v$.

\begin{definition} The  $\CO$-module of \emph{decorated trees} as a quotient of the $\CO$-module of ordered decorated trees, 
$$ \mathcal{T}ree[\FM_{\bullet}] := Tree[\FM_{\bullet}]/\!\sim \, , $$ 
by the equivalence relation  
\begin{equation}
     \label{eq:quotientEquiv} t[a_1, \dots , a_n] \, \sim \, \epsilon(t_{V, \sigma}) \,  t_{V, \sigma}  [a_{\sigma(1)} , \dots , a_{\sigma(1)}] .
 \end{equation}
for all vertices $V$ of $t$. \end{definition}

\end{enumerate}

\vspace{1mm}
\noindent
As an example  we regard a tree with five leaves, using the root as the vertex for the permutation:

\begin{center}
\begin{tabular}{ccc}
\adjustbox{valign=c}{
 \begin{forest}
for tree = {grow' = 90}, nice empty nodes, for tree={ inner sep=0 pt, s sep= 10 pt, fit=band, 
},
[
    [$a_1$, tier =1, edge label={node[midway, left]{$t_1$}}]
    [,edge label={node[left]{$t_2$}}
        [$a_2$, tier =1] 
        [$a_3$, tier =1]
        [$a_4$, tier =1]
    ]
    [$a_5$, tier =1, edge label={node[midway, right]{$t_3$}}]
]
\path[fill=black] (.parent anchor) circle[radius=2pt]
(!2.child anchor) circle[radius=2pt];
\end{forest}}
&
\adjustbox{valign=c}{
{\centering  $\sim$}}
& 
\adjustbox{valign=c}{$(-1)^{|a_1|(|a_2|+ |a_3| + |a_4|+ 1)}$}
\adjustbox{valign=c}{
 \begin{forest}
for tree = {grow' = 90}, nice empty nodes, for tree={ inner sep=0 pt, s sep= 10 pt, fit=band, 
},
[
[, ,edge label={node[left]{$t_2$}}
        [$a_2$, tier =1] 
        [$a_3$, tier =1]
        [$a_4$, tier =1]
    ]
    [$a_1$, tier =1,  edge label={node[midway, left]{$t_1$}}]
    [$a_5$, tier =1,  edge label={node[midway, right]{$t_3$}}]
]
\path[fill=black] (.parent anchor) circle[radius=2pt]
(!1.child anchor) circle[radius=2pt];
\end{forest}}.
\end{tabular}
\end{center}


\vspace{0.3cm}

\noindent
Let us give some properties of this quotient:

\begin{proposition}
\label{prop:stillfree}
The $\CO$-module $ \mathcal{T}ree[\FM_{\bullet}] $ of decorated trees is a positively graded $\CO $-module, which is free if the $\CO$-module $\FM_\bullet $ is free and projective if 
 $\FM_\bullet $ is projective.  
\end{proposition}
\begin{proof}
Assume first that $ \FM_\bullet$ is a free $\CO $-module.
Then $Tree[\FM_{\bullet}] $ is a free $\CO$-module as well.
The only difficulty is to show that its quotient $ \mathcal{T}ree[\FM_{\bullet}] $ is still free. It suffices to prove this result for trees with a fixed number $n$ of leaves. 

\vspace{1mm} 
 
\noindent
Let $W$ be a $\mathbb K$-vector space $W' \subset W  $ a vector subspace. Then the quotient of the free $\CO $-module $W \otimes_\mathbb K \CO $ by  $W' \otimes_\mathbb K \CO $ is the free $\CO$-module $ W/W' \otimes_{\mathbb K} \CO$. 
By assumption, $\FM_\bullet $ is free, i.e.\ it can be written as $\FM_\bullet = M_\bullet \otimes_\mathbb K \CO $ where $M_\bullet $ is a graded $\mathbb K$-vector space. Now put $W := 
Tree_n \otimes_\mathbb K M^{\otimes n}_\bullet $, which is a (graded) $\mathbb K$-vector space. 

\vspace{1mm} \noindent Since
$ \sim$ consists in dividing by linear relations whose coefficients are in $\mathbb K $, the equivalence relation can be seen effectively as forming the quotient of the above $W$ by a vector subspace $W'$. So the quotient of the $\CO$-modules becomes $ W/W' \otimes_{\mathbb K} \CO$, which is free. 

\vspace{1mm} 
 
\noindent If on the other hand, $ \FM_\bullet$ is projective, then it can be proven as follows. Let $t$ be a tree with $n$ leaves, and  let $t[\FM_{\bullet}, \dots, \FM_{\bullet}]$  be the sub-$ \CO$-module of $ \CT[\FM_{\bullet}]$ of all elements of the form $ t[a_1, \dots,a_n]$ with $ a_1, \dots, a_n \in \FM_\bullet $. Let $\phi \colon t[\FM_{\bullet}, \dots, \FM_{\bullet} ] \to  \mathcal F $ be an $\CO $-module morphism and let $\phi' \colon \mathcal E \to \mathcal F $ be a surjective $ \CO$-module morphism.  
There is a natural map
$ proj \colon \FM_\bullet^{\otimes n}  \to t[\FM_{\bullet}, \dots, \FM_{\bullet}]$ mapping $ a_1 \otimes \dots \otimes a_n$ to $ t[a_1, \dots, a_n]$.
Since $ \FM_\bullet^{\otimes n}$ is projective, there exists an $ \CO$-module morphism $\psi\colon \FM_\bullet^{\otimes n} \to \mathcal E $
making the following diagram commutative:
 \begin{equation}
 \label{eq:moduleprojective} \xymatrix{ \ar[d]_{proj} \ar[r]^{\psi}\FM_\bullet^{\otimes n}& \mathcal E \ar@{->>}[d]^{\phi'}\\ t[\FM_\bullet, \dots, \FM_\bullet] \ar[r]^{\hspace{.9cm}\phi} & \mathcal F}  \end{equation}
For any inner vertex $A$ of $t$, consider the set $ \Sigma_A$ of permutations  of the leaves of $t$ obtained by permuting  the children of $A$. The $ \CO$-module morphism $ \frac{1}{|\Sigma_A|}\sum_{\sigma \in \Sigma_A} \psi \circ \sigma $ (with the understanding that $ \sigma$ acts with Koszul sign on $\FM^{\otimes n} $) still satisfies the commutativity of
\eqref{eq:moduleprojective}.
When this operation is completed for all vertices, the henceforth morphism $ \psi$, being invariant under all the permutations that define $ proj$, comes from an $ \CO$-module morphism $ t[\FM_\bullet, \dots, \FM_\bullet] \to \mathcal E$. This proves that $t[\FM_\bullet, \dots, \FM_\bullet]$ is a projective $ \CO$-module. Since the direct sum of projective $ \CO$-modules is a projective $ \CO$-module, this proves that $\CT[\FM_\bullet] $ is projective.
 \end{proof}

\noindent
 In order to have a short and consistent definition of a yet to be constructed Koszul-Tate differential, we need to adopt a new convention and allow leaves of a decorated tree to be elements in $\CO$.
 
 \begin{convention} 
\label{conv:extensionO}
Let $t$ be a tree with $n>1$ leaves and $a_1, \dots, a_n  $ elements with are either in $\FM_{\bullet} $ or in $ \CO$. We then identify 
$ t[a_1, \dots, a_n] $ with some element in $ \CT[\FM_\bullet] $ using the following rule. Assume $ a_i=F \in \CO$:
\begin{enumerate}
    \item If the parent of the $i$-th leaf is a vertex with at least three children, then we set:
     $$ t[a_1, \dots,\underbrace{F}_{i^{th}},\dots a_n] :=  F \, t_i [ a_1, \dots, \widehat{F},\dots, a_n]  ,$$
    where $t_i$ is the tree obtained by erasing the $i$-th leaf, and
     $\hat{F} $ means that the term is omitted.
    \item  If the parent of the $i$-th leaf is a vertex with $2 $ children, then we set
    $$ t[a_1, \dots,\underbrace{F}_{i^{th}}, \dots, a_n] :=  0 .$$
\end{enumerate}
 \end{convention}
 
 \noindent Let us spell out this convention:
\vspace{1mm} 

\begin{tabular}{lccccc}
\adjustbox{valign = c}{
$1.$} & \adjustbox{valign = c}{\scalebox{0.5}{
\begin{forest}
for tree = {grow' = 90}, nice empty nodes, for tree={ inner sep=0 pt, s sep= 10 pt, fit=band, 
},
[
[, 
        [\scalebox{2}{$a_1$}, tier =1] 
        [\scalebox{2}{$a_2$}, tier =1]
    ]
    [\scalebox{2}{$F$}, tier =1, ]
    [\scalebox{2}{$a_3$}, tier =1,  ]
]
\path[fill=black] (.parent anchor) circle[radius=4pt]
(!1.child anchor) circle[radius=4pt];
\end{forest}}}
& \adjustbox{valign = c}{$=$} & 
\adjustbox{valign = c}{\scalebox{0.5}{
\begin{forest}
for tree = {grow' = 90}, nice empty nodes, for tree={ inner sep=0 pt, s sep= 10 pt, fit=band, 
},
[
[, 
        [\scalebox{2}{$Fa_1$}, tier =1] 
        [\scalebox{2}{$a_2$}, tier =1]
    ]
    [\scalebox{2}{$a_3$}, tier =1,  ]
]
\path[fill=black] (.parent anchor) circle[radius=4pt]
(!1.child anchor) circle[radius=4pt];
\end{forest}}}
& \adjustbox{valign = c}{$=$} & \adjustbox{valign = c}{$F$} \adjustbox{valign = c}{
\scalebox{0.5}{
\begin{forest}
for tree = {grow' = 90}, nice empty nodes, for tree={ inner sep=0 pt, s sep= 10 pt, fit=band, 
},
[
[, 
        [\scalebox{2}{$a_1$}, tier =1] 
        [\scalebox{2}{$a_2$}, tier =1]
    ]
    [\scalebox{2}{$a_3$}, tier =1,  ]
]
\path[fill=black] (.parent anchor) circle[radius=4pt]
(!1.child anchor) circle[radius=4pt];
\end{forest}}} \\ \\
\adjustbox{valign = c}{2.}& \adjustbox{valign = c}{\scalebox{0.5}{\begin{forest}
for tree = {grow' = 90}, nice empty nodes, for tree={ inner sep=0 pt, s sep= 10 pt, fit=band, 
},
[
[, 
        [\scalebox{2}{$a_1$}, tier =1] 
        [\scalebox{2}{$F$}, tier =1]
    ]
    [\scalebox{2}{$a_2$}, tier =1, ]
    [\scalebox{2}{$a_3$}, tier =1,  ]
]
\path[fill=black] (.parent anchor) circle[radius=4pt]
(!1.child anchor) circle[radius=4pt];
\end{forest}}}
& \adjustbox{valign = c}{$=0$.}

\end{tabular} 

 \vspace{.5cm}
\noindent
Let $S(\mathcal{T}ree[\FM_{\bullet}])$ be the graded commutative free algebra   generated  by $\mathcal{T}ree[\FM_{\bullet}] $ over $\CO $. 

\begin{remark}
\normalfont
Elements in $S(\CT[\FM_{\bullet}])$ can be thought of as $\CO$-linear combinations of forests of trees, i.e.\ elements in $\CT[\FM_{\bullet}]$, with some sign appearing when trees in the forest are permuted.
\end{remark}

\noindent
We need two more structures on $S(\mathcal{T}ree[\FM_{\bullet}])$: the root map and its inverse, the unroot map.

\vspace{0.1cm}

\noindent
Consider the $\mathbb K$-vector space $Tree$. For every $n \geq 2$, there is a natural $n$-linear operation $\r_n$ of degree $1$ which associates to $n$ rooted trees $(t_1, \dots, t_n)$ the tree obtained by  adding a root and linking it to the roots of $(t_1, \dots, t_n)$. For $n=3$, for instance
\vspace{-0.6cm}

\begin{center}
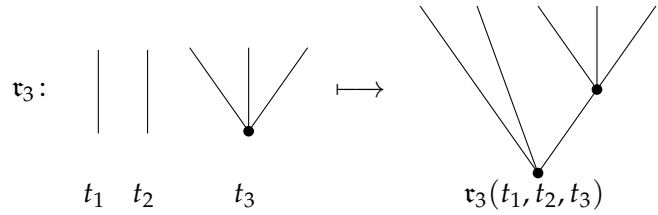
\begin{figure}[H]
\centering
\begin{tabular}{lccccc}
\adjustbox{valign = c}{$\r_3  \colon$}  &
\adjustbox{valign = c}{
\begin{forest}
for tree = {grow' = 90}, nice empty nodes, for tree={ inner sep=0 pt, s sep= 0 pt, fit=band, 
},
[
[, tier =1]
];
\end{forest}}
&\adjustbox{valign = c}{
\begin{forest}
for tree = {grow' = 90}, nice empty nodes, for tree={ inner sep=0 pt, s sep= 0 pt, fit=band, 
},
[
[, tier =1]
];
\end{forest}}
&\adjustbox{valign = c}{
\begin{forest}
for tree = {grow' = 90}, nice empty nodes, for tree={ inner sep=0 pt, s sep= 0 pt, fit=band, 
},
[
[, tier =1]
[, tier =1]
[, tier =1]
]
\path[fill=black] (.parent anchor) circle[radius=2pt];
\end{forest}}
& \adjustbox{valign = c}{$\longmapsto$}
&\adjustbox{valign = c}{
\begin{forest}
for tree = {grow' = 90}, nice empty nodes, for tree={ inner sep= 10 pt, s sep= 20 pt, 
},
[
[, tier =1]
[, tier =1]
[ 
[, tier =1]
[, tier =1]
[, tier =1]
]
]
\path[fill=black] (.parent anchor) circle[radius=2pt]
(!3.child anchor) circle[radius=2pt];
\end{forest}} \\
&$t_1$ & $t_2$ & $t_3$ & &$\r_3(t_1,t_2,t_3)$
\end{tabular}
\caption{Applying $\r_3$ to a forest of 3 trees}
\end{figure}
\end{center}
\vspace{-0.6cm}

\noindent
Altogether, the maps $(\r_n)_{n \geq 2}$ extend to a 
linear map
 $\r\colon  \otimes^{\geq 2} (Tree) 
    \xrightarrow{\sim} 
     Tree^{\hbox{$\,\geq \! 2$}}$
     which is an isomorphism
 of $\mathbb K$-vector spaces. 
It extends to decorated trees, then descends to the quotient defined in item (iv). We denote by the same letter and also call \emph{root map} that henceforth obtained map.

 \begin{proposition}
 \label{prop:2root}
The root map $\r$ is a degree $ +1$ isomorphism of $ \CO$-modules
 \begin{equation}\label{eq:rootIsIso} \xymatrix{ \r \colon   \SdegFM \ar[r]^{\hspace{.5cm}\sim} &  \TdegFM }  \end{equation}
For every $k\geq 2$, this isomorphism restricts to a degree $ +1$ isomorphism $\r_k $ from $S^{k} (\mathcal{T}ree[\FM_{\bullet}] ) $ 
 to the space of decorated trees whose root has $k$ children.
 \end{proposition}

 \begin{remark}
 \normalfont
 The isomorphism $\r$ does not extend to $ S^1(\mathcal{T}ree[\FM_{\bullet}])$ because we do not allow the root to have only one child.
 \end{remark}
 
 \begin{remark}
 \normalfont
 The constructions presented in this section  are easily checked to be compatible with Convention~\ref{conv:extensionO}.
 \end{remark}

\begin{remark}
\label{rmk:projective}
\normalfont
The constructions and results of the present section extend easily to several contexts of interest. We can choose $\CO $ to be the algebra  $C^\infty(M) $ of smooth functions on a manifold,
or replace it by the sheaf of holomorphic functions over a complex variety $M$.
In the smooth case, one has to assume that a there exists a $ \CO$-module resolution $(\FM_\bullet,d) $ given by a sequence of sections of some vector bundles, i.e. $ \FM_i = \Gamma(V_i)$ for all $i \geq 1$.
In the complex case, local resolutions $(\FM_\bullet,d)$ can be constructed  using properties of coherent sheaves.
In the smooth case, Proposition \ref{prop:stillfree} extends to state that $\mathcal{T}ree[\FM_{\bullet}]$ is made of sections of a graded vector bundle, which is of finite rank in every degree.
Finally, in both contexts, Proposition \ref{prop:2root} holds true verbatim, upon replacing free $\CO $-modules by the relevant sheaf of sections.
\end{remark}

\begin{definition}
\label{def:unroot}
We call   the inverse  $\r^{-1} $ of the root map $\r\colon  \SdegFM_{n}  \rightarrow \Tdeg[\FM_\bullet]_{n+1}$ the \emph{unroot map}.  
\end{definition}

\pagebreak
To provide an example:

\begin{center}
\begin{tabular}{cccccccc}
         \adjustbox{valign = c}{ $\r^{-1}\colon$} & \adjustbox{valign = c}{\begin{forest}
for tree = {grow' = 90}, nice empty nodes, for tree={ inner sep=0 pt, s sep= 10 pt, fit=band, 
},
[
[, ,edge label={node[left]{$t_1$}}
        [$a_1$, tier =1] 
        [$a_2$, tier =1]
        [$a_3$, tier =1]
    ]
    [$a_4$, tier =1,  edge label={node[midway, left]{$t_2$}}]
    [$a_5$, tier =1,  edge label={node[midway, right]{$t_3$}}]
]
\path[fill=black] (.parent anchor) circle[radius=2pt]
(!1.child anchor) circle[radius=2pt];
\end{forest}}
&\adjustbox{valign = c}{
$\longmapsto$}
&\adjustbox{valign = c}{
\begin{forest}
for tree = {grow' = 90},  nice empty nodes, for tree={ inner sep=0 pt, s sep= 0 pt, fit=band, 
},
[, 
        [$a_1$, tier =1] 
        [$a_2$, tier =1]
        [$a_3$, tier =1]
    ]
\path[fill=black] (.parent anchor) circle[radius=2pt];
\end{forest}}
&\adjustbox{valign = c}{
$\odot$}
&\adjustbox{valign = c}{
\begin{forest}
for tree = {grow' = 90},  nice empty nodes, for tree={ inner sep=0 pt, s sep= 0 pt, fit=band, 
},
[
[$a_4$, tier =1]
];
\end{forest}}
&\adjustbox{valign = c}{
$\odot$}
&\adjustbox{valign = c}{
\begin{forest}
for tree = {grow' = 90},  nice empty nodes, for tree={ inner sep=0 pt, s sep= 0 pt, fit=band, 
},
[
[$a_5$, tier =1]
];
\end{forest}} \hspace{5mm}.
\\
& $t$ & & $t_1$ & &$t_2$& & $t_3$
  \end{tabular}
  \end{center}
\vspace{2mm}

\subsection{Arborescent operations and the differential $\delta_{\psi}$}

\label{sec:differential}

\noindent 
\subsubsection{Some conventions and notations}

\begin{definition}
Let $(\FM_k)_{k \geq 1}$ be a sequence of free or projective $\CO$-modules.
We call a  linear map 
\begin{equation}\label{eq:obviousMap2}
 \begin{array}{cccc}
     \psi \colon & \Tdeg[ \FM_\bullet]  &\longrightarrow &\FM_{\bullet} .
 \end{array}
 \end{equation}
of degree $-1$ an \emph{arborescent operation}. 
\end{definition}
\vspace{0.1cm}
\noindent
Given any non-trivial ordered tree $t$ with $n$ leaves and $k$  vertices, we associate to it an $\CO $-module morphism
\begin{equation}\label{eq:arborescentFor_t} \psi_t \colon  \otimes^n \FM_\bullet \longrightarrow  \FM_{\bullet} 
\end{equation}
as follows
\begin{equation}
\label{eq:psitDef}
\psi_t\colon a_1\otimes \dots \otimes a_n
    \mapsto \left(\psi \circ {\mathrm{pr}}\right) \left( t[a_1, \dots, a_n]\right)  ,
    \end{equation}
for all $a_1, \dots, a_n \in \FM_\bullet$, where ${\mathrm{pr}}$ is the quotient map from $Tree[\FM_\bullet] $ to $\CT[\FM_{\bullet}]$. 

\vspace{0.1cm}
\noindent
By construction, the degree of $ \psi_t$ is $ k-1$. For every $a_1, \dots, a_n \in \CO \oplus \FM_\bullet$, we have
\begin{equation}\label{eq:signs_for_operations}\psi_t\left(a_1 , \dots, a_n\right) \, = \, \epsilon(t_{A,\sigma})\, \psi_{t_{A, \sigma}} \left(a_{\sigma(1)}, \dots, a_{\sigma(n)}\right) \end{equation}
where $\sigma $ is any permutation of the children of a vertex $A$ as on the right-hand side of  Equation~\eqref{eq:quotientEquiv}.
\begin{definition}
Given an arborescent operation $\psi$,
we call the map defined in \eqref{eq:psitDef}, that maps $t \in Tree^n$ to $\psi_t \in \mathrm{Hom}(\otimes^n \FM_\bullet , \FM_\bullet)$,   the \emph{arborescent operation associated to the ordered tree $t$}.  
\end{definition}

\noindent Note that  $\psi_t$ is attached to an ordered tree $t$, while $ \psi$ is defined on symmetrized \emph{decorated} trees. We remark, furthermore, that $\psi_t$ can be viewed upon as an $n$-ary product on $ \CO \oplus \FM_\bullet$. 

\vspace{1mm}\noindent
For future use, we will need to see arborescent operations as a map $\vert [\psi] $ taking values in trivial trees:
\begin{equation}
\label{eq:vertpsi} 
\vert[\psi] \colon t[a_1, \dots, a_n] \mapsto  | \left[ \psi_t(a_1, \dots, a_n)\right] 
\end{equation}
for any tree $t$ with $n$ leaves and any  $a_1 \otimes \dots \otimes a_n \in \FM_\bullet$.

\subsubsection{Description of the differential $\delta_\psi $}

\label{sec:constructingdeltapsi}

We now return to the original context, where $(\FM_\bullet,d) $ is an $\CO$-module resolution of $\CO/\CI $ as in \eqref{eq:MdO}.
We choose an arbitrary arborescent operation $\psi$ and construct a degree $-1 $ derivation $\delta_{\psi}$ of $S({\CT}[\FM_\bullet])$  by means of a recursion. 
Being a derivation, it suffices to give its restriction to the $\CO$-module ${\mathcal {T}}ree[\FM_\bullet]$. 
\begin{equation}\label{eq:ourrestriction} \delta_\psi \colon {\mathcal T}ree[\FM_\bullet] \longrightarrow S( {\mathcal T}ree[\FM_\bullet]).\end{equation}
 For degree reason, $\delta_{\psi}$ has to be identically zero on the subspace of elements of degree $0$.
Notice that it implies that $\delta_{\psi}$ is an $\CO$-linear derivation.
Elements of degree $1$ in $\CT[\FM_{\bullet}]$
are elements in $Tree_0^1\otimes_{\mathbb K} \FM_{1}$, i.e. elements of the form $ | \otimes_{\K} a= |\,[a] $ for some $a \in \FM_1 $. Here we require that $\delta_\psi\left(|\, [a] \right) = d(a) \in \CO$. For defining the action of $\delta_\psi$ on higher degrees, we decompose  ${\mathcal T}ree[\FM_\bullet]$ and $S( {\mathcal T}ree[\FM_\bullet])$: 
  $$  \delta_\psi \colon \begin{pmatrix}  \FM_\bullet \oplus \CO\\ \Tdeg[\FM_\bullet] \end{pmatrix} \longrightarrow \begin{pmatrix} \FM_\bullet \oplus \CO\\ \Tdeg[\FM_\bullet] \\ \SdegFM \end{pmatrix}     .$$
 Then the map \eqref{eq:ourrestriction} becomes a $ 3 \times 2$ matrix that we impose to be of the form 
 \begin{equation}
 \label{eq:23matrix}
 \delta_\psi = \begin{pmatrix}   d & - \psi  \\ 0& * \\ 0 & \r^{-1}  \end{pmatrix} .  \end{equation}
Here, $ |[\psi]$ is defined in Equation \eqref{eq:vertpsi}, $\r^{-1} $ is the unroot map of Definition \ref{def:unroot},
and $* $ is defined recursively as follows.
 Assume $\delta_\psi$ is defined on elements of degree less or equal to $n$.  Let $\p$ be the the projection to the last component of the direct sum decomposition $ S(\CT[\FM_\bullet]) \,=\,|[\FM_{\bullet}] \,\oplus\,  \Tdeg[\FM_\bullet]  \,\oplus\, \SdegFM$. 
 Then, we put
 \begin{equation} \label{eq:star} * := -\r \circ {\mathrm{p}}^{\geq 2} \circ \delta_{\psi} \circ \r^{-1} \, ,\end{equation}
i.e.\ 
 $$\resizebox{0.85\hsize}{!}{
 \xymatrix{
 \CT[\FM_\bullet]_{n+1}
 \ar@/^3.0pc/[rrrr]^{\hbox{\large{$*$}}}
 \ar[r]^{\r^{-1}}  &\ar[r]^{\delta_\psi} \SdegFM_{n} & \ar[r]^{\mathrm{p}^{\geq 2}}S(\CT[\FM_\bullet])_{n-1} & \ar[r]^{-\r}
 \SdegFM_{n-1} &   
 \CT[\FM_\bullet]_{n}
 } \, .}
 $$

\vspace{0.1cm}
\noindent This now defines $\delta_\psi$ entirely---which, without further restrictions on the arborescent operation $\psi$, does not square to zero in general.

\begin{example}
\normalfont
 To illustrate the description of $ *$, we consider the tree with two leaves decorated by elements $a_1$ and $a_2$ with $|a_1| =1$ and $|a_2|\geq 2$:
 
 \begin{tabular}{l}
 \adjustbox{valign = c}{
 $\r^{-1}\colon$}
 \adjustbox{valign = c}{
 \begin{forest}
for tree = {grow' = 90},  nice empty nodes, for tree={ inner sep=0 pt, s sep= 0 pt, fit=band, 
},
[, 
        [$a_1$, tier =1] 
        [$a_2$, tier =1]
    ]
\path[fill=black] (.parent anchor) circle[radius=2pt];
\end{forest}} \quad \adjustbox{valign = c}{$\mapsto$}\quad 
\adjustbox{valign = c}{
 \begin{forest}
for tree = {grow' = 90},  nice empty nodes, for tree={ inner sep=0 pt, s sep= 0 pt, fit=band, 
},
[, 
        [$a_1$, tier =1] 
    ]
\end{forest}}
\quad \adjustbox{valign = c}{
 \begin{forest}
for tree = {grow' = 90},  nice empty nodes, for tree={ inner sep=0 pt, s sep= 0 pt, fit=band, 
},
[, 
        [$a_2$, tier =1] 
    ]
\end{forest}} \adjustbox{valign = c}{,}
\qquad \adjustbox{valign = c}{
$\delta_{\psi}\colon$}
\adjustbox{valign = c}{
 \begin{forest}
for tree = {grow' = 90},  nice empty nodes, for tree={ inner sep=0 pt, s sep= 0 pt, fit=band, 
},
[, 
        [$a_1$, tier =1] 
    ]
\end{forest}}
\quad
\adjustbox{valign = c}{
 \begin{forest}
for tree = {grow' = 90},  nice empty nodes, for tree={ inner sep=0 pt, s sep= 0 pt, fit=band, 
},
[, 
        [$a_2$, tier =1] 
    ]
\end{forest}}
 \quad \adjustbox{valign = c}{$\mapsto$} \quad
 \adjustbox{valign = c}{
 \begin{forest}
for tree = {grow' = 90},  nice empty nodes, for tree={ inner sep=0 pt, s sep= 0 pt, fit=band, 
},
[, 
        [$(da_1)a_2$, tier =1] 
    ]
\end{forest}}
\adjustbox{valign = c}{
$ - $} \quad
\adjustbox{valign = c}{
 \begin{forest}
for tree = {grow' = 90},  nice empty nodes, for tree={ inner sep=0 pt, s sep= 0 pt, fit=band, 
},
[, 
        [$a_1$, tier =1] 
    ]
\end{forest}}
\quad
\adjustbox{valign = c}{
 \begin{forest}
for tree = {grow' = 90},  nice empty nodes, for tree={ inner sep=0 pt, s sep= 0 pt, fit=band, 
},
[, 
        [$da_2$, tier =1] 
    ]
\end{forest}}\adjustbox{valign = c}{,}
\\ \\
\adjustbox{valign = c}{
$\mathrm{p}^{\geq 2}\colon $} \adjustbox{valign = c}{ \begin{forest}
for tree = {grow' = 90},  nice empty nodes, for tree={ inner sep=0 pt, s sep= 0 pt, fit=band, 
},
[, 
        [$(da_1)a_2$, tier =1] 
    ]
\end{forest}}
\adjustbox{valign = c}{
 $-$} \quad \adjustbox{valign = c}{\begin{forest}
for tree = {grow' = 90},  nice empty nodes, for tree={ inner sep=0 pt, s sep= 0 pt, fit=band, 
},
[, 
        [$a_1$, tier =1] 
    ]
\end{forest}}
\quad
\adjustbox{valign = c}{
 \begin{forest}
for tree = {grow' = 90},  nice empty nodes, for tree={ inner sep=0 pt, s sep= 0 pt, fit=band, 
},
[, 
        [$da_2$, tier =1] 
    ]
\end{forest}}
\quad \adjustbox{valign = c}{$\mapsto$} \quad  \adjustbox{valign = c}{$-$} \quad \adjustbox{valign = c}{\begin{forest}
for tree = {grow' = 90},  nice empty nodes, for tree={ inner sep=0 pt, s sep= 0 pt, fit=band, 
},
[, 
        [$a_1$, tier =1] 
    ]
\end{forest}}
\quad \adjustbox{valign = c}{
 \begin{forest}
for tree = {grow' = 90},  nice empty nodes, for tree={ inner sep=0 pt, s sep= 0 pt, fit=band, 
},
[, 
        [$da_2$, tier =1] 
    ]
\end{forest}}\adjustbox{valign = c}{,}
 \\ \\\adjustbox{valign = c}{ $-\r\colon$  $-$}\quad \adjustbox{valign = c}{\begin{forest}
for tree = {grow' = 90},  nice empty nodes, for tree={ inner sep=0 pt, s sep= 0 pt, fit=band, 
},
[, 
        [$a_1$, tier =1] 
    ]
\end{forest}}
\quad\adjustbox{valign = c}{
 \begin{forest}
for tree = {grow' = 90},  nice empty nodes, for tree={ inner sep=0 pt, s sep= 0 pt, fit=band, 
},
[, 
        [$da_2$, tier =1] 
    ]
\end{forest}} \quad\adjustbox{valign = c}{$\mapsto$}\quad
\adjustbox{valign = c}{\begin{forest}
for tree = {grow' = 90},  nice empty nodes, for tree={ inner sep=0 pt, s sep= 0 pt, fit=band, 
},
[, 
        [$a_1$, tier =1] 
        [$da_2$, tier =1]
    ]
\path[fill=black] (.parent anchor) circle[radius=2pt];
\end{forest}}\adjustbox{valign = c}{.}
 \end{tabular}

\end{example}

\vspace{0.2cm}
\noindent
We will now derive an  explicit (i.e.\ non-recursive) description of the derivation $\delta_\psi$  in terms of the arborescent operation $ \psi$ and several elementary operations on trees:
For this purpose, we first equip the $\CO$-module of ordered decorated trees  $Tree[\FM_\bullet] $ of all decorated rooted trees with a natural differential  $ \partial $ as follows.

\begin{definition}
\label{def:WA}
For $t[a_1, \dots, a_n]$ a decorated ordered tree, we attach a  \emph{weight}  $W_A\in \mathbb{Z}$ to any inner vertex $A \in t$.
It is the degree of the tree obtained by summing the following two contributions:
 \begin{enumerate}
     \item  the number of edges in the path $ \gamma$ that goes from the root to $A$
 \item  the sum of degrees of all left subtrees of a vertex on the path $\gamma$ from the root to  $A$ whose root is not on $ \gamma$.
 By convention, we decide that the root of $t$ is included but $A$ is excluded.
\end{enumerate}
The weight can be defined in the same way for a leaf. We denote  the weight of the $i$-th leaf by $ W_i$. For the root $R$, $ W_R$ is $0$. 
\end{definition}

\begin{example}
\normalfont
 Let us consider the following tree with vertex $A$: 
 \begin{equation}
 \label{eq:AB}
\adjustbox{valign=c}{\scalebox{0.5}{
\begin{forest}
for tree = {grow' = 90}, nice empty nodes, for tree={ inner sep= 0 pt, s sep= 0 pt, 
},
[, {label=[mystyle]{\scalebox{2}{$R\,\,\, $}}} 
[\scalebox{2}{$a_1$}, tier =1]
[, ,{label=[mystyle3]{\scalebox{2}{$\,\,\, B$}}}, 
[, {label=[mystyle]{\scalebox{2}{$A\,\,\, $}}},
[\scalebox{2}{$a_2$}, tier = 1]
[
[\scalebox{2}{$a_3$}, tier = 1]
[\scalebox{2}{$a_4$}, tier = 1]
]
]
[\scalebox{2}{$a_5$}, tier =1]
]
]
\path[fill=black] (.parent anchor) circle[radius=4pt]
(!2.child anchor) circle[radius=4pt]
(!21.child anchor) circle[radius=4pt]
(!212.child anchor) circle[radius=4pt];
\end{forest}}}
\end{equation}
\noindent
The number of edges in the path from the root to $A$ is $2$.
The left subtree of the root, namely $ |[a_1]$, contributes to the degree by $|a_1|$. 
Since the root $A$ of the left subtree of $B$ is \emph{on} the path, it does not count. Overall, in this case, $ W_A=2+|a_1|$.   
\end{example}
    
    \begin{definition}
\label{def:deltaA}
    Let $t$ be an ordered tree, and $A$ one of its inner vertices. We denote the tree obtained by merging the vertex $A$ with its parent vertex by
 $\partial_A t$.
 \vspace{1mm}
 \noindent
 Since $\partial_A $ does not modify the cardinality of leaves and is compatible with the symmetry relations, $\partial_A $ can be defined on a decorated rooted tree as follows:
 $\partial_A \colon t[a_1, \dots,a_n]  \mapsto (\partial_A t)[a_1, \dots, a_n] $. 
 \end{definition}


\vspace{1mm}
\noindent We then consider the degree $-1$ map 

   \begin{equation} \label{eq:differential_on_trees} \begin{array}{rrcl} \partial \colon & Tree [\FM_\bullet] & \longrightarrow & Tree [\FM_{\bullet}]\\
   & t[a_1, \dots,a_n] 
 & \mapsto & \sum_{A \in \mathrm{InnVer}(\,t \,)}(-1)^{W_A}\partial_A t[a_1, \dots,a_n] \, ,
 \end{array} \end{equation}
 where we recall that $\mathrm{InnVer}(t)$ is the set of all inner vertices  of the tree $t$.

\begin{example} \normalfont
We illustrate $\partial$ on the following ordered decorated tree with the decorations $a, b, c, d, e$ of  definite degrees $|a|, |b|, |c|, |d|, |e|$, respectively.   

  $$\adjustbox{valign=c}{$\partial\colon$}  \adjustbox{valign=c}{\scalebox{0.5}
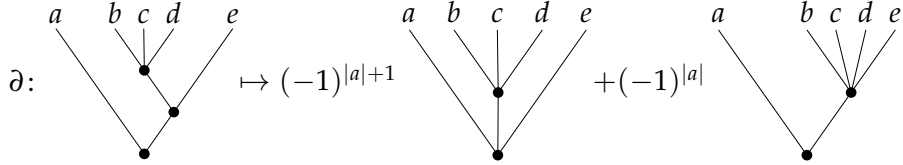
{
\begin{forest}
for tree = {grow' = 90}, nice empty nodes, for tree={ inner sep= 0 pt, s sep= 0 pt, 
},
[
[\scalebox{2}{$a$}, tier =1]
[ 
[
[\scalebox{2}{$b$}, tier = 1]
[\scalebox{2}{$c$}, tier = 1]
[\scalebox{2}{$d$}, tier = 1]
]
[\scalebox{2}{$e$}, tier =1]
]
]
\path[fill=black] (.parent anchor) circle[radius=4pt]
(!2.child anchor) circle[radius=4pt]
(!21.child anchor) circle[radius=4pt];
\end{forest}}}
\adjustbox{valign=c}{$\mapsto (-1)^{|a|+1}$}\adjustbox{valign=c}{\scalebox{0.75}{
\begin{forest}
for tree = {grow' = 90}, nice empty nodes, for tree={ inner sep= 0 pt, s sep= 0 pt, 
},
[
[\scalebox{1.33}{$a$}, tier =1] 
[
[\scalebox{1.33}{$b$}, tier = 1]
[\scalebox{1.33}{$c$}, tier = 1]
[\scalebox{1.33}{$d$}, tier = 1]
]
[\scalebox{1.33}{$e$}, tier =1]
]
\path[fill=black] (.parent anchor) circle[radius=2.67pt]
(!2.child anchor) circle[radius=2.67pt];
\end{forest}}}
\adjustbox{valign=c}{$+ (-1)^{|a|}$}\adjustbox{valign=c}{\scalebox{0.75}{
\begin{forest}
for tree = {grow' = 90}, nice empty nodes, for tree={ inner sep= 0 pt, s sep= 0 pt, 
},
[
[\scalebox{1.33}{$a$}, tier =1]
[ 
[\scalebox{1.33}{$b$}, tier = 1]
[\scalebox{1.33}{$c$}, tier = 1]
[\scalebox{1.33}{$d$}, tier = 1]
[\scalebox{1.33}{$e$}, tier =1]
]
]
\path[fill=black] (.parent anchor) circle[radius=2.67pt]
(!2.child anchor) circle[radius=2.67pt];
\end{forest}}}
$$
\end{example}

 \begin{proposition}
 \label{prop:partial} The map $\partial $ squares to zero. It descends to the quotient $\CT[\FM_{\bullet}]$.
 \end{proposition}
 \begin{proof}
This is proven by recursion on the number of inner vertices and is left to the reader.     
 \end{proof}
 \vspace{0.1cm}
\noindent
 By abuse of notations, we denote again this quotient map by $\partial$.
 \noindent
We need two more operations on trees.

\begin{definition}
\label{def:upanddown}
For any vertex $A$ of a tree $t$, we denote by $t_{\downarrow A} $
the tree obtained by replacing that vertex by a leaf, and $t_{\uparrow A}$  the subtree of $t$ with  root $A$. 
\end{definition}

\begin{example}
\normalfont
Let us give an example explaining Definition \ref{def:upanddown}.
Consider the tree
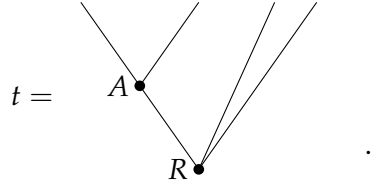
\normalfont
    $$
   \begin{tabular}{ccc} \adjustbox{valign = c}{ $t =$} &
    \adjustbox{valign = c}{\begin{forest}
for tree = {grow' = 90}, nice empty nodes, for tree={ inner sep=0 pt, s sep= 10 pt, fit=band, 
},
[, label={left:$R$}
[, ,edge label={node[left]{$A$}}
        [, tier =1] 
        [, tier =1]
    ]
    [, tier =1, ]
    [, tier =1,]
]
\path[fill=black] (.parent anchor) circle[radius=2pt]
(!1.child anchor) circle[radius=2pt];
 \end{forest}}
 &
 \adjustbox{valign = c}{ $\begin{array}{c} \\ \\ \\ . \end{array}$}\end{tabular}
$$
For the vertex being the root $R$, the operations are rather trivial: $t_{\uparrow R}=t $ and $t_{\downarrow R}$ is the trivial tree. For the vertex $A$, 
we get
$$
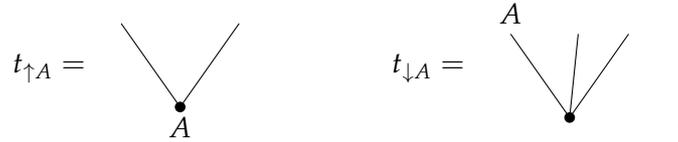
\begin{array}{rlcrlr}
\adjustbox{valign =c}{$t_{\uparrow A} = $}
&
\adjustbox{valign =c}{
\begin{forest}
for tree = {grow' = 90}, nice empty nodes, for tree={ inner sep=0 pt, s sep= 10 pt, fit=band, 
},
[, label={below:$ A$}
        [$ $, tier =1] 
        [$ $, tier =1]
]
\path[fill=black] (.parent anchor) circle[radius=2pt];
\end{forest}}
& \adjustbox{valign =c}{  \hspace{1cm} }&
\adjustbox{valign =c}{  $t_{\downarrow A} = $} &
\adjustbox{valign =c}{
\begin{forest}
for tree = {grow' = 90}, nice empty nodes, for tree={ inner sep=0 pt, s sep= 10 pt, fit=band, 
},
[, label={below:$ $}
        [$A $
        , tier =1] 
        [$ $, tier =1]
        [$ $, tier =1]
]
\path[fill=black] (.parent anchor) circle[radius=2pt];
\end{forest}}
&
\adjustbox{valign =c}{ $\begin{array}{c} \\ \\ \\ . \end{array} $}
\end{array}
    $$

\end{example}

\noindent
Consider now an \emph{ordered} tree $t$ with $n \geq 2 $ leaves and define, for all $a_1, \dots, a_n \in \FM_\bullet $: 
\begin{equation}
\label{eq:deltaexpression}
\begin{array}{rcl}
     \widehat{\delta_\psi} (t[a_1, \dots, a_n]) &=&  \r^{-1} \left( t[a_1, \dots, a_n] \right) +  \partial\left(t  [a_1, \dots, a_n]\right) + 
    \, \sum_{i=1}^n (-1)^{W_i} t[a_1, \dots, da_i, \dots, a_n] \\[5pt]
    & &
    - | \left[\psi_t(a_1, \dots, a_n)\right]
- \sum_{A \in \mathrm{InnVer}(t)}(-1)^{W_A}\left( 
t_{\downarrow A}[a_1, \dots,\psi_{t_{\uparrow A}}(a_A) , \dots a_n]\right) \\
 \end{array}
 \end{equation}
 where  in the third term, Convention \ref{conv:extensionO} must be applied when $ a_i \in \FM_1$ and where $ a_A$ denotes the decorations of the leaves that descend from the vertex $A$. Let us shorten Equation \eqref{eq:deltaexpression} by introducing a sum that runs over all vertices,  root and leaves included. For this purpose  we introduce the following 
\begin{convention}
If $A$ is an external vertex, i.e.\ either the root or a leaf, we put $\partial_A = 0$. In addition, if  $A$ is the $i$-th leaf, we set $t_{\uparrow A}[a_A] := | \otimes_{\K} a_i$ and  $ \psi_{|}(a_i) := - d(a_i) $.
\end{convention}
 \noindent Then 
Equation \eqref{eq:deltaexpression} reads
\begin{equation}
\label{eq:NewConventions}
\begin{array}{rcl}\widehat{\delta_\psi} (t[a_1, \dots, a_n])& =&  \r^{-1} \left( t[a_1, \dots, a_n] \right) 
\\&&+ \sum_{\scalebox{0.5}{{$\begin{array}{r}A \in \mathrm{Ver}(t)\end{array}$}}}(-1)^{W_A}\left( -
t_{\downarrow A}[a_1, \dots,\psi_{t_{\uparrow A}}(a_A) , \dots a_n]+(\partial_A t)[a_1, \dots,a_n] \right).\end{array}\end{equation}

\begin{proposition}
\label{prop:details}
The map  $\widehat{\delta_\psi}$ descends to $\CT[\FM_{\bullet}]$ where it coincides with $\delta_\psi$.
\end{proposition}
    
    
 
\begin{proof} 
For every tree $t$ with $n$ leaves, consider the subgroup of the group of permutations $\Sigma_n $ 
obtained by permuting children of a given vertex. 
Each one of the operations defining $\widehat{\delta_\psi}$ on the r.h.s. of \eqref{eq:deltaexpression} is equivariant under this subgroup, which proves that $\widehat{\delta_\psi}$ descends to the quotient. 

\vspace{1mm}
\noindent
Recall from \eqref{eq:23matrix}
 that $ \delta_\psi$ is the sum of three terms: $\r^{-1}, -\psi $ and $ *$. The first line in the equation is the unroot map $ \r^{-1} $, which is one of the components of $\delta_\psi $. 
For $A$ the root $R$ of the tree $t$, we have $\partial_R=0$, $W_R=0$, $a_R= a_1, \dots, a_n$ and:
$$ -t_{\downarrow R}[a_1, \dots,\psi_{t_{\uparrow R}}(a_R) , \dots a_n]  =-\vert \, [\psi_{t_{\uparrow R}} (a_1, \dots, a_n)] .$$
 One recognizes the term $ -\psi$.
We have therefore to check that the residual terms in the sum coincide with $* $.
This is done by a recursion on the number of inner vertices $k$.  For $k = 0$, these residual terms are read as:
$$ 
\adjustbox{valign =c}{$-$}
\adjustbox{valign =c}{
\scalebox{0.5}{
\begin{forest}
for tree = {grow' = 90}, nice empty nodes,
[
 [\scalebox{2}{$da_1$}, tier =1]
 [\scalebox{2}{$a_2$}, tier =1]
 [\scalebox{2}{$\dots$},edge= dashed,  tier =1]
 [\scalebox{2}{$a_n$}, tier = 1]
]
\path[fill=black] (.parent anchor) circle[radius=4pt];
\end{forest}}}
\adjustbox{valign =c}{$ + (-1)^{|a_1|+1}$}
\adjustbox{valign =c}{
\scalebox{0.5}{
\begin{forest}
for tree = {grow' = 90}, nice empty nodes,
[
 [\scalebox{2}{$a_1$}, tier =1]
 [\scalebox{2}{$da_2$}, tier =1]
 [\scalebox{2}{$\dots$},edge= dashed,  tier =1]
 [\scalebox{2}{$a_n$}, tier = 1]
]
\path[fill=black] (.parent anchor) circle[radius=4pt];
\end{forest}}}
\adjustbox{valign =c}{$
+
\dots +
(-1)^{|a_1|+\dots + |a_{n-1}|+1}$}
\adjustbox{valign =c}{
\scalebox{0.5}{
\begin{forest}
for tree = {grow' = 90}, nice empty nodes,
[
 [\scalebox{2}{$a_1$}, tier =1]
 [\scalebox{2}{$a_2$}, tier =1]
 [\scalebox{2}{$\dots$},edge= dashed,  tier =1]
 [\scalebox{2}{$da_n$}, tier = 1]
]
\path[fill=black] (.parent anchor) circle[radius=4pt];
\end{forest}}}
$$
In view of the Conventions \ref{conv:extensionO}, it is straightforward to check that the summands above add up to $*$. Assume that the results holds true fo tress with $k$ inner vertices. For  $k+1$ inner vertices the tree $t[a_1, \dots, a_n]$ we view as a rooting of $j$ trees $t_{A_1}[a_{A_1}], \dots, t_{A_j}[a_{A_j}]$:
$$
\adjustbox{valign =c}{$
t[a_1, \dots, a_n] =$}
\adjustbox{valign =c}{ 
\scalebox{0.5}{
\begin{forest}
for tree = {grow' = 90}, nice empty nodes,
[
 [\scalebox{2}{$t_{A_1}{[a_{A_1}]}$}, tier =1]
 [\scalebox{2}{$t_{A_2}{[a_{A_2}]}$}, tier =1]
 [\scalebox{2}{$\dots$},edge= dashed,  tier =1]
 [\scalebox{2}{$t_{A_j}{[a_{A_j}]}$}, tier = 1]
]
\path[fill=black] (.parent anchor) circle[radius=4pt];
\end{forest}}}
$$
Then the remaining sum can be rewritten using the recursion assumption as: 
$$
\adjustbox{valign =c}{ $\begin{array}{c} \\
\sum_{i=1}^{j}(-1)^{1+|t_{A_1}{[a_{A_1}]}|+\dots+ |t_{A_{i-1}}{[a_{A_{i-1}}]}|}\end{array}$}
\adjustbox{valign =c}{ \scalebox{0.5}{
\begin{forest}
for tree = {grow' = 90}, nice empty nodes,
[
 [\scalebox{2}{$t_{A_1}{[a_{A_1}]}$}, tier =1]
 [\scalebox{2}{$\dots$},edge= dashed,  tier =1]
 [\scalebox{2}{$\delta_{\psi}t_{A_i}{[a_{A_i}]}$},  tier =1]
 [\scalebox{2}{$\dots$},edge= dashed,  tier =1]
 [\scalebox{2}{$t_{A_j}{[a_{A_j}]}$}, tier = 1]
]
\path[fill=black] (.parent anchor) circle[radius=4pt];
\end{forest}}}
$$
which are identifed then with the expression $* = - \r \circ \p \circ \delta_{\psi} \circ \r^{-1}(t[a_1, \dots, a_n])$. The statement follows.
\end{proof}

\begin{example}
\normalfont
We illustrate the content of Proposition \ref{prop:details} by means of three examples:

\begin{enumerate} 
\item
For the tree with one root and two leaves decorated by elements $a_1, a_2 \in \FM_\bullet $ of degrees at least $2$, we have:
\begin{equation*}
            \begin{gathered}
\adjustbox{valign =c}{$\begin{array}{c}  \\ \delta_{\psi}:\end{array}$}
\adjustbox{valign =c}{
\begin{forest}
for tree = {grow' = 90}, nice empty nodes,
[
 [$a_1$, tier =1]
 [$a_2$, tier =1] 
]
\path[fill=black] (.parent anchor) circle[radius=2pt];
\end{forest}}
\adjustbox{valign =c}{$
\begin{array}{c}  \\ \longmapsto \end{array}$}
\adjustbox{valign =c}{
\begin{forest}
for tree = {grow' = 90}, nice empty nodes,
[
 [$a_1$, tier =1] 
]
\end{forest}}
\adjustbox{valign =c}{
\begin{forest}
for tree = {grow' = 90}, nice empty nodes,
[
[$a_2$, tier =1]
]
\end{forest}}
\adjustbox{valign =c}{$
\begin{array}{c}  \\- \end{array}$}
\adjustbox{valign =c}{
\begin{forest}
for tree = {grow' = 90}, nice empty nodes,
[
 [$da_1$, tier =1]
 [$a_2$, tier =1] 
]
\path[fill=black] (.parent anchor) circle[radius=2pt];
\end{forest}}
\adjustbox{valign =c}{$\begin{array}{c}  \\ +
(-1)^{|a_1|+1}\end{array}$}
\adjustbox{valign =c}{
\begin{forest}
for tree = {grow' = 90}, nice empty nodes,
[
 [$a_1$, tier =1]
 [$da_2$, tier =1] 
]
\path[fill=black] (.parent anchor) circle[radius=2pt];
\end{forest}}
\adjustbox{valign =c}{ $\begin{array}{c} \\ -\end{array}$ }
\adjustbox{valign =c}{ 
\begin{forest}
for tree = {grow' = 90}, nice empty nodes,
[
[$\psi_\vee(a_1{,\,}a_2)$, tier =1]
]
\end{forest}}
\end{gathered}
\adjustbox{valign =c}{$ \begin{array}{c} \\ \\ \\ . \end{array}$}
     \end{equation*}
     \item
For the tree with one root and two leaves decorated by elements $a_1 \in \FM_1$ of degree $1$ and $ a_2 \in \FM_\bullet $ of degree at least $2$, we have:
\begin{equation*}
            \begin{gathered}
\adjustbox{valign =c}{
$\begin{array}{c}  \\             \delta_{\psi}:\end{array}$}
\adjustbox{valign =c}{
\begin{forest}
for tree = {grow' = 90}, nice empty nodes,
[
 [$a_1$, tier =1]
 [$a_2$, tier =1] 
]
\path[fill=black] (.parent anchor) circle[radius=2pt];
\end{forest}}
\adjustbox{valign =c}{$
\begin{array}{c} \\ \longmapsto \end{array}$}
\adjustbox{valign =c}{
\begin{forest}
for tree = {grow' = 90}, nice empty nodes,
[
 [$a_1$, tier =1] 
]
\end{forest}}
\adjustbox{valign =c}{
\begin{forest}
for tree = {grow' = 90}, nice empty nodes,
[
[$a_2$, tier =1]
]
\end{forest}}
\adjustbox{valign =c}{$
\begin{array}{c} \\ +\end{array}$}
\adjustbox{valign =c}{
\begin{forest}
for tree = {grow' = 90}, nice empty nodes,
[
 [$a_1$, tier =1]
 [$da_2$, tier =1] 
]
\path[fill=black] (.parent anchor) circle[radius=2pt];
\end{forest}}
\adjustbox{valign =c}{$\begin{array}{c} \\ 
-\end{array} $}
\adjustbox{valign =c}{
\begin{forest}
for tree = {grow' = 90}, nice empty nodes,
[
[$\psi_\vee(a_1{,\,}a_2)$, tier =1]
]
\end{forest}}\adjustbox{valign =c}{$  \begin{array}{c} \\ \\ \\
. \end{array}$}
\end{gathered}
     \end{equation*}
     \item
For the decorated tree $t[a_1,a_2,a_3]$ with $t$ as below, with the elements $a_1, a_2, a_3$ of degrees at least $2$, formula \eqref{eq:deltaexpression} yields:
\begin{equation*} \begin{array}{rcl}

            \begin{gathered}
\adjustbox{valign =c}{$
         \begin{array}{c} \\   \delta_{\psi}: \end{array}$}
\adjustbox{valign =c}{
\scalebox{0.5}{\begin{forest}
for tree = {grow' = 90}, nice empty nodes,
[
[, ,edge label={node[left]{$ $}}
 [\scalebox{2}{$a_1$}, tier =1]
 [\scalebox{2}{$a_2$}, tier =1] 
 ]
 [\scalebox{2}{$a_3$}, tier = 1]
]
\path[fill=black] (.parent anchor) circle[radius=4pt]
                (!1.child anchor) circle[radius=4pt];
\end{forest}}} \end{gathered}&
\adjustbox{valign =c}{$ \begin{array}{c} \\
\longmapsto \end{array}$} & \begin{gathered} 
\adjustbox{valign =c}{
\begin{forest}
for tree = {grow' = 90}, nice empty nodes,
[,
 [$a_1$, tier =1] 
 [$a_2$, tier =1] 
]
\path[fill=black] (.parent anchor) circle[radius=2pt];
\end{forest}}
\adjustbox{valign =c}{
\begin{forest}
for tree = {grow' = 90}, nice empty nodes,
[
[$a_3$, tier =1]
]
\end{forest}}
\adjustbox{valign =c}{$\begin{array}{c} \\
- \end{array}$}
\adjustbox{valign =c}{
\begin{forest}
for tree = {grow' = 90}, nice empty nodes,
[
 [$a_1$, tier =1] 
 [$a_2$, tier =1] 
 [$a_3$, tier = 1]
]
\path[fill=black] (.parent anchor) circle[radius=2pt];
\end{forest}} \end{gathered}
\\[30pt] && \begin{gathered} \adjustbox{valign =c}{$\begin{array}{c} \\
+\end{array}$} 
\adjustbox{valign =c}{  \scalebox{0.5}{\begin{forest}
for tree = {grow' = 90}, nice empty nodes,
[
[
 [\scalebox{2}{$da_1$}, tier =1]
 [\scalebox{2}{$a_2$}, tier =1] 
 ]
 [\scalebox{2}{$a_3$}, tier = 1]
]
\path[fill=black] (.parent anchor) circle[radius=4pt]
                (!1.child anchor) circle[radius=4pt];
\end{forest}}}
\adjustbox{valign =c}{$ \begin{array}{c} \\ + (-1)^{|a_1|}\end{array}$}
\adjustbox{valign =c}{\scalebox{0.5}{\begin{forest}
for tree = {grow' = 90}, nice empty nodes,
[
[
 [\scalebox{2}{$a_1$}, tier =1]
 [\scalebox{2}{$da_2$}, tier =1] 
 ]
 [\scalebox{2}{$a_3$}, tier = 1]
]
\path[fill=black] (.parent anchor) circle[radius=4pt]
                (!1.child anchor) circle[radius=4pt];
\end{forest}}}
\adjustbox{valign =c}{$ + (-1)^{|a_1| + |a_2|}$}
\adjustbox{valign =c}{\scalebox{0.5}{\begin{forest}
for tree = {grow' = 90}, nice empty nodes,
[
[
 [\scalebox{2}{$a_1$}, tier =1]
 [\scalebox{2}{$a_2$}, tier =1] 
 ]
 [\scalebox{2}{$da_3$}, tier = 1]
]
\path[fill=black] (.parent anchor) circle[radius=4pt]
                (!1.child anchor) circle[radius=4pt];
\end{forest}}} \end{gathered}
\\[20pt] & & 
\begin{gathered}
\adjustbox{valign =c}{$\begin{array}{c} \\
-\end{array}$}
\adjustbox{valign =c}{
 \begin{forest}
for tree = {grow' = 90}, nice empty nodes,
[
[$\psi_{t} (a_1{,\,} a_2{,\,} a_3)$, tier =1]
]
\end{forest}}
\adjustbox{valign =c}{$  \begin{array}{c} \\
+ \end{array}$}
\adjustbox{valign =c}{
\begin{forest}
for tree = {grow' = 90}, nice empty nodes,
[
[$ \psi_{\vee} (a_1{,\,}a_2)$
 ]
 [$a_3$, tier =1] 
]
\path[fill=black] (.parent anchor) circle[radius=2pt];
\end{forest}}
\adjustbox{valign =c}{$  \begin{array}{c} \\ \\ \\
. \end{array}$}

\end{gathered}
\end{array}
     \end{equation*}
     \end{enumerate}
\end{example}
     
 \subsubsection{Compatibility conditions to ensure $\delta_\psi^2=0 $} 
\label{sec:compatibility.conditions}

Let $(\FM_\bullet,d) $ be a free or projective $\CO$-module resolution  of $\CO/\CI $ as in \eqref{eq:MdO}. Let $ \delta_\psi$ be the derivation of $S(\CT[\FM_\bullet]) $ associated to some arborescent operations $ \psi \colon t \mapsto \psi_t$ as in Section \ref{sec:constructingdeltapsi}.

\begin{proposition}
\label{prop:deltaPsiSquare}
Assume that the arborescent operations
$t \mapsto \psi_t $ have been chosen such that $\delta_\psi^2 $ vanishes on elements of degree less or equal to some $i \geq 1$.
Then the square of $\delta_\psi $ restricted to $\Tdeg[\FM_\bullet]_{i+1}$ takes values in $\vert \otimes_{\K} \FM_{i-1} $.
\end{proposition}

\begin{proof}
Let us use the recursive description of the differential $\delta_{\psi}$
given in Section \ref{sec:constructingdeltapsi}.
It follows from equation \eqref{eq:23matrix}
that $ \delta_\psi$, seen as a $ 3 \times 3$ matrix with respect to the decomposition 
\begin{equation} \label{eq:decomp}S(\CT[\FM_\bullet]) \,= \underbrace{\FM_{\bullet} \oplus \CO} \,\oplus\,  \Tdeg[\FM_\bullet]  \,\oplus\, \SdegFM.\end{equation}
takes the form:
\begin{equation}
 \label{eq:33matrix}
 \delta_\psi = \begin{pmatrix} d & - \psi & \ptr \circ \delta_\psi  \\[10pt] 0& * & \pv \circ \delta_\psi  \\[10pt] 0 & \r^{-1} & \p \circ \delta_\psi   \end{pmatrix} .  \end{equation}
 Above, we used the three canonical projections with respect to the decomposition in Equation \eqref{eq:decomp}:\footnote{The upper index of $\mathrm p$ refers to the polynomial degree in $\SdegFM $, while the symbols $| $ and $ \vee$ indicates whether one projects on trivial or on non-trivial trees.} 
\begin{equation}
\label{eq:three_projections}
\begin{array}{lcrcl}
 \ptr &\colon& S(\CT[\FM_\bullet]) &\longrightarrow& \FM_\bullet \oplus \CO , \vspace{1mm}
    \\
    \pv  &\colon& S(\CT[\FM_\bullet]) &\longrightarrow &\Tdeg[\FM_\bullet]   ,\vspace{2mm}\\
\p &\colon& S(\CT[\FM_\bullet]) &\longrightarrow&
    \SdegFM.\end{array}  
\end{equation}

\vspace{1mm}
\noindent
Let us compute the matrix that correspond to $\delta_\psi^2$ restricted to the subspace $S(\CT[\FM_\bullet])_{i+1}$: 
\begin{equation}\label{eq:matricdeltasquare} \delta_\psi^2 = \begin{pmatrix} {\mathrm{id}} \otimes d^2 &  - d\circ \psi - \psi \circ *   +\ptr\circ \delta_\psi \circ \r^{-1} & 0 \\[10pt] 0& *^2 
 + \pv\circ \delta_\psi \circ \r^{-1}  & 0 \\[10pt] 0 & p^{\geq 2} \circ \delta_\psi \circ \r^{-1} + \r^{-1} \circ * & 0 \end{pmatrix}   .\end{equation}
 In the first column, all terms are zero since $ d^2=0$. In the last column, all terms are zero since, by assumption of the proposition, the derivation  $\delta_\psi^2 =\tfrac{1}{2}\left[\delta_\psi,\delta_\psi\right]$ is zero on the subspace $\oplus_{k=0}^i S(\CT[\FM_\bullet])_k$.
 To establish Proposition  \ref{prop:deltaPsiSquare}, it remains only to show that the second and third terms in the second column are zero. For the third term, it is a tautological consequence of Equation \eqref{eq:star}:
$$ \p \circ \delta_\psi \circ \r^{-1} + \r^{-1} \circ * = \p \circ \delta_\psi \circ \r^{-1} -\r^{-1} \circ \r \circ \p \circ \delta_{\psi} \circ \r^{-1} =0.$$
For the second term, the computation is more involved:
$$  \begin{array}{rcl} *^2 &= & \r \circ {\mathrm{p}}^{\geq 2} \circ \delta_{\psi} \circ \r^{-1} \circ \r \circ {\mathrm{p}}^{\geq 2} \circ \delta_{\psi} \circ \r^{-1} \\ 
& =& \r \circ {\mathrm{p}}^{\geq 2} \circ \delta_{\psi}  \circ {\mathrm{p}}^{\geq 2} \circ \delta_{\psi} \circ \r^{-1}  \\
\end{array}   $$
Now, in view of $ \p + \ptr + \pv = {\mathrm{id}}$, we have
$$ \begin{array}{rcl} \delta_{\psi}  \circ {\mathrm{p}}^{\geq 2} \circ \delta_{\psi} &=& \delta_{\psi}^2 -  \delta_{\psi}\circ \ptr \circ \delta_{\psi} -  \delta_{\psi}\circ \pv \circ \delta_{\psi}
. \end{array} $$
The first term on the r.h.s. of this last equation can be dropped since when combined with  $\r^{-1} $ it gives no contribution by our induction hypothesis.
Hence:
$$   \begin{array}{rcl} *^2 &= & - \r \circ {\mathrm{p}}^{\geq 2} \circ \delta_{\psi}  \circ \ptr \circ \delta_{\psi} \circ \r^{-1} - \r \circ {\mathrm{p}}^{\geq 2} \circ \delta_{\psi}  \circ \pv \circ \delta_{\psi} \circ \r^{-1} \end{array}  $$
Obviously, $ {\mathrm{p}}^{\geq 2} \circ \delta_{\psi}  \circ \ptr = 0$ since  $ \delta_\psi$ of a trivial tree is a trivial tree. Therefore the first term vanishes. 
Now, upon restriction to the space $\CT[\FM_\bullet] $ of decorated trees, $  \p \circ \delta_\psi = \r^{-1} $, 
so that we are left with
$$   \begin{array}{rcl} *^2 &= & -  \pv \circ \delta_{\psi} \circ \r^{-1} \end{array}  $$
which means that the second term in the second column in the matrix \eqref{eq:matricdeltasquare} vanishes as well. This completes the proof. \end{proof}

\vspace{0.1cm}
\noindent
Let us derive a corollary
of Proposition \ref{prop:deltaPsiSquare}

\begin{corollary}
\label{cor:differentialexists}
There exists a choice of arborescent operations $t \mapsto \psi_t $ such that $\delta_\psi^2 =0$.
\end{corollary}

\noindent
To prove it, we first establish the following lemma.

\begin{lemma}
\label{cor:spadesuit}
Assume that the arborescent operations have been chosen such that $\delta_\psi^2 $ vanishes on elements of degree less or equal to $i$, 
then for any element $t[a_1, \dots, a_n] \in \CT[\FM_\bullet]_{i+1}$,
we have $\delta_\psi^2 (t[a_1, \dots, a_n]) = | \otimes_{\K} B(t[a_1, \dots, a_n ])$ with $ B(t[a_1, \dots, a_n ]) $ being an element in $\FM_{i-1} $ that satisfies $d (B(t[a_1, \dots, a_n ]))=0 $.
\end{lemma}
\begin{proof}
The recursion assumption implies $ \delta_\psi^2 \circ \delta_\psi (t[a_1, \dots, a_n])=0$, so that $\delta_\psi \circ \delta_\psi^2(t[a_1, \dots, a_n]) = \delta_\psi (| \otimes_\K  B(t[a_1, \dots, a_n ]))= | \otimes_\K d( B(t[a_1, \dots, a_n ])) =0$. 
\end{proof}

\begin{proof}[Proof of Corollary \ref{cor:differentialexists}]
The proof consists in constructing the arborescent operation $\psi$ degree by degree. More precisely, we intend to define $\psi_t[a_1 \dots,a_n]$ for every decorated tree $t[a_1, \dots,a_n] $ of degree $ i+1$ provided it is already defined on all decorated trees of degree less or equal to $i$. If the relation $\delta^2_{\psi} = 0$ holds on decorated trees of degree at most $i$,
Proposition \ref{prop:deltaPsiSquare} 
states that restriction of  $ \delta^2_{\psi} $ to elements of degree $i+1 $ 
is of the form $|\otimes_\K B $ for some $\CO$-linear map  $B$ from $ {\mathcal T}ree[\FM]_{i+1} \to \FM_{i-1}$.
Corollary \ref{cor:spadesuit} then assures that 
$B $ is valued in $d$-cycles of $\FM_{i-1}$. 
 By exactness of the complex $(\FM_{\bullet}, d)$,  $B $ takes values in $d$-boundaries.
 Since the $\CO$-module $(\Tdeg[\FM_\bullet])_{i+1}$
is free or projective by Proposition \ref{prop:stillfree}, 
 there exists an $\CO$-linear map 
 $$
\psi^{(i+1)} \colon (\Tdeg[\FM_\bullet])_{i+1} \rightarrow \FM_{i}
$$
which gives back $B$ when composed with $d\colon \FM_i \to \FM_{i-1}$, $d \circ \psi^{(i+1)} = B$.
The arborescent operation $\psi:t\mapsto \psi_t$, now 
extended by $ \psi^{(i+1)}$ on decorated trees of degree $ i+1$,
induces a differential $\delta_\psi $ that satisfies $\delta_{\psi}^2 = 0$ on all elements of degree $i+1$: By equation \eqref{eq:23matrix}, for any ordered decorated tre $t[a_1, \dots, a_n] $ of degree $ i+1$, 
the new value of $ \delta_\psi$ differs from the previous one by $ -|\otimesk \psi^{(i+1)}{{\circ {\mathrm{pr}}}}(t[a_1, \dots, a_n])=0$. When  applying $ \delta_\psi$ again to the expression just obtained, we get $| \otimesk B \circ {{\mathrm{pr}}}(t[a_1, \dot, a_n]) - | \otimesk d \psi^{(i+1)}{\circ\mathrm{pr}}(t[a_1, \dots, a_n]) $.
The construction continues by recursion.
\end{proof}

\begin{proposition}
    The  derivation $\delta_\psi $ of degree $ -1$ associated to an arborescent operation $ \psi \colon t \mapsto \psi_t$ squares to zero if and only if
\begin{equation}
\label{eq:deltaPsiSquare}\left(  \sum_{A \in \mathrm{InnVer}(t)} \left(\psi_{t_{\downarrow A}} \circ_{A} \psi_ {t_{\uparrow A}} - (-1)^{W_A} \psi_{\partial_A t}\right) - [d, \psi_t] + \pi^1\circ d\right) (a_1\otimes \dots \otimes a_n)=0. 
\end{equation}
Here $d$ is extended to the tensor product by the graded Leibniz rule and $\pi^1 $ stands for the canonical projection $\oplus_{j=1}^{\infty}\left(\otimes^{j}\FM_\bullet \right)\rightarrow \otimes^1 \FM_\bullet= \FM_\bullet$.
In the above, we use two abbreviations:
\begin{equation}\label{eq:conventions}\begin{array}{rrcl}  \left(\psi_{t_{\downarrow A}} \circ_{A} \psi_ {t_{\uparrow A}} \right) \colon &
a_1\otimes \dots \otimes a_n &\mapsto&  (-1)^{W_A}\psi_ {t_{\downarrow A}}(a_1, \dots, \psi_{t_{\uparrow A}}(a_A), \dots, a_n) \, ,\\[10pt]
 \left[ d, \psi_t \right] \colon
& a_1\otimes \dots \otimes a_n& \mapsto & d\circ \psi_t (a_1, \dots, a_n) + \sum_{i=1}^n (-1)^{W_i} \psi_t(a_1, \dots, da_i, \dots, a_n) \, .\\[10pt] 
.\end{array} 
\end{equation}
\end{proposition}
\vspace{-7mm} \noindent We remark that $ \pi^1\circ d ( a_1\otimes \dots \otimes a_n)$ can be non-zero only if $n$ is equal to $1$ or $2$ and, in the latter case, with at least one of the two elements being of degree $1$. 

\begin{proof}
The relation  $ \delta_\psi^2=0$ holds if and only if all terms in the matrix \eqref{eq:matricdeltasquare} are zero.
It was checked in the proof of Proposition \ref{prop:deltaPsiSquare} that all terms are zero, except maybe for the first term in the second column, which is equal to: 
$$   -d\circ \psi \circ {\mathrm{pr}} (t[a_1, \dots, a_n]) - \psi \circ * \circ {\mathrm{pr}}(t[a_1, \dots, a_n])  +\ptr\circ \delta_\psi \circ \r^{-1} \circ {\mathrm{pr}}(t[a_1, \dots, a_n])=0. $$
 Here $ {\mathrm{pr}}$ is the projection from ordered trees onto symetrized trees as in Equation \eqref{eq:psitDef}.
We have to check that this relation is equivalent to Equation \eqref{eq:deltaPsiSquare}.

\vspace{1mm}
\noindent
First we remark that $\ptr \circ \delta_{\psi} \circ \r^{-1}$ is non-zero only when evaluated on a tree with one root and two leaves:
$$
\adjustbox{valign =c}{$\begin{array}{c} \\ \vee[a_1, a_2] =\end{array}$} 
\adjustbox{valign =c}{\begin{forest}
for tree = {grow' = 90}, nice empty nodes,
[
 [$a_1$, tier =1]
 [$a_2$, tier =1] 
]
\path[fill=black] (.parent anchor) circle[radius=2pt];
    \end{forest}}$$
    where  $a_1$ or $a_2$ are of degree  $1$. This allows us to identify $\ptr \circ \delta_{\psi} \circ \r^{-1} (t[a_1, \dots, a_n])$ with $\vert \otimes_{\K} \pi^1 \circ d (a_1 \otimes \dots \otimes a_n)$.

\vspace{1mm}
\noindent
In view of Equations \eqref{eq:star} and \eqref{eq:NewConventions} 
we have
$$
* \circ \mathrm{pr} ( t[a_1, \dots, a_n]) = 
\sum_{\scalebox{0.5}{$\begin{array}{r}A \in \mathrm{Ver}(t)
\end{array}$}}-(-1)^{W_A}
\left( t_{\downarrow A}[a_1, \dots,\psi_{t_{\uparrow A}}(a_A) , \dots a_n]+(-1)^{W_A}(\partial_A t)[a_1, \dots,a_n].
\right)$$
Therefore, we can rewrite $-d\circ \psi  - \psi \circ *$ as  
\begin{equation}  \label{eq:psi}
-d\circ \psi  - \psi \circ * = \sum_{A \in \mathrm{InnVer}(t)} \left(\psi_{t_{\downarrow A}} \circ_{A} \psi_ {t_{\uparrow A}} - (-1)^{W_A}\psi_{\partial_A t}\right) - [d, \psi_t].
\end{equation}
This completes the proof.
\end{proof}

\noindent
For future use, it will be useful to write Equation \eqref{eq:deltaPsiSquare}
under the following form:

 \begin{equation}
\label{eq:psi3}
 \begin{gathered}
     d \circ \psi_t (a_1, \dots, a_n) 
 \, = \\ \,\left(  \sum_{A \in \mathrm{InnVer}(t)} \left(\psi_{t_{\downarrow A}} \circ_{A} \psi_ {t_{\uparrow A}} - (-1)^{W_A}\psi_{\partial_A t}\right) -\sum_{i=1}^n(-1)^{W_i}\psi_t(a_1, \dots, da_i, \dots,a_n) + \pi^1\circ d\right) (a_1\otimes \dots \otimes a_n).
 \end{gathered}
 \end{equation}
This form describes the recursion procedure used to construct $ \psi$, since $t[a_1, \dots, a_n] $ has degree $ i+1$, the terms that appear in the second line of \eqref{eq:psi3} have degree less or equal to $ i$.
For $ n \geq 3$, Equation \eqref{eq:psi3} also has the compact expression: 
\begin{equation}
   \label{eq:psi4} 
\sum_{\scalebox{0.5}{$\begin{array}{r}A \in \mathrm{Ver}(t)\end{array}$}} \psi_{t_{\downarrow A}} \circ_{A} \psi_ {t_{\uparrow A}} (a_1, \dots, a_n)= 
\sum_{\scalebox{0.5}{$\begin{array}{r}A \in \mathrm{InnVer}(t)\\   \end{array}$}}  (-1)^{W_A} \psi_{\partial_A t} (a_1, \dots, a_n), 
   \end{equation}  with the additional convention that $\psi_{t_{\downarrow A}} = -d$ and 
   $\psi_{t_{\uparrow A}} = \psi_t$
   if $A$ is the root of $t$, while $\psi_{t_{\uparrow A}} = -d$ and $\psi_{t_{\downarrow A}} = \psi_t$ if $A$ is a leaf of $t$.

\subsection{Arborescent Koszul-Tate resolutions}

\noindent
We now present the main theorem of the article.

\begin{theorem}
\label{thm:isKT}
    Let $\CI$ be an ideal in a commutative algebra $\CO$.
    For any free or projective $\mathcal O$-module resolution of $\CO/\CI$ 
$$
\begin{tikzcd}
	\cdots \arrow[r ] &\FM_k \arrow[r, "d" ]& \cdots \arrow[r, "d" ] &\FM_1 \arrow[r, "d" ] &\CO \arrow[r] & \CO/\CI \, \, ,
\end{tikzcd} 
$$
\noindent 
\begin{enumerate}
\item 
there exist arborescent operations $\psi \colon t \mapsto \psi_t $ such that the derivation $\delta_\psi$
described in Section~\ref{sec:constructingdeltapsi} 
squares to zero.
\item For any such a choice of $\psi $, the pair $(S(\CT[\FM_\bullet]), \delta_\psi) $ is a Koszul-Tate resolution of $\CO/\CI$. 
\end{enumerate}
    \end{theorem}


\begin{proof}
The first item was established in Corollary \ref{cor:differentialexists}.
In order to prove the second item, it suffices to check that for any such a choice $\psi $, the homology of $\delta_\psi $ is zero in every positive degree and $\CO/\CI $ in degree $0$.

\vspace{0.1cm}

\noindent
Since $S(\CT[\FM_{\bullet}])_1 = | \otimes_\K \FM_1 $,  $\delta_\psi (| \otimes a) = da $ for every $a \in \FM_1 $, and  the image of $d$  is $ \CI$, the degree $0$ homology of $\delta_\psi $ is $\CO/\CI $.

\vspace{0.1cm}

\noindent
Let us now prove that any $\delta_\psi$-closed element of degree $i \geq 1 $ is exact. 
Let $a\in S( \CT[\FM_{\bullet}])_i$ be a $\delta_\psi$-closed element of degree $i$, $\delta_\psi a =0$. Let us decompose $a$ as a sum of two terms: 
\begin{equation}
\label{element_a} a =a^{\geq 2}+a^1  \end{equation}
 with
 $a^{\geq 2} \in  S^{\geq 2}( \mathcal{T}ree[\FM_{\bullet}])$ and $a^1 \in  S^{1}( \mathcal{T}ree[\FM_{\bullet}]) \simeq \mathcal{T}ree[\FM_{\bullet}]$.

\vspace{0.1cm}
\noindent
Let us compute $\delta_\psi (\r (a^{\geq 2}))$. It admits a component in $ S^{\geq 2}(\CT[\FM_\bullet])$, which is $a^{\geq 2}$ by definition of $\delta_\psi $, see  Equation \eqref{eq:23matrix}.
As a consequence,
$$ a':=a -\delta_\psi \left( \r(a^{\geq 2})\right)  $$
 and $ \delta_\psi$-closed  element in  $ S^1(\CT[\FM_\bullet])_i \simeq \CT[\FM_\bullet]_i$.

\vspace{1mm}
\noindent
Recall from Equation \eqref{eq:23matrix} that the component of $\delta_\psi $ that goes from $\CT[\FM_\bullet]_i $ to $S^{\geq 2}(\CT[\FM_\bullet)_{i-1} $ is the unroot map $\r^{-1} $.
Since $a'$ is  $\delta_\psi $-closed, it belongs to the kernel of the unroot map $\r^{-1} $. The kernel of $\r^{-1} $ is the space of trivial decorated trees; this implies that there is $b \in \FM_i $ such that $ a' = | \otimesk b $.
Since $0 = \delta_\psi (a')= | \otimes_\K db $, the element  $b$ has to be $ d $-closed. By exactness of $\FM_\bullet $, $ b = d b'  $ for some $b' \in \FM_{i+1} $. By construction
 $ \delta_\psi ( | \otimes_\K b') = | \otimes_\K b = a'$.
 This proves that 
   $$ a = \delta_\psi (\r(a_{\geq 2}) + | \otimes_\K b') $$
   is exact, so that the homology of $\delta_\psi $ is zero.  
\end{proof}

\vspace{0.1cm}

\noindent
Theorem \ref{thm:isKT} allows to give the following definition.

\begin{definition}
For $(\FM_\bullet, d) $ a free or projective  $\CO$-module resolution of $\CO/\CI $, we call any Koszul-Tate resolution of $ \CO/\CI$ as in Theorem \ref{thm:isKT} an \emph{arborescent Koszul-Tate resolution} of $\CO/\CI$.  
 
\end{definition}

\section{The induced $A_\infty$-algebra structure on $(\FM_\bullet,d)$.}
\label{sec:Cinfty}

\subsection{$(\FM_\bullet,d)$ as a deformation retract of its arborescent Koszul-Tate resolution}

\noindent Consider an arborescent Koszul-Tate resolution of $\CO/\CI $ associated to a free or projective $\CO$-module resolution $(\FM_\bullet,d) $ with arborescent operations $ \psi$  and differential  $\delta_\psi $. 
 Since  \emph{(i)} $(S(\CT[\FM_\bullet]),\delta_\psi) $  is also a free or projective $\CO$-module resolution of $\CO/\CI$ by  the second item of Theorem \ref{thm:isKT}, \emph{(ii)} any  two  free or projective $\CO$-module resolutions of the same $ \CO$-module are homotopy equivalent as $\CO $-modules \cite{Eisenbud}, $(S(\CT[\FM_\bullet]),\delta_\psi) $ 
 has to be homotopy equivalent to $(\FM_\bullet,d) $ as $\CO $-modules.
So while existence is a general fact, in our case, we are able describe an explicit
expression for this $\CO$-linear homotopy equivalence and show that it is even a deformation retract.

\vspace{0.1cm}
\noindent
The  $\CO $-linear inclusion map
 $$ \begin{array}{rrcl}\mathrm{Incl} \colon &  \FM_\bullet  \oplus \CO &\to & S(\CT[\FM]_\bullet)  \\& a +  F &\mapsto & |\, [a]  + F\end{array} $$
 is obviously a chain map. Let us consider a second map, 
$$\mathrm{Proj}_\psi \colon  S(\CT[\FM_\bullet]) \to \FM_\bullet \oplus \CO $$ 
defined, in the decomposition \eqref{eq:decomp}, by the matrix
 \begin{equation} \label{eq:matricesInclProj} {\mathrm{Proj}}_\psi =  \begin{pmatrix} {\mathrm{id}}& 0 & \psi \circ \r  \end{pmatrix} . \end{equation}
In particular, for all $ j \geq 2$, one has  \begin{align*} \mathrm{Proj}_\psi \left( t_1[a_1^1, \dots, a_{n_1}^1] \odot \cdots \odot  t_j [a_1^j, \dots, a_{n_j}^j]  \right)
       =  \psi_{\r(t_1, \dots, t_j )}(a_1^1, \dots, a_{n_1}^1, \dots, a_1^j, \dots, a_{n_j}^j ) \, ,  \end{align*}
       where the tree $ \r(t_1, \dots, t_j ) $ is obtained by rooting the trees $t_1, \dots, t_j$ altogether.

\begin{proposition}
\label{prop:Homotopy}
Consider an arborescent Koszul-Tate resolution of $\CO/\CI $ associated to a free or projective $\CO$-module resolution $(\FM_\bullet,d) $ with arborescent operations $ \psi$  and differential  $\delta_\psi $.
The $\CO $-linear maps 


\begin{center}
\begin{tikzpicture}[scale=1]
\coordinate (A) at (0,0);
\coordinate (B) at (2,0);

\node at (A) [left] {$S(\CT[\FM])$};
\node at (B) [right] {$\FM$};

\draw[->] (0,0.2) -- (2,0.2) node[midway, above] {$\mathrm{Proj}_{\psi}$};
\draw[<-] (0,-0.2) -- (2,-0.2) node[midway, below] {$\mathrm{Incl}$};
\end{tikzpicture}
\end{center}
are homotopy inverse one to an other. More precisely,
 $$ \left\{ \begin{array}{rcl} {\mathrm{Proj}}_\psi \circ {\mathrm{Incl}} &=& {\mathrm{id}} \\ 
{\mathrm{Incl}} \circ {\mathrm{Proj}}_\psi  &= & {\mathrm{id}} - \left(h \circ \delta_\psi + \delta_\psi \circ h \right)  \end{array}\right.$$
where the homotopy $h$ is given by $ h:= \r \circ \p 
$.
Moreover, the so-called side relations (see \cite{zbMATH03080175}) are satisfied, i.e.: 
 \begin{equation}\label{eq:additional} h^2 =0 \, , \, h \circ {\mathrm{Incl}} =0 \, , \, {\mathrm{Proj}}_\psi \circ h =0. \end{equation}
%
\end{proposition}
\begin{proof}
A direct computation gives the three relations listed in Equation \eqref{eq:additional}. Noting that 
 $${\mathrm{Incl}} = \begin{pmatrix} {\mathrm{id}}\\ 0 \\ 0  \end{pmatrix} $$
the relation $  {\mathrm{Proj}}_\psi \circ {\mathrm{Incl}} = {\mathrm{id}}$ is evident. 
We are left with  proving:
\begin{equation}
    {\mathrm{Incl}} \circ {\mathrm{Proj}}_\psi  =  {\mathrm{id}} - \left(h \circ \delta_\psi + \delta_\psi \circ h \right). 
    \label{sss}
\end{equation}
For the left-hand side we have
$${\mathrm{Incl}} \circ {\mathrm{Proj}}_\psi  = \begin{pmatrix} \mathrm{id} & 0 &  \psi \circ \r 
 \\[10pt] 0 &0 & 0 \\[10pt] 0&0 &  0  \end{pmatrix}   .$$
On the other hand, using \eqref{eq:33matrix} to represent $\delta_\psi$ and noting that $h$ becomes 
 $$ h = \begin{pmatrix}  0&0 & 0\\0&0 & \r \\ 0&0 & 0\end{pmatrix}   $$
in the  decomposition \eqref{eq:decomp}, 
the matrix multiplications yield 
 $$ h \circ \delta_\psi + \delta_\psi \circ h = \begin{pmatrix} 0 & 0 &  -\psi \circ r  
 \\[10pt] 0 &\r \circ \r^{-1} &  \r \circ  \p \circ \delta_\psi + * \circ \r \\[10pt] 0&0 &  \r^{-1} \circ \r \end{pmatrix}  .$$
By definition, $\r \circ \r^{-1}$ is the identity on $ \Tdeg[\FM_\bullet]  $, while $ \r^{-1} \circ \r$ is the identity on $\SdegFM $. 
Equation \eqref{sss} is now seen to hold true provided only that $* \circ \r + \r \circ  \p \circ \delta_\psi=0$ on $\SdegFM $, which follows immediately upon using the definition of $ *$ given in Equation \eqref{eq:star}.
\end{proof}

\vspace{1mm}
\noindent
We conclude this section with a statement that may be of interest in its own.

\begin{proposition}
\label{prop:quotient}
Let $\mathfrak J $ be a graded differential ideal of an arborescent Koszul-Tate resolution $(S(\CT[\FM_\bullet]), \delta_{\psi}) $, which is  $h$-stable, $\delta_{\psi}$-stable, and contained in the kernel of $\mathrm{Proj}_\psi$.
Then $S(\CT[\FM_\bullet])/\mathfrak J$, equipped with the differential $ \overline{\delta}_\psi$ induced from $ \delta_\psi$, is a graded commutative algebra whose homology is $0$ in every positive degree and $ \CO/\CI$ in degree $0$. 
\end{proposition}
\begin{proof}  The short exact sequence of differential graded $ \CO$-modules
$$ 0 \longrightarrow (\mathfrak J, \left.\delta_{\psi}\right|_{\mathfrak J}) \longrightarrow (S(\CT[\FM_\bullet]), \delta_{\psi}) \longrightarrow (S(\CT[\FM_\bullet])/\mathfrak J, \underline{\delta}_{\psi}) \longrightarrow 0
$$ 
yields a long exact sequence 
 $$ \cdots \longrightarrow H_i (\mathfrak J,\left.\delta_{\psi}\right|_{\mathfrak J}) \longrightarrow  H_i (S(\CT[\FM_\bullet]),\delta_{\psi}) \longrightarrow  H_i (S(\CT[\FM_\bullet])/\mathfrak J,\underline{\delta}_{\psi})  \longrightarrow H_{i-1} (\mathfrak J,\left.\delta_{\psi}\right._{|_{\mathfrak J}}) \longrightarrow \cdots .$$
Since $\mathfrak J$ is in the kernel of $\mathrm{Proj}_\psi$ and is $h$-stable, the relation $\left.\delta_{\psi}\right|_{\mathfrak J}\circ \left.h\right|_{\mathfrak J} + \left.h\right|_{\mathfrak J} \circ \left.\delta_{\psi}\right|_{\mathfrak J} = \id_{\mathfrak J}$ holds. We conclude that $H_{i}(\mathfrak J,  \left.\delta_{\psi}\right|_{\mathfrak J}) = 0 $ for all $ i \geq 0$. Hence the homologies of 
$S(\CT[\FM_\bullet]),\delta_{\psi}) $ and  $ S(\CT[\FM_\bullet])/\mathfrak J,\underline{\delta}_{\psi}) $ are isomorphic. This proves that $(S(\CT[\FM_\bullet])/\mathfrak J,\left.\delta_{\psi}\right|_{\mathfrak J})$ has trivial homology except in degree $ 0$ where it is  $ \CO/\CI$. 
\end{proof}

\noindent In Example \ref{ex:4.1} below we illustrate how this proposition can be applied in concrete cases.

\subsection{Construction of the $A_\infty$-algebra structure on $(\FM_\bullet,d)$.}

\noindent
It is an old problem (see, e.g., \cite{zbMATH03699062} and \cite{KUSTIN1994371}) to know whether or not projective $ \CO$-resolutions $(\FMb,d) $ of $\CO/\I $ admit a graded commutative algebra product. It was observed by Buchsbaum and Eisenbud \cite{Buchsbaum-Eisenbud} that this is possible only up to homotopy.
We make this statement more precise in this section and show that there is always an $ A_\infty$-algebra structure on $ \FMb \oplus \CO$, which is "graded commutative" in the sense that it is even a $C_\infty $-algebra structure  \cite{MR2503530,Getzler-Cheng}. 

\vspace{1mm}\noindent 
In particular, we show here that some components of the arborescent operations $\psi$ of an arborescent Koszul-Tate resolution encode a $C_\infty$-algebra structure on $ (\FMb \oplus \CO ,d)$. 
Let us first provide the definition of an $A_\infty $-algebra using the sign and degree conventions of \cite{Getzler-Cheng} (for the original definition \cite{StasheffAinfty1, StasheffAinfty2, Keller1999IntroductionT} see Appendix \ref{app:alternative}).

\begin{definition}
Let $ (\mathfrak A,d)$ be a complex of $ \CO$-modules. A sequence of  degree $+1$ $\mathcal O $-linear maps $m_i \colon \otimes^i \mathfrak A_\bullet \to \mathfrak A_{\bullet} $, 
is said to be an  $ A_\infty$-algebra if it satisfies the "higher associativity" conditions  
\begin{equation}
\label{eq:ainfty1}
\sum_{\scalebox{0.5}{$\begin{array}{c} i,k \geq 0, j \geq 1 \\ i+j+k = n \end{array}$}} 
m_{n-j+1}(\id^{\otimes i}\otimes m_{j} \otimes \id^{\otimes k}) = 0
\end{equation}
for all $n \geq 1$ where $m_1=
d $.
\end{definition}

\vspace{1mm}
\noindent
In \cite{Kontsevich:2000yf},   it is stated that every complex $(\mathfrak A_\bullet ,d)$ which is homotopy equivalent to an $A_\infty$-algebra inherits one as well.
 This result is called the \emph{homotopy transfer theorem}. 
The  homotopy transfer theorem makes more than just announcing the existence of an $A_\infty$-algebra structure, an explicit formula as a particular summation over ordered trees was given in \cite{Kontsevich:2000yf}, see also \cite{Polishchuk_Ainfty}, \cite{markl2004transferring},
\cite{Getzler-Cheng} or \cite{Vallette}.
Let us recall how this construction works when one starts from a homotopy retract of a differential graded algebra $ (S_\bullet,\star, d_S)$, where the above-mentioned sum is over binary ordered trees only.
We follow the sign conventions of \cite{Getzler-Cheng}
and make a shift in degree by considering $ S_i$ to be of degree $ i-1$. We denote this new degree for any homogeneous $a \in S_\bullet$ by $ |a|'$, $ |a|'=|a|+1$. 
Consider the new binary operation on $S_\bullet$ given for all homogeneous $a,b\in S_\bullet$ by
 \begin{equation}\label{eq:starprime}  a \star' b := (-1)^{|a|'} a \star b .\end{equation}
 We now have for all $ a,b \in S_\bullet$:
  $$ d_S(  a \star' b)= -(d_S a) \star' b  +  a \star' (d_S b) .$$
Assume we are given a homotopy retract from  $(S_\bullet,d_S) $  to a complex $({\mathfrak{A}}_\bullet,d) $:
\vspace{-9mm}
$$  
\begin{tikzcd}[column sep = 4em]
(S_\bullet,d_S)  \arrow[r, bend left=20, "{\mathrm{proj}}", shift ={(0 ,1mm)}]  \arrow[out=225, in=135, looseness=8, loop, " h "] & (\mathfrak A_\bullet,d) \arrow[l,  bend left=20, swap, "{\mathrm i}", shift = {(0, -1mm)}]
\end{tikzcd}.
$$

\vspace{-9mm}
\noindent
Here, $h$ is a homotopy $h$, i.e.\ $  d \circ h+ h \circ d = {\mathrm{id}} + {\mathrm{i}} \circ {\mathrm{proj}} $, the inclusion $\mathrm{i} $  and the projection $ {\mathrm{proj}}$ are chain maps, and the \emph{side relations}  are satisfied:
$$ h^2 =0 \, , \, h \circ \mathfrak i =0 \, , \, {\mathrm{proj}}\circ h=0 .$$ 
We now also shift the degree by $ +1$  on $ \mathfrak A_\bullet$ by $ +1$ and denote this new graded vector space by  ${\mathfrak{A}}_\bullet[1] $.
Then the $A_\infty $-algebra structure on $({\mathfrak{A}}_\bullet[1] , d)$ is given by $ m_1=d$ and
\begin{equation}
\label{eq:allchi}
\begin{array}{llllll}
        m_n\colon & \otimes^n {\mathfrak{A}}_\bullet[1] & \rightarrow & {\mathfrak{A}}_\bullet [1] \\
     & a_1, \dots, a_n & \mapsto & \sum_{ t \in Y_{n}}   \chi_t ( a_1, \dots,  a_n) 
\end{array}
\end{equation}
where $Y_n$ stands for the set of ordered binary trees with $n$ leaves, and $\chi_t \colon \otimes^n \mathfrak A_\bullet[1] \to \mathfrak A_\bullet[1] $ is the map described by the following procedure:  
\begin{enumerate}
    \item Put $ {\mathrm{i}}(a_1), \dots, {\mathrm{i}}(a_n)$ to the leaves of $t$,  
    \item replace all inner edges by $h$ and   each inner vertex as well as the root by the product $\star'$;
    \item  finally apply ${\mathrm{proj}} $ to the result.
\end{enumerate}
For instance, for the tree $\ltree $, the procedure gives
the map  
\begin{equation} 
\label{eq:ltree0}
\begin{array}{rcll}
\chi_{\ltree}  \colon &  \otimes^n \mathfrak A[1] & \longrightarrow & \mathfrak A[1]   \\ &(a,b,c) &\mapsto &  {\mathrm{proj}}( h( {\mathrm{i}}(a) \star'  {\mathrm{i}}(b)) \star'  {\mathrm{i}}(c) ) .\end{array}
\end{equation}
The homotopy transfer theorem then states in this case that the collection of maps $ (m_n)_{n \geq 1}$ defines an $A_\infty$-algebra.

\vspace{0.1cm}
\noindent
According to Proposition \ref{prop:Homotopy}, an arborescent Koszul-Tate resolution $(S(\CT[\FM_{\bullet}]) , \delta_\psi)$ constructed out of an $\CO $-module resolution  $(\FM_\bullet, d)  $ is a DGA which is is homotopy equivalent to the initial module $\CO$-resolution $(\FM_\bullet \oplus \CO, d) $ and the homotopy equivalence satisfies the side relations. 
The above construction therefore applies.
In our case, Equation \eqref{eq:starprime} becomes
$$  {\mathrm{pr}} (t_1[a_1, \dots, a_n]) \star'{\mathrm{pr}}  (t_2[[b_1, \dots, b_m])= (-1)^{|t_1[a_1, \dots, a_n]|'} {\mathrm{pr}} (t_1[a_1, \dots, a_n]) \odot {\mathrm{pr}} (t_2[b_1, \dots, b_m]) $$
 for all ordered decorated trees $ t_1[a_1, \dots, a_n] $ and $ t_2[b_1, \dots, b_m]$.
 Recall that ${\mathrm{pr}} $ stands for the projection from ordered decorated trees to $ \CT[\FM_{\bullet}] $.
Note that if  $ t[a_1, \dots,a_n] \in Tree[\FMb] $ is  a binary tree, then the initial degree is given by
$$ |t[a_1, \dots,a_n]|=  |a_1| + \cdots + |a_n| + n-1,$$
so that the shifted degree $|t[a_1, \dots,a_n]|'=|t[a_1, \dots,a_n]|+1$ is given by
\begin{equation}\label{eq:degreetree} |t[a_1,\dots,a_n] |'= |a_1|' + \cdots + |a_n|'.\end{equation}

\vspace{1mm}
\noindent
Since the homotopy is the root map $\r$, the inclusion of $ \FMb $ is the inclusion using trivial trees,  and the projection is ${\mathrm{Proj}}_\psi = \psi \circ \r $, one checks directly that for any binary  decorated tree $ t[a_1, \dots,a_n]$
\begin{equation} \label{sign}
\chi_t ( a_1, \dots, a_n) = (-1)^{P'(t[a_1, \dots, a_n])} \psi_t(a_1, \dots, a_n)  .   
\end{equation} 
Here $P'(t[a_1, \dots,a_n])$ is the sum of the shifted degrees of all left  decorated subtrees of inner vertices and roots in $t[a_1, \dots,a_n]\in \CT[\FM_\bullet]$, where one makes use of Equation \eqref{eq:degreetree} for explicit calculations.
Alternatively $ P'(t[a_1, \dots,a_n])$, modulo $2$, is the sum of the degrees of those decorations for which the path from the root to that leaf makes an odd number of left turns.
We illustrate  this by computing $ m_2$ and $ m_3$ explicitly.

\vspace{1mm}
\noindent
 For $ n=2$, the sum in Equation \eqref{eq:allchi} runs only over  one tree, namely $ \vee$. Thus, in this case, the procedure gives 
$$\begin{array}{rcl}  m_2(a_1,a_2)&=&\chi_\vee (a_1,a_2)  
=\psi \circ \r ( |[a_1]\star' |[a_2])  = (-1)^{|a_1|'} \psi \circ \r ( |[a_1]\odot |[a_2])   \\[2mm] & =&  (-1)^{|a_1|'} \psi \circ {\mathrm{pr}} \, ( \vee[a_1,a_2]) =  (-1)^{|a_1|'}  \psi_\vee (a_1,a_2)\end{array}$$
for every $a_1,a_2 \in \FM_\bullet$.
This evidently agrees with \eqref{sign}.

\vspace{1mm}
\noindent
For $n=3$, it follows from Equation \eqref{eq:ltree0} that $\chi_{\ltree}$ is given by: 
$$
\begin{array}{rcl}\chi_{\ltree} ( a_1,a_2,a_3) &=&     \psi \circ \r \left(\r (|[a_1] \star' |[a_2] ) \star' |[a_3] \right)\\[1mm]
&=&     (-1)^{|a_1|'}  \,\psi \circ \r \left(\r (|[a_1]  \odot  |[a_2] )  \star'  |[a_3] \right)
\\[1mm]& =&  (-1)^{|a_1|' }  \, \psi \circ \r \left( {\mathrm{pr}} (\vee[a_1,a_2])  \star' |[a_3] \right) 
\\[1mm] & =&  (-1)^{|a_1|'+|\vee[a_1,a_2]|'}  \, \psi \circ \r \left( {\mathrm{pr}} (\vee[a_1,a_2])  \odot |[a_3] \right) \\[1mm]
&=&  (-1)^{|a_1|'+|a_1|'+|a_2|'} \, \psi \circ {\mathrm{pr}} \left( \ltree [a_1,a_2,a_3] \right) \\[1mm]
&=&  (-1)^{|a_2|'} \, \psi_{\ltree}  (a_1,a_2,a_3).
\end{array}
$$
A similar computation can be done for $ \psi_{\rtree}$, and 
we obtain that for every $a_1,a_2,a_3\in \FMb$:
 \begin{equation}
 \label{eq:m3}
 m_3(a_1,a_2,a_3)=  (-1)^{|a_2|'}  \psi_{\ltree}(a_1,a_2,a_3) + (-1)^{|a_1|'+|a_2|'}\psi_{\rtree}(a_1, a_2, a_3)  . \end{equation}

\noindent Furthermore, the homotopy transfer theorem described above yields more when $(S_\bullet, d_S,\star) $ is a graded \emph{commutative} algebra.
Theorem 12 in \cite{Getzler-Cheng} states that, in this case, the $ A_\infty$-algebra obtained using the explicit formulas above is a $ C_\infty$-algebra:

\begin{definition}
\cite{MR2503530,Getzler-Cheng} An $A_\infty$-algebra structure $(m_n)_{n \geq 1} $  on a graded $ \CO$-module ${\mathfrak{A}}_\bullet $  is said to be a \emph{$ C_\infty$-algebra} if the following condition is satisfied: for every $ n \geq 1$ and all homogeneous $ a_1, \dots,a_n \in {\mathfrak A}_\bullet$, one has
\begin{equation}\label{eq:Cinfty} \sum_{\sigma \in {\mathrm{Sh}}(i,n)} 
\theta'(\sigma,a)  \, m_n(a_{\sigma(1)}, \dots, a_{\sigma(n)})  =0 .\end{equation}
 Here ${\mathrm{Sh}}(i,n)$ stands for the set of all $(i,n)$-shuffles of $\{1, \dots, n\} $, i.e. all permutations $ \sigma$ of $ \{1, \dots,n\}$ such that  $ \sigma(a) \leq \sigma (b) $ if   $1 \leq a \leq b \leq i$ or $i+1 \leq a \leq b \leq n$.
Furthermore, 
$\theta'(\sigma,a) $ is the Koszul signature of $\sigma$  defined to be the unique sign that satisfies
 $$ a_{\sigma(1)} \odot \dots \odot a_{\sigma(n)} = \theta'(\sigma,a) \, a_1 \odot \dots \odot a_n  $$
  within the graded symmetric algebra $ S(\mathfrak A_\bullet)$  generated by $\mathfrak A_\bullet $.
\end{definition}

\noindent
The next theorem concludes this discussion. 

\begin{theorem}
\label{thm:ainfty}
Consider an arborescent Koszul-Tate resolution of $\CO/\CI $ associated to a free or projective $\CO$-module resolution $(\FM_\bullet,d) $ with arborescent operations $ \psi$  and differential  $\delta_\psi $.
Then the following structure maps $(m_n)_{n\geq 1} $  define an $A_\infty $-algebra structure on the graded free module defined to be $\FM_{i+1}$ in degree $i\geq 0$, $ 
\CO$ in degree $-1$, and $0$ in degree $ \leq -2$:

\begin{enumerate}
\item For $n=1 $, we set $ m_1=
d$.
\item For $ n=2$, we set for all $ a_1,a_2 \in \FM_\bullet$ and $ F_1,F_2 \in \CO$  $$\begin{array}{rrcl}
    m_2\colon & \otimes^2  \left( {\FM}_\bullet \oplus \CO \right) & \rightarrow &  {\FM}_\bullet  \oplus \CO \\
     & (a_1 + F_1 \, ,  \, a_2+F_2) & \mapsto & -F_1a_2+ (-1)^{|a_1|+1} F_2 a_1  +  (-1)^{|a_1|+1} \psi_\vee (a_1,a_2) - F_1 F_2
 \end{array} .$$
\item For all $n \geq 3$, we set for all homogeneous $ a_1, \dots, a_n \in \FM_\bullet$ and $F_1, \dots, F_n \in\CO$
 \begin{equation}
     \label{eq:homotopytransfer} \begin{array}{rrcl}
    m_n\colon & \otimes^n  \left( {\FM}_\bullet \oplus \CO\right) & \rightarrow & {\FM}_\bullet  \oplus  \CO \\
     & (a_1+F_1, \dots, a_n+F_n) & \mapsto &\sum_{ t \in Y_n} (-1)^{P'(t[a_1, \dots,a_n])
     } 
     \, \,  \psi_t(a_1, \dots, a_n), 
 \end{array} 
 \end{equation}
 where $Y_n$ is the set of ordered binary trees with $n$ leaves, and 
$P'(t[a_1, \dots,a_n])$ is defined in the paragraph following Equation \eqref{sign}.

 \end{enumerate}
Moreover, the structure maps $(m_n)_{n \geq 1} $ satisfy the defining property \eqref{eq:Cinfty} of a $C_\infty$-structure.

\end{theorem}

\noindent
It is difficult to find a precise and complete proof of the homotopy transfer theorem in the literature. Also, having an independent argument is of use. 
A direct proof of Theorem \ref{thm:ainfty} is therefore given in Appendix \ref{app:alternative}.

\vspace{1mm}
\noindent
To conclude this section,
in view of Theorem \ref{thm:ainfty}, the following definition makes sense.

\begin{definition}
We call the $A_\infty$-algebra brackets  defined in Theorem \ref{thm:ainfty}
the \emph{arborescent $A_\infty$-algebra structure} on $\FM_\bullet\oplus \CO$. 
\end{definition}

 \section{Examples of arborescent Koszul-Tate resolutions}
 \label{sec:examples}

\noindent  In section \ref{sec:basicExamples}, we describe examples of arborescent Koszul-Tate resolutions of ideals $ \CI \subset \CO$ which are not complete intersections, and are therefore not covered by the Koszul resolution, and for which there is no explicit construction known. In section
\ref{sec:DGCA}, we assume that the $\CO$-module resolution $(\FM_\bullet,d) $  of  $ \CO/\CI$ admits a graded commutative associative algebra structure, which gives another class of examples for arborescent Koszul-Tate resolutions.
In particular, this class contains Koszul resolutions of complete intersections and Taylor resolutions of monomial ideals.

\subsection{Non-complete intersections}
\label{sec:basicExamples}

Here we will provide three examples, where the ideals are not complete intersections.
In each one of these cases, we will first provide an elementary $\CO$-module resolution and then construct arborescent operations $ \psi$ ensuring $ \delta_\psi^2=0$.
\vspace{0.1cm}

\begin{example}[Quadratic Functions]\label{ex:4.1}
\normalfont
  Let $\CO$ be  the polynomial algebra in the two variables $x$ and $y$, and $\CI\subset \CO $ the the ideal generated by quadratic elements, i.e. by $x^2, xy, y^2$.  This ideal has the following $\CO$-module resolution of length two:
$$
    \begin{tikzcd}
    0 \arrow[r ] &\underbrace{\FM_2}_{rk = 2} \arrow[r, "d" ] &\underbrace{\FM_1}_{rk = 3} \arrow[r, "d" ] &\CO \arrow[r, ->>] & \CO/\CI \, .
\end{tikzcd}
$$ 
Let us denote the three generators of degree $1$ by $\pi^{xx}$, $\pi^{xy}$, and $\pi^{yy}$ and the two generators of degree $2$ by $\pi^{xxy}$ and $ \pi^{xyy}$. The differential $d$ is defined as follows on them:  
\begin{equation*}
\begin{array}{ll}
    \hbox{Degree $1$} & d \pi^{xx} = x^2, \qquad d \pi^{xy} = xy,\qquad d \pi^{yy} = y^2, \\
    \hbox{Degree $2$} & d \pi^{xxy} = x \pi^{xy} - y \pi^{xx}, \qquad  d \pi^{xyy} = x \pi^{yy} - y \pi^{xy}.
\end{array}
\end{equation*}
Here the arborescent operations  $\psi_t$ are unique. For degree reasons, they vanish except for $t=\vee$, where one finds by a direct computation: 
   \begin{equation*}
\begin{array} {l}
       \psi_{\vtree} (\pi^{xx},\pi^{xy}) = x\, \pi^{xxy}, \\
       \psi_{\vtree}(\pi^{xy},\pi^{yy}) = y\,\pi^{xyy} \\
      \psi_{\vtree}(\pi^{xx},\pi^{yy}) = x\,\pi^{xyy}- y\, \pi^{xxy}. 
\end{array}
\end{equation*}

\vspace{0.1cm} \noindent
These data determine $\delta_{\psi}$. 

\vspace{0.1cm} \noindent 
Arborescent Koszul-Tate resolutions are in general not minimal. While we will define the notion of minimality carefully in Appendix \ref{app:minimal}, we observe here simply the following. Applying the rules for the action of $\delta_\psi$, Equation \eqref{eq:deltaexpression}, we find in the present case for example: 
\begin{equation}
    \label{eq:contr.pairs}
\adjustbox{valign =c}{$\delta_\psi$} \left(
\adjustbox{valign =c}{
\scalebox{0.5}{
\begin{forest}
for tree = {grow' = 90}, nice empty nodes,
[
[
 [\scalebox{2}{$\pi^{xx}$}, tier =1]
 [\scalebox{2}{$\pi^{xy}$}, tier =1] 
]
 [\scalebox{2}{$\pi^{yy}$}, tier = 1]
 ]
]
\path[fill=black] (.parent anchor) circle[radius=4pt] (!1.child anchor) circle[radius=4pt];\end{forest}}}\right) \adjustbox{valign =c}{$=$} \adjustbox{valign =c}{\scalebox{1}{
\begin{forest}
for tree = {grow' = 90}, nice empty nodes,
[
 [\scalebox{1}{$\pi^{xx}$}, tier =1]
 [\scalebox{1}{$\pi^{xy}$}, tier =1]
]
\path[fill=black] (.parent anchor) circle[radius=2pt];
\end{forest}}}\adjustbox{valign =c}{\scalebox{1}{
\begin{forest}
for tree = {grow' = 90}, nice empty nodes,
[
 [\scalebox{1}{$\pi^{yy}$}, tier = 1]
]
\end{forest}}}\adjustbox{valign =c}{$-$}  \adjustbox{valign =c}{\scalebox{0.5}{
\begin{forest}
for tree = {grow' = 90}, nice empty nodes,
[
 [\scalebox{2}{$\pi^{xx}$}, tier =1]
 [\scalebox{2}{$\pi^{xy}$}, tier =1] 
 [\scalebox{2}{$\pi^{yy}$}, tier = 1]
]
\path[fill=black] (.parent anchor) circle[radius=4pt];
\end{forest}}} \adjustbox{valign =c}{$+ x$} \adjustbox{valign =c}{\scalebox{1}{
\begin{forest}
for tree = {grow' = 90}, nice empty nodes,
[
 [\scalebox{1}{$\pi^{xxy}$}, tier =1]
 [\scalebox{1}{$\pi^{yy}$}, tier =1]
]
\path[fill=black] (.parent anchor) circle[radius=2pt];
\end{forest}}}\adjustbox{valign =c}{$.$}
\end{equation}
\vspace{0.1cm} \noindent 
Since $\delta_\psi$ squares to zero, this yields
\begin{equation}
    \label{eq:redundanttree}
\adjustbox{valign =c}{$\delta_\psi$} \left(
\adjustbox{valign =c}{
\scalebox{0.5}{
\begin{forest}
for tree = {grow' = 90}, nice empty nodes,
[
 [\scalebox{2}{$\pi^{xx}$}, tier =1]
 [\scalebox{2}{$\pi^{xy}$}, tier =1] 
 [\scalebox{2}{$\pi^{yy}$}, tier = 1]
 ]
]
\path[fill=black] (.parent anchor) circle[radius=4pt];\end{forest}}}\right)
                \adjustbox{valign =c}{$=\delta_{\psi}$}\left( \adjustbox{valign =c}{\scalebox{1}{
\begin{forest}
for tree = {grow' = 90}, nice empty nodes,
[
 [\scalebox{1}{$\pi^{xx}$}, tier =1]
 [\scalebox{1}{$\pi^{xy}$}, tier =1]
]
\path[fill=black] (.parent anchor) circle[radius=2pt];
\end{forest}}}\adjustbox{valign =c}{\scalebox{1}{
\begin{forest}
for tree = {grow' = 90}, nice empty nodes,
[
 [\scalebox{1}{$\pi^{yy}$}, tier = 1]
]
\end{forest}}} \adjustbox{valign =c}{$+ x$} \adjustbox{valign =c}{\scalebox{1}{
\begin{forest}
for tree = {grow' = 90}, nice empty nodes,
[
 [\scalebox{1}{$\pi^{xxy}$}, tier =1]
 [\scalebox{1}{$\pi^{yy}$}, tier =1]
]
\path[fill=black] (.parent anchor) circle[radius=2pt];
\end{forest}}} \right)\adjustbox{valign =c}{$.$}
\end{equation}

\noindent
From this equation we learn that the $\delta_\psi$-cycle on the l.h.s., obtained by the application of $\delta_\psi$ to the tree $\wtree [\pi^{xx}, \pi^{xy},\pi^{yy}]$, is already $\delta_\psi$-exact with respect to the generators up to degree 4 omitting that one tree. This implies that, for what concerns exactness of the Koszul-Tate resolution, there is no reason for including the generator $\wtree [\pi^{xx}, \pi^{xy},\pi^{yy}]$. 

\vspace{1mm} \noindent
To consistently eliminate this generator---together with $\ltree [\pi^{xx}, \pi^{xy}, \pi^{yy}]$---, we may apply 
 Proposition \ref{prop:quotient}. Consider the following two elements of degree $ 5$ and $4$
$$
   \adjustbox{valign =c}{
\scalebox{0.5}{
\begin{forest}
for tree = {grow' = 90}, nice empty nodes,
[
[
 [\scalebox{2}{$\pi^{xx}$}, tier =1]
 [\scalebox{2}{$\pi^{xy}$}, tier =1] 
]
 [\scalebox{2}{$\pi^{yy}$}, tier = 1]
 ]
]
\path[fill=black] (.parent anchor) circle[radius=4pt]
                (!1.child anchor) circle[radius=4pt];
\end{forest}}}
 \hbox{and }\delta_\psi\left( 
\adjustbox{valign =c}{
\scalebox{0.5}{
\begin{forest}
for tree = {grow' = 90}, nice empty nodes,
[
[
 [\scalebox{2}{$\pi^{xx}$}, tier =1]
 [\scalebox{2}{$\pi^{xy}$}, tier =1] 
]
 [\scalebox{2}{$\pi^{yy}$}, tier = 1]
 ]
]
\path[fill=black] (.parent anchor) circle[radius=4pt]
                (!1.child anchor) circle[radius=4pt];\end{forest}}}\right),
   $$
   respectively. 
We construct a graded $\CO$-module $\mathfrak N_{\bullet}$ as follows:
$$
\mathfrak N_{i} = \begin{cases}
    0 \hbox{ if } i\leq 3,\\   \langle\delta_\psi(\ltree [\pi^{xx}, \pi^{xy}, \pi^{yy}]) \rangle \hbox { if } i =4,\\
    \langle \ltree [\pi^{xx}, \pi^{xy}, \pi^{yy}] \rangle\hbox{ if } i=5,\\ 
    \langle \hbox{decorated trees of degree $i$ containing a subtree $\in \r ((\mathfrak N_4 \oplus \mathfrak N_5)\odot S(\CT[\FMb]))$ } \rangle \hbox{ for $i\geq 6$}.
\end{cases}
$$ 
There exists graded free $\CO $-sub-modules $ \CEb \subset \CT[\FM_\bullet]$ such that 
 $  \mathfrak N_\bullet $ and $ \CE_\bullet   $ form free generators of the symmetric algebra $ S(\CT[\FMb])$. 
It is routine to check that the ideal $\mathfrak J $ generated by $\mathfrak N_\bullet $ satisfies all requirements of Proposition \ref{prop:quotient}.

\vspace{0.2cm}
\noindent
The Koszul-Tate resolution obtained on the quotient can be described as follows: it is the symmetric algebra generated by all decorated trees, where the following trees are omitted:
 $$  \adjustbox{valign =c}{\scalebox{0.5}{
\begin{forest}
for tree = {grow' = 90}, nice empty nodes,
[
[
 [\scalebox{2}{$\pi^{xx}$}, tier =1]
 [\scalebox{2}{$\pi^{xy}$}, tier =1] 
]
 [\scalebox{2}{$\pi^{yy}$}, tier = 1]
 ]
]
\path[fill=black] (.parent anchor) circle[radius=4pt]
                (!1.child anchor) circle[radius=4pt];
\end{forest}}} \hbox{ and }\adjustbox{valign =c}{\scalebox{0.5}{
\begin{forest}
for tree = {grow' = 90}, nice empty nodes,
[
 [\scalebox{2}{$\pi^{xx}$}, tier =1]
 [\scalebox{2}{$\pi^{xy}$}, tier =1] 
 [\scalebox{2}{$\pi^{yy}$}, tier = 1]
]
\path[fill=black] (.parent anchor) circle[radius=4pt];
\end{forest}}}  $$
together with any decorated tree admitting these trees as a subtree. 
The Koszul-Tate differential acts as before on the remaining trees with the following amendment: 
Every term of  $\delta_\psi t$ containing $\ltree[\pi^{xx},\pi^{xy},\pi^{yy}]$ is simply put to zero. However, if we obtain the tree $\wtree[\pi^{xx},\pi^{xy},\pi^{yy}]$, we perform the following substitution: 
$$
\adjustbox{valign =c}{
\scalebox{0.5}{
\begin{forest}
for tree = {grow' = 90}, nice empty nodes,
[
 [\scalebox{2}{$\pi^{xx}$}, tier =1]
 [\scalebox{2}{$\pi^{xy}$}, tier =1] 
 [\scalebox{2}{$\pi^{yy}$}, tier = 1]
 ]
]
\path[fill=black] (.parent anchor) circle[radius=4pt];\end{forest}}}
                \adjustbox{valign =c}{$\longrightarrow$} \adjustbox{valign =c}{\scalebox{1}{
\begin{forest}
for tree = {grow' = 90}, nice empty nodes,
[
 [\scalebox{1}{$\pi^{xx}$}, tier =1]
 [\scalebox{1}{$\pi^{xy}$}, tier =1]
]
\path[fill=black] (.parent anchor) circle[radius=2pt];
\end{forest}}}\adjustbox{valign =c}{\scalebox{1}{
\begin{forest}
for tree = {grow' = 90}, nice empty nodes,
[
 [\scalebox{1}{$\pi^{yy}$}, tier = 1]
]
\end{forest}}} \adjustbox{valign =c}{$+ x$} \adjustbox{valign =c}{\scalebox{1}{
\begin{forest}
for tree = {grow' = 90}, nice empty nodes,
[
 [\scalebox{1}{$\pi^{xxy}$}, tier =1]
 [\scalebox{1}{$\pi^{yy}$}, tier =1]
]
\path[fill=black] (.parent anchor) circle[radius=2pt];
\end{forest}}}\adjustbox{valign =c}{$.$}
$$
And likewise so if the $\wtree$ appears as a subtree, e.g.:
$$
\adjustbox{valign = c}{\scalebox{0.5}{\begin{forest}
for tree = {grow' = 90}, nice empty nodes, for tree={ inner sep=2 pt, s sep= 2 pt, fit=band, 
},
[
[
        [\scalebox{2}{$\pi^{xx}$}, tier =1] 
        [\scalebox{2}{$\pi^{xy}$}, tier =1]
    [\scalebox{2}{$\pi^{yy}$}, tier =1 ]
    ]
    [\scalebox{2}{$\pi^{xx}$}, tier =1]
]
\path[fill=black] (.parent anchor) circle[radius=4pt]
(!1.child anchor) circle[radius=4pt];
 \end{forest}}} \longrightarrow
\adjustbox{valign = c}{\scalebox{0.5}{\begin{forest}
for tree = {grow' = 90}, nice empty nodes, for tree={ inner sep=2 pt, s sep= 2 pt, fit=band, 
},
[
[
        [\scalebox{2}{$\pi^{xx}$}, tier =1] 
        [\scalebox{2}{$\pi^{xy}$}, tier =1]
    ]
    [\scalebox{2}{$\pi^{yy}$}, tier =1 ]
    [\scalebox{2}{$\pi^{xx}$}, tier =1]
]
\path[fill=black] (.parent anchor) circle[radius=4pt]
(!1.child anchor) circle[radius=4pt];
 \end{forest}}} + x
\adjustbox{valign = c}{\scalebox{0.5}{\begin{forest}
for tree = {grow' = 90}, nice empty nodes, for tree={ inner sep=2 pt, s sep= 2 pt, fit=band, 
},
[
[
        [\scalebox{2}{$\pi^{xxy}$}, tier =1] 
        [\scalebox{2}{$\pi^{yy}$}, tier =1]
    ]
    [\scalebox{2}{$\pi^{xx}$}, tier =1]
]
\path[fill=black] (.parent anchor) circle[radius=4pt]
(!1.child anchor) circle[radius=4pt];
 \end{forest}}}
$$
\end{example}

 \noindent
\begin{example}
\label{ex:4.2}
\normalfont
Let  $\CO$ be the polynomial algebra in the three variables $x, y, z$, and  $\CI\subset \CO $ the ideal generated by    $ x^2, xy, y^2, xz$.
 The ideal $\CI$ admits a free $\CO$-module resolution of length three:
    
       $$ \begin{tikzcd}
    0 \arrow[r ] &\underbrace{\FM_3}_{rk = 1}\arrow[r, "d" ] &\underbrace{\FM_2}_{rk = 4} \arrow[r, "d" ] &\underbrace{\FM_1}_{rk = 4} \arrow[r, "d" ] &\CO \arrow[r, ->>] & \CO/\CI
\end{tikzcd}
    $$
    \noindent
    We denote generators of $\FM_1 $ by $\pi^{xx},\pi^{xy}, \pi^{yz},\pi^{yy}  $, generators of degree $2$ by $\pi^{xyy},\pi^{yxx},\pi^{zxx},\pi^{zxy} $, and the generator of degree $3$ by $ \pi$.
    These choices of indices are natural in view of the differential $d$ that we now introduce:
    \begin{equation*}
\begin{array}{ll}
    \hbox{Degree $1$} & d \pi^{xx} = x^2, \qquad d \pi^{xy} = xy,\qquad d \pi^{xz} = xz, \qquad d \pi^{yy} = y^2, \\
    \hbox{Degree $2$} & d \pi^{xyy} = x \pi^{yy} - y \pi^{xy}, \qquad  d \pi^{yxx} = y \pi^{xx} - x \pi^{xy}, \\ \multicolumn{1}{c}{\hbox{}} &
    d \pi^{zxx} = z \pi^{xx} - x \pi^{xz}, \qquad d \pi^{zxy} = z \pi^{xy} - y \pi^{xz},\\
    \hbox{Degree $3$} & d\pi = z\pi^{yxx}- y \pi^{zxx} + x\pi^{zxy}.
\end{array}
\end{equation*}
 Since $\FM_i \neq 0$ for $i=1,2,3$ only, all maps $\psi_t$ in the arborescent Koszul-Tate resolution vanish except for $t=\vee$, the tree with one root and two leaves, and $t=\wtree$, the tree with one root and three leaves.
 Those can be computed explicitly:
   \begin{equation*}
\begin{array} {ll}
      \\ \psi_{\vtree} (\pi^{xx},\pi^{xy}) = -x\, \pi^{yxx}, &
       \psi_{\vtree}(\pi^{xx},\pi^{xz}) = -x\,\pi^{zxx} \\
      \psi_{\vtree}(\pi^{xx},\pi^{yy}) = x\,\pi^{xyy}- y\, \pi^{yxx}, &    \psi_{\vtree}(\pi^{xy},\pi^{xz}) = - x\,\pi^{zxy} \\
       \psi_{\vtree}(\pi^{xz},\pi^{yy}) =z\,\pi^{xyy} + y\,\pi^{zxy}, &     \psi_{\vtree}(\pi^{xy},\pi^{yy}) = y \,\pi^{xyy}
\\& \\
      \psi_{\vtree}(\pi^{xx},\pi^{zxy}) = x\,\pi, &    \psi_{\vtree}(\pi^{xy},\pi^{zxx}) = -x\,\pi\\
     \psi_{\vtree}(\pi^{xz},\pi^{yxx}) = x\,\pi, &    \psi_{\vtree}(\pi^{yy},\pi^{zxx}) = -y\,\pi \\ & 
     \\\psi_{\wtree}(\pi^{xx},\pi^{xy},\pi^{xz}) = -x^2\,\pi, &   \psi_{\wtree}(\pi^{xx},\pi^{xz},\pi^{yy}) = xy\,\pi.\\\\
\end{array}
\end{equation*}
 In this example, the product $\mu_2 $ is associative, which is obvious also for degree reasons, and the higher products vanish, $\mu_n=0 $ for all $n \geq 3 $.
\end{example}


\vspace{.5cm}
\noindent
\begin{example}
\normalfont
\label{ex:Katthan}
 Let $\CO$ be  the polynomial algebra in the four variables $x, y, z,w $  and $\CI\subset \CO $ the ideal generated by $x^2, xy, y^2z^2, zw, w^2$. 
 This example was studied in \cite{Katthan2019}, where it was shown that the resolution does not admit a differential graded algebra structure.
 A free $\CO$-module resolution of  length four is given by:
       $$ \begin{tikzcd}
    0 \arrow[r ] &\underbrace{\FM_4}_{rk = 1}\arrow[r, "d" ]&\underbrace{\FM_3}_{rk = 5}\arrow[r, "d" ] &\underbrace{\FM_2}_{rk = 8} \arrow[r, "d" ] &\underbrace{\FM_1}_{rk = 5} \arrow[r, "d" ] &\CO \arrow[r, ->>] & \CO/\CI
\end{tikzcd}
    $$
    
\vspace{0.1cm}
\noindent
The generators and the differential $d$ are given by
\begin{equation*}
    \begin{array}{ll}
        \hbox{Degree 1} & d\pi^a= x^2, \quad d\pi^b = xy, \quad d\pi^c = y^2z^2, \quad d\pi^d = wz, \quad d\pi^e = w^2, \\
        \hbox{Degree 2} & d\pi^{ab}= x\pi^b - y\pi^a, \quad  d\pi^{ad}= x^2 \pi^d - wz\pi^a, \quad d\pi^{ae} = x^2 \pi^e - w^2 \pi^a,\\
        \hbox{} & d\pi^{bc}= x\pi^c - yz^2\pi^b, \quad d\pi^{bd}= xy\pi^d - wz\pi^b, \quad d\pi^{be} = xy \pi^e- w^2\pi^b,\\
        \hbox{} & d\pi^{cd}= y^2z\pi^d- w\pi^c, \quad d\pi^{de} = z\pi^e - w\pi^d,\\
        \hbox{Degree 3} &  d \pi^{abd}= wz\pi^{ab}+x\pi^{bd}-y\pi^{ad}, \quad d\pi^{abe} = w^2\pi^{ab}+x\pi^{be}-y\pi^{ae}, \\ \hbox{} & 
 d\pi^{ade} = x^2\pi^{de} + w \pi^{ad}- z\pi^{ae}, d\pi^{bde}= xy\pi^{de} + w\pi^{bd} - z\pi^{be}, \\
        \hbox{} & d\pi^{bcd} = w\pi^{bc} + x\pi^{cd} - yz\pi^{bd},\\
        \hbox{Degree 4} & d\pi^{abde} = -w\pi^{abd} + z\pi^{abe}- y\pi^{ade}+ x\pi^{bde}.\\
    \end{array}
\end{equation*}
By a direct calculation the following complete list of non-vanishing arborescent operations: 
 \begin{equation*}
\begin{array} {ll}
       \psi_{\vtree} (\pi^{a},\pi^{b}) = x\, \pi^{ab}, &
       \psi_{\vtree}(\pi^{a},\pi^{c}) = yz^2\,\pi^{ab} +  x\, \pi^{bc}\\
      \psi_{\vtree}(\pi^{a},\pi^{d}) = \pi^{ad}, &    \psi_{\vtree}(\pi^{a},\pi^{e}) = \pi^{ae} \\
       \psi_{\vtree}(\pi^{b},\pi^{c}) =y\,\pi^{bc}, &     \psi_{\vtree}(\pi^{b},\pi^{d}) = \pi^{bd}\\
      \psi_{\vtree}(\pi^{b},\pi^{e}) = \pi^{be}, &    \psi_{\vtree}(\pi^{c},\pi^{d}) = z\,\pi^{cd}\\
     \psi_{\vtree}(\pi^{c},\pi^{e}) = w\,\pi^{cd}+y^2z\,\pi^{de} , &    \psi_{\vtree}(\pi^{d},\pi^{e}) = w\,\pi^{de} 
\end{array}
\end{equation*}
\begin{equation*}
\begin{array} {ll}
\psi_{\vtree} (\pi^{ab},\pi^{d}) = \pi^{abd}, &
       \psi_{\vtree}(\pi^{ab},\pi^{e}) = \pi^{abe}\\
      \psi_{\vtree}(\pi^{ad},\pi^{b}) = -x\, \pi^{abd}, &    \psi_{\vtree}(\pi^{ad},\pi^{c}) = -yz^2\,\pi^{abd} - xz\, \pi^{bcd} \\
       \psi_{\vtree}(\pi^{ad},\pi^{e}) =w\,\pi^{ade}, &     \psi_{\vtree}(\pi^{ae},\pi^{b}) = -x\,\pi^{abe}\\
      \psi_{\vtree}(\pi^{ae},\pi^{c}) = -wx\,\pi^{bcd} - yz^2\, \pi^{abe} - xyz\, \pi^{bde}, &    \psi_{\vtree}(\pi^{ae},\pi^{d}) = -w\,\pi^{ade}\\
     \psi_{\vtree}(\pi^{bc},\pi^{d}) = z\,\pi^{bcd}, &    \psi_{\vtree}(\pi^{bc},\pi^{e}) = w\,\pi^{bcd}+ yz\,\pi^{bde}\\
     
     \psi_{\vtree} (\pi^{bd},\pi^{a}) = x\, \pi^{abd}, &
       \psi_{\vtree}(\pi^{bd},\pi^{c}) = -yz\, \pi^{bcd}\\
      \psi_{\vtree}(\pi^{bd},\pi^{e}) = w\, \pi^{bde}, &    \psi_{\vtree}(\pi^{be},\pi^{a}) = x\,\pi^{abe} \\
       \psi_{\vtree}(\pi^{be},\pi^{c}) =-wy\,\pi^{bcd} - y^2z\,\pi^{bde}, &     \psi_{\vtree}(\pi^{be},\pi^{d}) = -w\,\pi^{abe}\\
      \psi_{\vtree}(\pi^{cd},\pi^{a}) = yz\,\pi^{abd} +x \, \pi^{bcd}, &    \psi_{\vtree}(\pi^{cd},\pi^{b}) = y\,\pi^{bcd}\\
   \psi_{\vtree}(\pi^{de},\pi^{a}) = \pi^{ade} ,& \psi_{\vtree}(\pi^{de},\pi^{b}) = \pi^{bde}\\
     \psi_{\vtree} (\pi^{ab},\pi^{de}) = \pi^{abde}, &
       \psi_{\vtree}(\pi^{ad},\pi^{be}) = -wx\, \pi^{abde}\\
      \psi_{\vtree}(\pi^{ae},\pi^{bd}) = wx\, \pi^{abde}, &    \psi_{\vtree}(\pi^{abd},\pi^{e}) = w\,\pi^{abde} \\
       \psi_{\vtree}(\pi^{abe},\pi^{d}) =-w\,\pi^{abde}, &     \psi_{\vtree}(\pi^{ade},\pi^{b}) = x\,\pi^{abde}\\
      \psi_{\vtree}(\pi^{ade},\pi^{c}) = yz^2\,\pi^{abde} &    \psi_{\vtree}(\pi^{bde},\pi^{a}) = -x\,\pi^{abde}\\
   & \\
   \psi_{\wtree}(\pi^a, \pi^b, \pi^d) = x\,\pi^{abd}, & \psi_{\wtree}(\pi^a, \pi^b, \pi^e) = x\,\pi^{abe}\\
   \psi_{\wtree}(\pi^a, \pi^c, \pi^d) = yz^2\,\pi^{abd}+ xz\, \pi^{bcd}, & \psi_{\wtree}(\pi^a, \pi^c, \pi^e) = yz^2\,\pi^{abe}+ wx\,\pi^{bcd}+ xyz\,\pi^{bde}\\
   \psi_{\wtree}(\pi^a, \pi^d, \pi^e) = w\,\pi^{ade}, & \psi_{\wtree}(\pi^b, \pi^c, \pi^d) = yz\,\pi^{bcd}\\
   \psi_{\wtree}(\pi^b, \pi^c, \pi^e) = wy\,\pi^{bcd}+ yz^2\,\pi^{bde}, & \psi_{\wtree}(\pi^b, \pi^d, \pi^e) = w\,\pi^{bbe}\\ & \\
   \psi_{\wtree}(\pi^{ab}, \pi^d, \pi^e) = w\,\pi^{abde}, & \psi_{\wtree}(\pi^{ad}, \pi^b, \pi^e) = -wx\,\pi^{abde}\\
   \psi_{\wtree}(\pi^{ad}, \pi^c, \pi^e) = -wyz^2\,\pi^{abde}, & \psi_{\wtree}(\pi^{bd}, \pi^a, \pi^e) = wx\,\pi^{abde}\\
   \psi_{\wtree}(\pi^{cd}, \pi^a, \pi^e) = -wyz\,\pi^{abde}, & \psi_{\wtree}(\pi^{de}, \pi^a, \pi^b) = x\,\pi^{abde}\\
   \psi_{\wtree}(\pi^{de}, \pi^a, \pi^c) = yz^2\,\pi^{abde}, & \psi_{\rtree}(\pi^{a}, \pi^e, \pi^c) = yz\,\pi^{abde}\\
   \psi_{\fourtree}(\pi^a, \pi^b, \pi^d, \pi^e) = wx\pi^{abde}, & \psi_{\fourtree}(\pi^a, \pi^c, \pi^d, \pi^e) = wyz^2 \pi^{abde}. 
\end{array}
\end{equation*}
\vspace{1mm}

\noindent
One can then check that the associator of the bilinear operation $\psi_\vee $ is zero except in the following cases:
\begin{equation}
\label{eq:abcd}
\begin{array}{rcccl}
 \psi_{\vtree}(\pi^{a}, \psi_{\vtree}(\pi^e, \pi^c)) - \psi_{\vtree}(\psi_{\vtree}(\pi^a, \pi^e), \pi^c)&=& d\circ (\psi_{\ltree} +\psi_{ \rtree})(\pi^a, \pi^e, \pi^c) &=&  \, \, \,  yz \, d\pi^{abde}, \\
 \psi_{\vtree}(\pi^{a}, \psi_{\vtree}(\pi^c, \pi^e)) - \psi_{\vtree}(\psi_{\vtree}(\pi^a, \pi^c), \pi^e)&=& d\circ (\psi_{\ltree} +\psi_{ \rtree})(\pi^a, \pi^c, \pi^e) &=& \! \! -yz \, d\pi^{abde}.
\end{array}
\end{equation}
\vspace{1mm}
\noindent
This explain the occurrence of the exceptional term $
\psi_{\rtree }(\pi^a, \pi^c, \pi^e) = -yz \pi^{abde}$.
For degree reasons, there are no non-zero arborescent operations corresponding to trees with $\geq 3$ vertices, and we obtain a complete description of the arborescent Koszul-Tate operations.
The corresponding
arborescent $A_\infty $-algebra products  admits a non-zero term $\mu_3 $ while all  $ \mu_n$ are zero for $n \geq 4 $.
\end{example}

\subsection{Resolutions having a DGCA structure}
\label{sec:DGCA}

\subsubsection{An explicit arborescent Koszul-Tate resolution}

Several $ \CO$-module resolutions $(\FM_\bullet,d)$ of $\CO/\CI $  admit a structure of a differential graded commutative algebra  (DGCA). Still, they may not be Koszul-Tate resolutions, because they may not be graded symmetric algebras. However,  such $ \CO$-module resolutions $(\FM_\bullet,d)$ allow the construction of an explicit arborescent Koszul-Tate resolution in quite a straightforward manner.
Before stating our result, let us introduce a convention: for all $n \geq 2$, denote by $t_n$ the tree with $n$ leaves and one root:
  \begin{equation}
  \label{eq:onerootnleaves}\begin{array}{c}t_n = \\ \vspace{.4cm} \end{array} \overbrace{
  \begin{forest}
for tree = {grow' = 90}, nice empty nodes,
[
 [, tier =1]
 [, tier =1]
 [,edge= dashed,  tier =1]
 [,edge= dashed,label={[mystyle2] {$...$}}, tier = 1]
]
\path[fill=black] (.parent anchor) circle[radius=2pt];
\end{forest}}^{\hbox{$n$ leaves}}
 . \end{equation}
 Notice that $t_2$ and $ t_3$ already appeared as $ \vee$  and $ \scalebox{1.5}{\wtree}$ in the expressions $ \psi_\vee$ and $ \psi_{\wtree}$.

\begin{proposition}
\label{prop:psi.dga}
    Let $(\FM_{\bullet}, d)$ be a free $\CO$-module resolution of $ \CO/\CI$ such that $ \FM_\bullet \oplus \CO$ is equipped with a graded commutative associative multiplication $\star$ turning it into a DGCA. Here the multiplication $\star$ is assumed to extend the given $\CO$-module structure: 
    $F * a =a * F=Fa $ for all $ F \in \CO$, $a\in \FM_\bullet \oplus \CO$. Choose the arborescent operations $\psi \colon t \mapsto \psi_t$  as follows
 $$ \psi_{t_n} \left(    a_1 \otimes \dots \otimes a_n\right) = a_1 \star \cdots  \star a_n \hbox{ for all $ a_1, \dots, a_n \in \FM_\bullet$}$$
 and  $\psi_t=0 $  for any tree not of the previous form.
 Then the pair $(S(\CT[\FM_{\bullet}]),\delta_\psi)$ 
 is an arborescent Koszul-Tate resolution of $ \CO/\CI$.
\end{proposition}
\begin{proof}
It suffices to check Equation \eqref{eq:deltaPsiSquare} for \emph{i)} the trees $t_n$ with $n \in \mathbb N$ and  for \emph{ii)} the trees obtained by plugging $ t_n$ on some leaf of $ t_m$ for  $ n,m \geq 2$. 
\emph{i)} holds true because $d$ is a derivation of $*$ and \emph{ii)} holds true because $*$ is associative.
\end{proof} 
\vspace{1mm}
\noindent We now have two DGCAs: $(\FM_\bullet \oplus \CO,\star, d) $ and $(S(\CT[\FM_{\bullet}]),\odot,\delta_\psi)$. Proposition \ref{prop:Homotopy} provides a projection ${\mathrm{Proj}}_\psi $ from the second complex to the first complex which is a chain map.

\begin{proposition}
\label{prop:backtoKoszul}
For an arborescent Koszul-Tate resolution as in Proposition \ref{prop:psi.dga}, the projection $${\mathrm{Proj}}_\psi \colon S(\CT[\FM_{\bullet}]) \longrightarrow \FM_\bullet \oplus \CO$$ is a morphism of differential graded  commutative algebras.    
\end{proposition}
\begin{proof}
Let us check that ${\mathrm{Proj}}_\psi $  is a  morphism of graded commutative algebras,
i.e.
\begin{equation} \label{eq:morphismneeded} {\mathrm{Proj}}_\psi( T_1  \odot  \cdots \odot T_k) =  {\mathrm{Proj}}_\psi( T_1  ) \star \cdots \star   {\mathrm{Proj}}_\psi( T_k  ) \end{equation}
for all $ T_1, \dots, T_k$ in $\CT[\FM_{\bullet}] $.
If  $ T_1=|\otimesk a_1 $, $\ldots$,  $T_k = |\otimesk a_k$ are all trivial trees, then, by the definition of the projection in Equation \eqref{eq:matricesInclProj}:
$$ {\mathrm{Proj}}_\psi \left(\begin{array}{c} \\ \scalebox{0.5}{\begin{forest}
for tree = {grow' = 90}, nice empty nodes,
[
 [\scalebox{2}{$a_1$}, tier =1] 
]
\end{forest}} \begin{array}{c}  \odot \dots \odot \\ \vspace{.2cm}  \end{array} \scalebox{0.5}{\begin{forest}
for tree = {grow' = 90}, nice empty nodes,
[
 [\scalebox{2}{$a_k$}, tier =1] 
]
\end{forest}} \\ \end{array}\right) \begin{array}{c} \\= \psi_{t_k}(a_1,\cdots,a_k)=  \\ \\ \end{array}  \begin{array}{c} \\  a_1\star \cdots \star a_k \, .\vspace{.5cm} \end{array}$$
If, on the other hand, one of them is not a trivial tree, 
 then ${\mathrm{Proj}}_\psi( T_1 \odot \dots \odot T_k)=(\psi \circ \r) (T_1 \odot \cdots  \odot T_k)=0$ since $ \r (T_1 \odot \cdots  \odot T_k)$ is not a tree of the form of Equation \eqref{eq:onerootnleaves}.

\end{proof}

\begin{remark}
    \normalfont
Note that the kernel of ${\mathrm{Proj}}_\psi $ satisfies the requirements on the ideal $\mathfrak J $ in Proposition~\ref{prop:quotient}.
\end{remark}

\vspace{0.1cm}

\noindent
Examples \ref{ex:4.1} and \ref{ex:4.2} in Section \ref{sec:basicExamples} above provided two examples of  $ \CO$-module resolutions coming with a graded commutative associative product. We conclude this discussion here with a last and more exotic example.

\begin{example}
\normalfont
 Consider $ \CO= \mathbb K[x,y]/\langle xy \rangle$ with the two generators $\overline{x}$ and $\overline{y} $. Let $\CI $ be the ideal generated by $\overline{x}$. 
Let $\FM_i  $ be the free $ \CO$-module with one generator $ e_i$ for all $i \geq 1$.
 Then $ \FM_\bullet$ becomes an $ \CO$-module resolution by choosing the boundary map $d$ according to
 $$d(e_{i} )=\left\{\begin{array}{ll} \overline{y} \, e_{i-1} & \hbox{ for $i$ even} \\  \overline{x} \, e_{i-1} & \hbox{ for $i$ odd.} \end{array}\right.  $$
 $(\FM_\bullet \oplus \CO,d) $ is turned into a DGCA by equipping it with the product $\star$ defined by means of 
 $ e_i \star e_j = \mu_{i,j} \, e_{i+j}$ with 
 $$ \mu_{i,j} = 
 \left\{\begin{array}{cl} 0 & \hbox{if $i,j $ odd}  \\[10pt]
 \begin{pmatrix} \tfrac{1}{2}(i+j) \\ \tfrac{1}{2}j \end{pmatrix} & \hbox{ if $i,j $ are even} \\[10pt] 
 \begin{pmatrix} \tfrac{1}{2}(i+j-1) \\ \tfrac{1}{2}(j-1) \end{pmatrix}& \hbox{ if $i $ is even and $j$ is odd} \end{array}\right. 
$$
where $\begin{pmatrix} n \\ k   
\end{pmatrix}$ denotes  binomial coefficients. 
This can be verified by a direct calculation.

\vspace{2mm}
\noindent
Let us remark that this DGCA is a Koszul-Tate resolution, since it is isomorphic, as a graded algebra, to the symmetric algebra of the free $ \CO$-module with one generator of  degree $1$ and one of  degree~$2$. 

\end{example}

\subsubsection{Koszul complex}
\label{sec:Koszul}

Assume that an ideal $\CI \subset \CO$ is generated by $r$ elements $ \varphi_1, \dots, \varphi_r$. 
Let $V$ be a vector space of dimension $k$ with basis $e_1, \dots, e_k $.
The pair $(K_\bullet(k),\delta)$ given by $K_\bullet(k) := \mathcal O \otimes_{\mathbb K} \wedge V^* $,  and $ \delta = \sum_{i=1}^r \varphi_i \otimes \mathfrak i_{e_i}  $ is a differential graded commutative algebra called \emph{Koszul complex}. 
The following result, attributed to Jean-Louis Koszul, is well-known.
Recall that  $\varphi$  is said to be a \emph{regular sequence} if  $\phi_1 \hbox{ is not a zero divisor in } \CO$, and for all $i\in \{2, \dots, k\}$, $\phi_i $ is not a zero divisor in $ \CO/\langle \phi_1, \dots \phi_{i-1} \rangle $.

\begin{proposition} \cite{Eisenbud,zbMATH00043569} Let $\CO$ be a Noetherian ring. The Koszul complex $(K_\bullet(k),\delta)$ is a Koszul-Tate resolution of $\CO/\CI $ if and only if  $\varphi_1, \dots, \varphi_k \in \CO $ is  a regular sequence.
\end{proposition}

\begin{example}
\normalfont
\label{ex:K1}
Consider $\CI$ an ideal generated by a non-zero divisor $ \varphi_1 \in \CO$. Then $\varphi_1 $ is a regular sequence of length one.  As an $ \CO$-module, $K_\bullet(1) $  has therefore two generators: the first one is the unit $1 $ and we denote the second one by $\theta $.
The differential is given by $ \delta(1)=0$ and $ \delta (\theta)=\varphi_1$.
\end{example}

\begin{example}
\normalfont
Consider a polynomial $\psi \in \mathbb C[x_1, \dots, x_n]$ which is homogeneous of degree $> 0 $. If $\psi $ has an isolated singularity at $0$, then  it is a classical result \cite{Shafarevich} that the sequence
 $ \frac{\partial \psi}{\partial x_1}, \dots,  \frac{\psi}{\partial x_n} $ is a regular sequence. Its Koszul resolution is canonically isomorphic to the graded algebra of polynomial polyvector fields on $\mathbb C^n$ equipped with the contraction by the exact $1$-form $d\psi $.
\end{example}

\vspace{0.1cm}

\noindent
Since the Koszul resolution is an $\CO $-module resolution, one can construct its arborescent Kozul-Tate resolution as in Proposition \ref{prop:psi.dga}. 

\begin{example}
\normalfont
\label{ex:K1_2} 
The Koszul resolution of Example \ref{ex:K1} is an $ \CO$-module resolution for which $\FM_i=0 $ for all $ i \geq 2$ while $ \FM_1$ has one generator $ \theta$.
Since this generator $\theta $ is odd,  for every decorated tree $t$ with $n \geq 2 $ leaves, $t[ \theta , \dots, \theta ]$ has a class in the quotient space  $ \CT[\FM_\bullet]$ which is equal to $0$.  As a consequence, the whole construction of Section \ref{sec:arborescent} becomes trivial, and the arborescent Koszul-Tate resolution coincides with the initial Koszul resolution. 
\end{example}

\vspace{0.1cm}

\noindent
Example \ref{ex:K1_2} is a very particular case. It follows from Proposition \ref{prop:backtoKoszul} that the construction of Section \ref{sec:arborescent}, applied to the  $ \CO$-module resolution given by the Koszul resolution $(K_\bullet(k), \delta) $ of a regular sequence  of length  $k \geq 2$ yields an  arborescent Koszul-Tate resolution that will be much larger than $(K_\bullet(k), \delta) $.  
For Koszul complexes, there is no gain in considering arborescent Kozul-Tate resolution, even on the contrary, it makes things more complicated there. However, in the following we provide examples for which our construction is very useful: one is given an $\CO$-resolution that happens to be a DGCA but not a Koszul-Tate resolution and obtain an aborescent Koszul-Tate resolution by the use of Proposition \ref{prop:psi.dga}.

\subsubsection{Taylor resolutions of monomial ideals}
\label{sec:taylor.monomials}

\vspace{0.1cm}

\noindent
Let $\CO$ be the polynomial ring $\CO =\mathbb K[x_1, \dots, x_n]$. An ideal $\CI \subset \CO $ is said to be a {\emph{monomial ideal}} if it admits a set  $\varphi = \lbrace \varphi_1, \dots, \varphi_k \rbrace$ of monomial generators. Then $ \CO/\CI$  admits resolutions equipped with a differential graded commutative algebra structure,
the so-called \emph{Taylor resolutions}, that we now describe  following \cite{peeva2010graded}.

\vspace{1mm}
\noindent
\begin{construction}[Taylor resolutions] 
\normalfont
The \emph{Taylor complex} is the chain complex 
\[
    \begin{tikzcd}[row sep=2.5em, column sep=2.5em]
   \cdots \arrow[r, "d"] & C_i(k) \arrow[r, "d"] & \cdots \arrow[r, "d"] & \mathcal C_0(k) 
  \end{tikzcd}
    \]
    of free $\CO$-modules $C_i(\varphi)$  constructed as follows:
\begin{itemize}
    \item The graded vector space $C_\bullet(k) $ only depends on the number $k$ of monomials generating $ \CI$.
Let $C_\bullet $ be the free $ \CO$-module  generated by elements $e_\tau $ indexed by subsets  $ \tau \subset  \{1, \dots, k\} $:
 $$ C_\bullet(k) = \oplus_{\tau \subset \{1, \dots, k\}}  \CO \, e_{\tau} .$$
We equip $C_\bullet $ with a grading by imposing the degree of $ e_{\tau}$ to be the cardinality of $ \tau$.  For instance, $ C_0(k) = \CO \, e_{{\emptyset}} \simeq \CO$ while $C_1(k)= \oplus_{i=1}^k \CO \, e_{\{i\}} $.
    \item For every subset $\tau$ of $ \{1, \dots,k\}$, let $m_\tau $ be
    the least common multiple of the monomials $ \varphi_i$ whose index $i$ is in $ \tau$. Explicitly, for $ \tau= \{i_1,\dots, i_m\}  \subset \{1, \dots,k\}$, we have $m_\tau = \mathrm{lcm} \left( \varphi_{i_1}, \ldots,  \varphi_{i_m}  \right)$.  
   Then the differential $d$ is defined by:
    $$
    d(e_{\tau}) = \sum_{j \in \tau} \, \mathrm{sign}(j, \tau) \, \frac{m_\tau}{m_{\tau \backslash \{j\}}} \, e_{\tau \backslash \{j\}}
    $$
    where $\mathrm{sign}(j, \tau)$ is $+1$ if $ \tau$ contains an even number of elements strictly smaller than $j$ and $ -1$ otherwise. 
\end{itemize}
\end{construction}

    \noindent
     The complex $(C_\bullet(k),d)$ is an $\CO$-module resolution of $\CO/\CI$, see e.g. Theorem 26.7 in \cite{peeva2010graded} (the minimality condition on the generators imposed there can be dropped). This complex is called the \emph{Taylor resolution} of $ \CO/\CI$.

\begin{remark}
\normalfont
Taylor resolutions need not be minimal $\CO $-resolutions in general. For instance, for the ideal of $\mathbb C[x,y]$ generated by $ x^2,xy,y^2$, the Taylor resolution has length three, while the $ \CO$-resolution in Section \ref{sec:basicExamples} case I has length two.
\end{remark}

\vspace{0.1cm}

\noindent
As explained in \cite{Scarf},
    the Taylor resolution is equipped with an associative and graded-commutative multiplication $\star$: for every subsets $ \sigma, \tau \subset \{1, \dots, k\} $ we pose
    \begin{equation}
        \label{eq:mult.taylor}
       e_{\sigma}\star e_{\tau} =  \left\{ \begin{array}{ll}
     \mathrm{sign}(\sigma, \tau) \; \dfrac{\, m_{\sigma} \, m_{\tau}\, }{m_{\sigma\cup \tau}} \; \:e_{\sigma\cup \tau},  & 
     \hbox{if $\sigma \cap \tau= \varnothing $} \\[10pt] 0 & 
       \hbox{if $\sigma\cap \tau \neq \varnothing $.}
       \end{array}\right.
    \end{equation}  
    Here $\mathrm{sign}(\sigma , \tau)=(-1)^N$ where $N$ is the cardinality of pairs $ (i,j)$ with $i \in \sigma, j \in \tau$ such that $ i > j $. It is routine to check that this multiplication is compatible with $d$  as defined above, turning 
the Taylor resolution of  $ \CO/\CI$ into a DGCA.

    \begin{remark}
    \normalfont
If $\varphi_1, \dots, \varphi_k $ is a regular sequence of monomials, then the Taylor resolution  coincides with the Koszul complex $(K_\bullet(k), \delta)$ of $\CO/\CI$ as DGCAs.
    \end{remark}

\vspace{1mm}
\noindent
In general, Taylor resolutions are \emph{not} Koszul-Tate resolutions, because they may  not be symmetric algebras.
Being free $\CO$-module resolutions, Proposition  \ref{prop:psi.dga} allows us to construct an arborescent Koszul-Tate resolution out of them. 
 As far as we know,  this is the first explicit construction of a Koszul-Tate resolution which works for any monomial ideal even if the Koszul complex does not provide a resolution.

\section{Complexity of Koszul-Tate resolutions}
\label{sec:complexity}
\subsection{Counting computations needed for constructing  Koszul-Tate resolutions}

 Arborescent Koszul-Tate resolutions are explicit in the sense that finitely many computations often suffice to construct it. 
 This section intends to make this statement precise and rigorous.

\vspace{1mm}
\noindent
Let us consider the following two classes of problems.
\begin{itemize}
    \item "\emph{$\CO $-linear problem of the first type.}" Given free or projective $\CO$-modules $\mathcal N_1 $ and $ \mathcal N_2 $ and an $ \CO$-linear map $d_1 \colon \mathcal N_1 \to \mathcal N_2 $, find a free or projective $\CO$-module $ \mathcal N_0$ and an $ \CO$-linear map $ d_0 \colon \mathcal N_0 \to \mathcal N_1 $ such that $ {\mathrm{im}}(d_0)={\mathrm{ker}}(d_1) $.
    \item "\emph{$\CO $-linear problem of the second type.}" Given three $ \CO$-modules $\mathcal N_0,\mathcal N_1, \mathcal N_2 $ with $ \mathcal N_0$ free or projective, and given 
    $ \chi_1\colon \mathcal N_0 \to \mathcal N_2$ and $ \chi_2 \colon \mathcal N_1 \to \mathcal N_2$ $\CO $-linear maps with $ \chi_2$ surjective, 
    construct an $ \CO$-linear map $\psi $ making the following diagram commutative:
     $$ \begin{tikzcd} \mathcal N_0 \arrow[r, dotted,"\psi"]  \ar[dr,"\chi_1"']& \mathcal N_1 \ar[d,->>,"\chi_2"]\\ & \mathcal N_2  \end{tikzcd} .$$     
\end{itemize}
For both problems above a solution exists. 

\vspace{1mm}
\noindent
Recall that an $ \CO$-module resolution $ (\FM_\bullet, d)$ of $\CO/\CI $ is said to have \emph{finite length} if $\exists N \in \mathbb N$ s.t. $ \forall i > N$, $ \FM_i=0$.

 \begin{theorem}
 \label{th:finitelyMany}
 For every  $ \CO$-resolution $(\FM_\bullet, d)$ of $\CO/\CI $ of finite length,  its arborescent Koszul-Tate resolutions can be computed by solving finitely many $\CO$-linear problems of the second type.
 \end{theorem}
 \begin{proof} 
 The recursion describing the construction of the arborescent operations $\psi $ such that $ \delta_\psi=0$ is given in Equation \eqref{eq:psi3}. It consists in solving an $ \CO$-linear problem  of the second type. 
 Let $N$ be the length of $(\FM_\bullet, d) $. For degree reasons, $ \psi_t(a_1, \dots, a_n)$ is zero if $t[a_1, \dots, a_n]$ is of degree greater than $N+1$. There are therefore finitely many trees for which $\psi_t \neq 0 $.
 One therefore needs to solve finitely many $ \CO$-linear problems of the second type only.
 \end{proof}
 
\vspace{1mm}
\noindent
We say that a Noetherian  algebra $\CO$ is a \emph{Syzygy algebra} if every finitely generated free $\CO$-module has a  finite length $\CO $-module resolution by finitely generated free   $\CO $-modules.
For example, by Hilbert's Syzygy theorem and its generalizations \cite{Eisenbud}, polynomial rings in $N$ variables, germs at a point  of holomorphic or real analytic functions in $N$ variables, or formal power series in $N$ variables are all Syzygy algebras \cite{Tougeron1976}. Moreover every finitely generated module over the previous Syzygy algebras admit free $ \CO$-module resolutions of length less than or equal to~$N$.  

\begin{corollary}
\label{cor:finmanylinear} 
Let $ \CO$ be a Syzygy algebra.
For every ideal $\CI \subset \CO$, one can construct an arborescent Koszul-Tate resolution of $ \CO/\CI$  by solving a finite number of $\mathcal O$-linear problems involving free $ \CO$-modules of finite ranks only.
\end{corollary}

\begin{proof}
The $n$ first terms of a free $ \CO$-module resolution of $\CO/\CI $ can be computed for all $n$ by solving finitely many $\mathcal O$-linear problems of the first type. Since $\CO$ is a Syzygy algebra, this free $\CO$-module resolution $(\FM_\bullet, d) $  can be chosen of finite length. The result then follows from Theorem~\ref{th:finitelyMany}.
\end{proof}

\noindent
If one restricts further to the polynomial algebra,  $ \CO=\mathbb K[x_1, \dots,x_n]$, one can obtain even a much stronger statement  than
Corollary \ref{cor:finmanylinear}.
$\CO $-linear problems of the first or the second type, with all $ \CO$-modules involved being of finite ranks, can be solved by an algorithm that uses only finitely many algebraic operations  $ +, - ,\times, / $. This follows from the theory of Gr\"obner basis and Schreyer’s theorem, see e.g. Theorems 2.4.2 and 2.4.4 in \cite{zbMATH00638938}, or Chapter 15 in \cite{Eisenbud}. 

\begin{corollary}
\label{cor:finmanyoperations}
For any ideal $\CI $ of $\CO=\mathbb K[x_1, \dots,x_n]$, an arborescent Koszul-Tate resolution of $ \CO/\CI$ can be constructed through an algorithm that uses finitely many algebraic operations.
\end{corollary}

\vspace{.4cm}

\noindent
In contrast, the Tate algorithm requires solving infinitely many $\CO $-linear problems of the first type, even for a Syzygy or polynomial algebra. The next proposition and the discussion below make this statement precise. The ranks  of free $\CO $-modules $\CEb=\oplus_{i \geq 1} \CE_i$ entering a Koszul-Tate resolution of $\CO/\CI $ can not be arbitrarily small.
In Appendix \ref{app:reduced}, we give a lower bound to these ranks. 
It goes through the notion of a reduced complex $(\CEb\otimes \CO/\CI, \underline \delta)$ of $\CO/\CI $, which is a complex of free or projective $\CO/\CI $-modules.
We show that it is unique up to homotopy equivalence. Therefore, its
homology
does not depend on the choice of a Koszul-Tate resolution.
The same holds for its homology of the complex of $\mathbb K $-vector spaces obtained by tensoring by $ \mathbb K \simeq \CO/\CJ$ with $ \CJ$ a maximal ideal containing $ \CI$. This  homology is a graded vector space of dimension $ b_i(\CO/\CJ,\CI) \in \mathbb N_0 \cap \{+\infty\}$ in degree $ i \geq 1$.  By Theorem \ref{thm:lowernboundonrank}, for any free $\CO $-module $\CEb=\oplus_{i \geq 1} \CE_i$ entering a Koszul-Tate resolution of $\CO/\CI $, the rank $\CE_i $ is greater or equal to $b_i(\CO/\CI,\CJ)$. This has consequences for the complexity of the Tate algorithm.

\begin{proposition}
\label{prop:TateIsInfinite}
   If infinitely many of the integers $(b_i(\CO/\CI,\CJ))_{i \in \mathbb N}$ are non-zero for a maximal ideal $ \CJ$ containing $ \CI$, then the Tate algorithm requires to solve an infinite number of $ \CO$-linear problems of the first type.
\end{proposition}
\begin{proof}
Consider a Koszul-Tate resolution $ (S(\CE), \delta)$  coming from the Tate algorithm as in Construction \ref{const:tate.algorithm}.
Theorem \ref{thm:lowernboundonrank} implies that $\CE_i \neq 0$ for every $i \in \mathbb N$ such that $b_i(\CO/\CI,\CJ) \neq 0$.
When  $ \CE_i$ is non-zero, the $i$-th step of the Tate algorithm requires to solve an $ \CO$-linear problem of the first type, since the $i$-th step consists in finding generators of the kernel of $ \delta_{i-1}$.
\end{proof}

\vspace{0.1cm}

\noindent
Proposition \ref{prop:TateIsInfinite} is useful when one can compute the homology of the reduced complex of $\CO/\CI $ tensored with $\CO/\CJ $ for some maximal ideal $ \CJ$.
Arborescent Koszul-Tate resolutions allow precisely to make such computations, since it allows to describe explicitly a reduced complex of $ \CO/\CI$.

\vspace{0.1cm}

\noindent
We conclude this section by applying this strategy to monomial ideals of a polynomial algebra.

\begin{proposition}
\label{prop:nonregular} 
Let $\CI \subset \CO = \mathbb K[x_1, \dots, x_n]$ be a monomial ideal minimally generated by a sequence $\phi_1, \dots, \phi_k $ which is not regular.
Then the Tate algorithm requires infinitely many algebraic operations.
\end{proposition}
\vspace{0.2cm}
\noindent
We start with a technical lemma.
\begin{lemma}
\label{lem:strictdivisor}
    Let $\varphi_1, \dots, \varphi_k$ be a not a regular sequence of monomials, then there exist distinct $i ,j \in \{1, \dots, k\}$ such that $\mathrm{lcm}\lbrace \varphi_i, \varphi_j \rbrace \neq \varphi_i\cdot \varphi_j$.
\end{lemma}
 \begin{proof}
Let $\varphi_1, \dots, \varphi_k$ be a not regular sequence in $\CO$. By definition, there exists $m\geq 2$ (case $m =1$ is excluded) such that $a\cdot \varphi_m = \alpha_1 \cdot \varphi_1 + \dots \alpha_{m-1}\cdot \varphi_{m-1}$, $a, \alpha_{\bullet} \in \CO$ and $a\notin \langle \varphi_1, \dots, \varphi_m \rangle_{\CO}$. Without any loss of generality we choose $a, \alpha_{\bullet}$ to be monomials and we impose that any non-zero summand $\alpha_i\varphi_i$ has the same multipolynomial degree as $a\varphi_m$. Then it easily follows if $\alpha_i\varphi_i \neq 0$ there exists $b\in \mathbb K^*$ such that $b\alpha_i\varphi_i = a\varphi_m$. Thus $\varphi_i$ divides  $a\varphi_m$. Since $\varphi_i$ does not divide $a$ alone according to the assumptions, the greatest common divisor of $\varphi_i$ and $\varphi_m$ is non-trivial and therefore their least common multiplier is not just given by their product. Therefore there exists a monomial $\chi \in \CJ $ such that $ \varphi_i\cdot \varphi_j=\chi \cdot \,\mathrm{lcm}\lbrace \varphi_i, \varphi_j \rbrace$.
\end{proof}

\vspace{0.2cm}
\noindent
Let $ \CJ \subset \CO$ be the maximal ideal generated by $x_1,\dots,x_n$. Notice that $ \CI \subset \CJ$.

\begin{lemma}
\label{lem:nonregular}
    For every  $i \geq 0 $, we have $ b_{2i+1}(\CO/\CI,\CJ) \neq 0$.
 
\end{lemma}
\begin{proof}
By Lemma \ref{lem:strictdivisor}, there exists indices $i,j$ and a monomial $ \chi \in \CJ$ such that \begin{equation}\label{eq:explicitproductmonomials} e_{\{i\}} \star e_{\{j\}}= \chi \, e_{\{i,j\}}.\end{equation} 

\vspace{1mm}
\noindent
Consider the arborescent Koszul-Tate resolution $(S(\CT[C(k)]), \delta_\psi)$, where the arborescent operations are chosen as in Proposition \ref{prop:psi.dga}.
In particular
\begin{equation}\label{eq:explicitproductmonomials2} \psi_\vee (e_{\{i\}} , e_{\{j\}})= \chi \, e_{\{i,j\}}\end{equation} 
takes values in $\CJ \CT[C(k)]$.
Let us now examine the ordered decorated tree $T_m$ of degree $ 2m+1$
         $$\adjustbox{valign=c}{$\begin{array}{c} \\T_m \, := \, \end{array}$}
         \adjustbox{valign=c}{\scalebox{0.5}{
    \begin{forest}
for tree = {grow' = 90}, nice empty nodes,
[,label = {[mystyle] \scalebox{2}{$ A_1 \quad$}}
[, edge = dashed, 
            [,label = {[mystyle] \scalebox{2}{$ A_m \quad$}}, edge = dashed,
      [\scalebox{2}{$e_{\{i\}}$}, tier =1 ]
      [\scalebox{2}{$e_{\{j\}}$}, tier =1]
         ]
         [\scalebox{2}{$e_{\{j\}}$}, tier = 1 ]
 ]
    [
    \scalebox{2}{$e_{\{j\}}$}, tier =1
    ]
 ]
\path[fill=black] (.parent anchor) circle[radius=3pt]
                (!11.child anchor) circle[radius=3pt]
                (!1.child anchor) circle[radius=3pt];
\end{forest}
}}\adjustbox{valign=c}{$.$}
$$
A direct computation using Equation \eqref{eq:deltaexpression} gives:
$$\adjustbox{valign=c}{$\begin{array}{c} \\
\delta_\psi T_m = (-1)^m\end{array}$}
\adjustbox{valign=c}{\scalebox{0.5}{
    \begin{forest}
for tree = {grow' = 90}, nice empty nodes,
[,label = {[mystyle] \scalebox{2}{$ A_1 \quad$}}
[, edge = dashed, 
            [,label = {[mystyle] \scalebox{2}{$ A_{m-1} \quad\quad $}}, edge = dashed,
      [\scalebox{2}{$\psi_\vee\left(e_{\{i\}}, e_{\{j\}}\right)$}, tier =1 ]
      [\scalebox{2}{$e_{\{j\}}$}, tier =1]
         ]
         [\scalebox{2}{$e_{\{j\}}$}, tier = 1 ]
 ]
    [
    \scalebox{2}{$e_{\{j\}}$}, tier =1
    ]
 ]
\path[fill=black] (.parent anchor) circle[radius=3pt]
                (!11.child anchor) circle[radius=3pt]
                (!1.child anchor) circle[radius=3pt];
\end{forest}
}} \adjustbox{valign=c}{$\begin{array}{c} \\+\end{array}$}\adjustbox{valign=c}{
\scalebox{0.5}{
    \begin{forest}
for tree = {grow' = 90}, nice empty nodes,
[,label = {[mystyle] \scalebox{2}{$ A_1 \quad$}}
[, edge = dashed, 
            [,label = {[mystyle] \scalebox{2}{$ A_{m-1} \quad \quad$}}, edge = dashed,
      [\scalebox{2}{$e_{\{i\}}$}, tier =1 ]
      [\scalebox{2}{$e_{\{j\}}$}, tier =1]
         ]
         [\scalebox{2}{$e_{\{j\}}$}, tier = 1 ]
 ]
    [
    \scalebox{2}{$e_{\{j\}}$}, tier =1
    ]
 ]
\path[fill=black] (.parent anchor) circle[radius=3pt]
                (!11.child anchor) circle[radius=3pt]
                (!1.child anchor) circle[radius=3pt];
\end{forest}
}}\adjustbox{valign=c}{$\begin{array}{c} \\
\odot \end{array}$}\adjustbox{valign=c}{
\scalebox{1.4}{\begin{forest}
for tree = {grow' = 90}, nice empty nodes,
[
 [\scalebox{0.714}{$e_{\{j\}}$}, tier =1] 
]
\end{forest}} \\ \vspace{0.5cm}}\adjustbox{valign=c}{$\begin{array}{c} \\,\end{array}$}
$$
all the remaining terms being zero either for symmetry reasons, or because the only non-vanishing arborescent operations $  \psi_t$ are those where $t$ is of the form given in Equation \eqref{eq:onerootnleaves}. 
The first term on the r.h.s. is in $ \CJ \, \CT[C(k)]$ in view of Equation \eqref{eq:explicitproductmonomials2}.
Since the second term is in $\Sdeg (\CT[C(k)])$,
we have that $$\delta_\psi  (T_m ) \in \CJ \, \CT[C(k)]\oplus \Sdeg (\CT[C(k)]).$$
Therefore, $T_m$ induces a closed element $[T_m]$ in the reduced complex $(\CT[C(k)] \otimes \CO/\CJ, \underline{\delta_\psi})$. 
Also, $[T_m]$ can not be exact: this follows from Equation \eqref{eq:deltaexpression}  since the image of $\underline{\delta_\psi}$ is made of linear combinations of trees which are either  non-binary trees, or trees with $d$-exact decorations, or trees with decorations in the image of $\psi_t$ for some tree $t$ as in Equation \eqref{eq:onerootnleaves} and therefore valued in  $C(k)_{\geq 2}$.
Since $ T_m$ has degree $ 2m+1$, we have $b_{2m+1}(\CO/\CI, \CJ) \neq 0$.
\end{proof}

\vspace{0.1cm}
\noindent Similarly to Lemma \ref{lem:nonregular}, one can also prove $b_{2i}(\CO/\CI,\CJ) \neq 0$, but since we do not need it for proving Proposition \ref{prop:nonregular}, we omit it here. 

\vspace{0.1cm}
\noindent
\begin{proof}[Proof of Proposition \ref{prop:nonregular}.]
The statement is an immediate consequence of Proposition \ref{prop:TateIsInfinite} and Lemma \ref{lem:nonregular}. 
\end{proof}

\begin{remark}
\normalfont
A statement similar
to Proposition \ref{prop:nonregular}
appears under the name "Rigidity Theorem" in Avramov's \cite{Avramov}, where Koszul-Tate resolutions are denoted by $ \hat{S}$ but are not given a name.
\end{remark}

\subsection{A  word of conclusion about simplicity}

\vspace{0.1cm}

\noindent
We now have two main types of Koszul-Tate resolutions:  the arborescent Kozul-Tate resolutions obtained of a well-chosen free or projective $\CO$-module resolution, and those obtained out of the Tate algorithm. Hence a natural question: Which of the two is simpler?

\vspace{0.1cm}

\noindent
From the point of view of complexity, i.e. the quantity of computations, arborescent Koszul-Tate resolutions are significantly simpler, as can be seen by comparing, e.g., Theorem \ref{th:finitelyMany}  with Proposition \ref{prop:nonregular}. 
In particular, they can often be computed using finitely many elementary operations, see, e.g. Corollary \ref{cor:finmanyoperations}.

\vspace{0.1cm}

\noindent
However, 
as we see in Appendix \ref{app:minimal}, if properly constructed, the Tate algorithm is minimal. It suffices to choose a minimal set of generators of the homology of the truncated complex $$H_k\left( S({\mathcal E}_{1}\oplus \dots \oplus  {\mathcal E}_{k}), \delta_k \right)$$ at each step of Construction \ref{const:tate.algorithm}, thus forming a "smallest possible" free $\CO$-module $\CE_{k+1}$. 
In contrast, the arborescent Koszul-Tate resolution is almost never minimal. 

\vspace{0.1cm}
\noindent
In conclusion, arborescent Koszul-Tate resolutions are simpler from the point of view of complexity, but not from the point of view of  size.

\vspace{0.1cm}
\noindent
Also, arborescent Koszul-Tate resolutions have easier homotopies. 
Assume a module resolution $(\FM_\bullet,d) $  of $ \CO/\CI$ of length $N$ is given.

\vspace{0.1cm}
\noindent
First, 
both its arborescent Koszul-Tate resolutions and any Koszul-Tate resolution obtained by the Tate algorithm are homotopy equivalent to $(FM_\bullet,d) $. While in our construction this homotopy is completely explicit, see Proposition \ref{prop:Homotopy}, while in the second case,  the construction of such a homotopy equivalence would again need  infinitely many computations.  

\vspace{0.1cm}
\noindent 
Also, Proposition \ref{prop:Homotopy} gives explicit expression of a homotopy $h$ for $\delta_\psi $ upon restricting to a subspace of the kernel of ${\mathrm{Proj}}_\psi $, i.e.\ in particular to $\CT[\FM_\bullet]_{\geq N+1}$. This is very useful to have in applications like the BV or BFV constructions, see \cite{SashaThomas}.

\section*{Acknowledgements}
\noindent A.H. and T.S. are thankful to Erwin Schrodinger Institute for a stay during a thematic programme "Higher structures and Field Theory",  where this work was partially conducted. 
C.L.-G.  would like to thank the Tsing Hua University \begin{CJK*}{UTF8}{bsmi}國立清華大學\end{CJK*}  and the NCTS \begin{CJK*}{UTF8}{bsmi}國家理論科學研究中心\end{CJK*}  for their hospitality.
Last, all three authors thank the Institut Henri Poincaré for its hospitality in 2023.\\

\vspace{0.1cm}

\noindent
We thank Ruben Louis, Vladimir Salnikov, Rong Tang and, in particular, Barbara Fantechi, Bernhard Keller, Oleksii Kotov and Pavel Mnev for useful  discussions.

\appendix

\section{An ab initio proof of Theorem \ref{thm:ainfty}}

\label{app:alternative}
\noindent The purpose of this Appendix is to give a detailed proof of Theorem \ref{thm:ainfty} without using the homotopy transfer theorem. To be more in line with the rest of the paper, we keep unshifted  degrees.
This requires to reformulate the defining properties of  $ A_\infty$ and $ C_\infty$ algebras correspondingly. For the case of the $ A_\infty$ algebra this then actually agrees with its original definition in the literature.  

\begin{definition} \cite{StasheffAinfty1, StasheffAinfty2, Keller1999IntroductionT}.
Let $ (\mathfrak A,d)$ be a complex of $ \CO$-modules. A sequence of $\CO $-linear maps $\mu_i \colon \otimes^i C_\bullet \to C_{\bullet} $,
  with $\mu_i $  of degree $i-2 $ 
is said to be a $ A_\infty$-algebra if it satisfies the "higher associativity" conditions for all $n \geq 1$: 
\begin{equation}
\label{eq:ainfty2}
\sum_{\scalebox{0.5}{$\begin{array}{c} i,k \geq 0, j \geq 1 \\ i+j+k = n \end{array}$}} 
(-1)^{i+jk}
\mu_{n-j+1}(\id^{\otimes i}\otimes \mu_{j} \otimes \id^{\otimes k}) = 0
\end{equation}
with the convention that $\mu_1=
 -
d $.
An $A_\infty$-algebra structure $(\mu_n)_{n \geq 1} $  on a graded $ \CO$-module ${\mathfrak{A}}_\bullet $  is said to be a \emph{$ C_\infty$-algebra} if the following condition is satisfied: for every $ n \geq 1$ and all homogeneous $ a_1, \dots,a_n \in {\mathfrak A}_\bullet$, one has
\begin{equation}\label{eq:Cinfty2} \sum_{\sigma \in {\mathrm{Sh}}(i,n)} 
 e(\sigma)
\theta(\sigma,a)  \mu_{n}(a_{\sigma(1)}, \dots, a_{\sigma(n)})  =0 .\end{equation}
 Here ${\mathrm{Sh}}(i,n)$ stands for the set of all $(i,n)$-shuffles of $\{1, \dots, n\} $, i.e. all permutations $ \sigma$ of $ \{1, \dots,n\}$ such that  $ \sigma(a) \leq \sigma (b) $ if   $1 \leq a \leq b \leq i$ or $i+1 \leq a \leq b \leq n$.
Furthermore, 
  $e(\sigma) \in \{-1,1\}$ is  the ordinary signature of $ \sigma$ and 
$\theta(\sigma,a) $ is the Koszul signature of $\sigma$  defined to be the unique sign that satisfies
 $$ a_{\sigma(1)} \odot \dots \odot a_{\sigma(n)} = \theta(\sigma,a) \, a_1 \odot \dots \odot a_n  $$
  within the graded symmetric algebra $ S(\mathfrak A)$  generated by $\mathfrak A $.
\end{definition}

\noindent
For the non-shifted degree,  theorem  \ref{thm:ainfty} becomes:

\begin{theorem}
\label{thm:ainfty2}
Consider an arborescent Koszul-Tate resolution of $\CO/\CI $ associated to a free or projective $\CO$-module resolution $(\FM_\bullet,d) $ with arborescent operations $ \psi$  and differential  $\delta_\psi $.
Then the following structure maps $(\mu_n)_{n\geq 1} $  define an $A_\infty $-algebra on $ \FM_\bullet \oplus \CO$:

\begin{enumerate}
\item For $n=1 $, we set $ \mu_1=
 -
d$.
\item For $ n=2$, we set for all $ a_1,a_2 \in \FM_\bullet$ and $ F_1,F_2 \in \CO$  $$\begin{array}{rrcl}
    \mu_2\colon & \otimes^2  \left( {\FM}_\bullet \oplus \CO \right) & \rightarrow &  {\FM}_\bullet  \oplus \CO \\
     & (a_1 + F_1 \, ,  \, a_2+F_2) & \mapsto & F_1a_2+ F_2 a_1  +\psi_\vee (a_1,a_2) + F_1 F_2
 \end{array} .$$
\item For all $n \geq 3$, we set for all homogeneous $ a_1, \dots, a_n \in \FM_\bullet$ and $F_1, \dots, F_n \in\CO$
 \begin{equation}
     \label{eq:homotopytransfer2} \begin{array}{rrcl}
    \mu_n\colon & \otimes^n  \left( {\FM}_\bullet \oplus \CO\right) & \rightarrow & {\FM}_\bullet  \oplus  \CO \\
     & (a_1+F_1, \dots, a_n+F_n) & \mapsto & \sum_{ t \in Y_n} (-1)^{P(t[a_1, \dots,a_n])
 +\sum_{i=1}^{n}(n-i)|a_i|
     } 
     \, \,  \psi_t(a_1, \dots, a_n), 
 \end{array} 
 \end{equation}
 where $Y_n$ is the set of ordered binary trees with $n$ leaves, and 
$P(t[a_1, \dots,a_n])$ is the sum of the degrees of all left subtrees
of vertices in $t[a_1, \dots,a_n]$.
 \end{enumerate}
{ Moreover, the structure maps $(\mu_n)_{n \geq 1} $ satisfy \eqref{eq:Cinfty2}, proving that one even has a $C_\infty$-structure.
}
\end{theorem}

\vspace{0.1cm}
\noindent
Let us explain  the third item.
Let $t[a]$ be an ordered decorated tree with $k$ leaves where $a$ denotes a list of $k$ elements in  $\FM_\bullet$.  A strict subtree $t_1[b]$, with $ b$ a sublist of $a$,  can be located to  the left or to the right
of a vertex $A$ of $t$: 
For instance, in the following ordered binary tree, both $ t_2[c] $ and $t_3[e] $ have  the vertex $ A$ as parent of their roots, but only $ t_2[c]$ is located to the left of $A$. 
\begin{equation}
    \label{eq:exampleRightLeft}
\adjustbox{valign =c}{ $\begin{array}{c} t[a] =\\ \\ \\ \\ \\  \end{array}$\scalebox{0.5}{
    \begin{forest}
for tree = {grow' = 90}, nice empty nodes,
[
[ , edge = dashed 
         [\scalebox{2}{$t_1[b]$}, tier = 1 ]
            [,label = {[mystyle] \scalebox{2}{$A\enspace$}},
      [\scalebox{2}{$t_{2}[c]$}, tier =1]
      [\scalebox{2}{$t_3[e]$}, tier =1] 
         ]
 ]
 ]
\path[fill=black] (.parent anchor) circle[radius=3pt]
                (!1.child anchor) circle[radius=3pt]
                (!12.child anchor) circle[radius=3pt];
\end{forest}
}}
\end{equation}
\vspace{-10mm}

\noindent As an example, let us describe  $\mu_3 $ with its signs.
We have $ Y_3=\{\ltree,\rtree\}$.
For the decorated tree $\ltree[a_1,a_2,a_3] $, the left subtrees are the subtrees $| \otimes a_1 $  and $ \vtree[a_1,a_2]$. Hence $ P(\ltree[a_1,a_2,a_3])=|a_2|+1$.
For the tree $\rtree[a_1,a_2,a_3] $, the left subtrees are $|\otimes a_1$ and $|\otimes a_2$, so that
$ P(\rtree[a_1,a_2,a_3])=|a_1|+|a_2|$.
One then computes 
\begin{equation} 
\label{eq:mu3}
\mu_3(a_1,a_2,a_3)= - \psi_{\ltree}(a_1,a_2,a_3) + (-1)^{|a_1|}\psi_{\rtree}(a_1,a_2,a_3)  .
\end{equation}
(Notice that this is consistent with Equation \eqref{eq:m3}).

\vspace{0.1cm}
\noindent
Also  the set $Y_n$ of ordered binary trees with  $n$ leaves:
$$Y_1 = \lbrace \vert \rbrace, \enspace Y_2 = \left\{ \begin{array}{c}\scalebox{0.55}{\begin{forest}
for tree = {grow' = 90}, nice empty nodes, for tree={ inner sep=0 pt, s sep= 0 pt, fit=band, 
},
[[, tier =1] [, tier= 1]]
\path[fill=black] (.parent anchor) circle[radius=2pt];
\end{forest}} \end{array} \right\}, \quad 
Y_3 = \left\{ \begin{array}{c}\scalebox{0.3}{
 \begin{forest}
for tree = {grow' = 90}, nice empty nodes,
[
 [
 [, tier =1] 
 [, tier = 1]
 ]
 [, tier =1]
]
\path[fill=black] (.parent anchor) circle[radius=2pt]
                (!1.child anchor) circle[radius=2pt];
\end{forest} }
,
\scalebox{0.3}{\begin{forest}
for tree = {grow' = 90}, nice empty nodes,
[
 [, tier =1]
 [
 [, tier =1] 
 [, tier = 1]
 ]
]
\path[fill=black] (.parent anchor) circle[radius=2pt]
                (!2.child anchor) circle[radius=2pt];
\end{forest}}
\end{array}\right\}, \cdots
$$
is described recursively \cite{Loday2002order} by:
\begin{equation}\label{eq:recYn}\quad Y_n =  \left\{ \begin{array}{c}  \scalebox{0.5}{
\begin{forest}
    for tree = {grow' = 90}, nice empty nodes,
[
 [\scalebox{2}{$Y_i$}, tier =1]
 [\scalebox{2}{$Y_k$}, tier =1] 
]
\path[fill=black] (.parent anchor) circle[radius=4pt];
\end{forest}}\\ \end{array}, \, i+k =n \right\} \simeq \bigsqcup_{i+k=n}Y_i\times Y_k.\end{equation}

\vspace{1mm}
\noindent
Before proving Theorem \ref{thm:ainfty2}, some considerations about ordered binary trees are required.
For any homogeneous $ a_1, \dots, a_n \in \FM_\bullet \oplus \CO$, 
we define an element $ k_n[a_1, \dots, a_n] \in \
Tree[\FMb]$ by the recursive formula
 \begin{equation}
 \label{eq:defkn}
   k_n[a_1, \dots,a_n] = \sum_{j=1}^{n-1} (-1)^{|k_j[a_1, \dots, a_j]|} \r \left( k_j[a_1, \dots, a_j] \otimes k_{n-j}[a_{j+1}, \dots, a_n]  \right)  
 \end{equation}    
with initial condition $ k_1(a_1)= |[a_1]$ for all $ a_1\in \FMb$. Here $\r $ is the root map defined on ordered decorated trees, and $|\cdot|$ is the degree of a decorated tree as defined in Section \ref{sec:space}. 
Equation \eqref{eq:defkn} may be also depicted graphically as follows: 
$$ 
\adjustbox{valign =c}{
$
k_n[a] =\sum_{j = 1}^{n-1} (-1)^{|k_j[a_{\leq j}]|}$}
\adjustbox{valign =c}{\begin{forest}
    for tree = {grow' = 90}, nice empty nodes,
[
 [$k_j{[a_{\leq j}]}$, tier =1]
 [$k_{m-j} {[a_{>j}]}$, tier =1] 
]
\path[fill=black] (.parent anchor) circle[radius=2pt];
\end{forest}} 
$$
where now $ a= a_1, \dots, a_n$,  $ a_{\leq j} =a_1, \dots, a_j$, and  $a_{>j} = a_{j+1}, \dots, a_n$.
\vspace{1mm}
\noindent


\begin{example}
\normalfont
Let us compute the first terms. For $n=2,3$, we obtain:
$$ 
\adjustbox{valign =c}{$k_2[a_1, a_2] = (-1)^{|a_1|}$}
\adjustbox{valign =c}{\begin{forest}
    for tree = {grow' = 90}, nice empty nodes,
[
 [$a_1$, tier =1]
 [$a_2$, tier =1] 
]
\path[fill=black] (.parent anchor) circle[radius=2pt];
\end{forest}}
$$

\begin{equation}
\label{eq:k3}
\adjustbox{valign =c}{$\quad k_3[a_1, a_2, a_3] = (-1)^{|a_1| + |a_2|}$}
\adjustbox{valign =c}{
\scalebox{0.5}{
\begin{forest}
for tree = {grow' = 90}, nice empty nodes,
[
 [\scalebox{2}{$a_1$}, tier =1]
 [
 [\scalebox{2}{$a_2$}, tier =1] 
 [\scalebox{2}{$a_3$}, tier = 1]
 ]
]
\path[fill=black] (.parent anchor) circle[radius=4pt]
                (!2.child anchor) circle[radius=4pt];
\end{forest}}
}
\adjustbox{valign =c}{$
+
(-1)^{2|a_1| +|a_2| +1}$}
\adjustbox{valign =c}{
\scalebox{0.5}{
\begin{forest}
for tree = {grow' = 90}, nice empty nodes,
[
[
 [\scalebox{2}{$a_1$}, tier =1]
 [\scalebox{2}{$a_2$}, tier =1] 
]
 [\scalebox{2}{$a_3$}, tier = 1]
 ]
]
\path[fill=black] (.parent anchor) circle[radius=4pt]
                (!1.child anchor) circle[radius=4pt];
\end{forest}}}.
\end{equation}
\end{example}

\vspace{1mm}\noindent We intend to construct $ \mu_n(a_1,\dots,a_n)$ out of  $ \psi$ and the fixed element $k_n[a_1, \dots,a_n]$, see Equation \eqref{eq:defmuoutofk} below. For this purpose it is  useful to first study some properties of $k_n[a_1, \dots, a_n]$: 
\begin{lemma}
\label{lem:knexplicit}
 For every $n \geq 1$, and every $a=(a_1, \dots,a_n) \in (\FM_\bullet)^n$, 
 $k_n[a]$ has degree $\sum_{i=1}^n |a_i| + n-1 $. It is 
 a linear combination  of all  ordered binary trees with $n$ leaves decorated with $ a_1, \dots, a_n$. The prefactor of all terms of this linear combination takes values in $\{ +1,-1\}$ and   explicitly one has:
\begin{equation}
\label{eq:kmExpl} k_n [a] = \sum_{t \in Y_n} (-1)^{P(t[a])} \,t[a]
\end{equation}
 where $P(t[a])$ is the same integer that was defined in Theorem \ref{thm:ainfty2} and the text following it.
\end{lemma}

\begin{proof} Equation \eqref{eq:recYn} 
implies that all $ t[a]$, with $t \in Y_n$, appears once and only once in the expression of $ k_n[a]$ given by the recurrence relation \eqref{eq:defkn}. 
This proves that $k_n[a] $ is a linear combination of all ordered  binary trees
decorated by $ a_1, \dots,a_n$. Being binary, such trees always have $n-1 $ vertices, which thus leads to the claimed degree of $ k_n[a]$.  

\vspace{1mm}\noindent 
Let us compute the exact value of the prefactors.
We denote by 
$\mathcal C_l(t[a])$
the 
list of 
all left subtrees of  an ordered tree $t[a] $.  For example,  in the tree $t[a]$ in Equation \eqref{eq:exampleRightLeft}, the subtrees $ t_1[b]$ and $t_2[c] $ would belong to that list (completed by the contributions of all strict left subtrees of $ t_1[b]$, $t_2[c] $,  and $t_3[e] $).  For all ordered binary trees $ t_1$ and  $t_2$ decorated by $b$ and $c$, respectively, one has
  $$ \mathcal C_l( \r (t_1 [b] \otimes t_2[c] )) = \mathcal C_l(t_1[b]) \sqcup \mathcal C_l(t_2[c]) \sqcup \{t_1[b]\}.  $$
Since $P(t[a]) $ is the sum of the degrees of the trees in $ \mathcal C_l(t[a])$, we find \begin{equation}\label{eq:Pn} P(\r (t_1 [b] \otimes t_{2}[ c] )) = 
 P(t_1[b] )+
P (t_2 [c] ) + |t_1[b]|. 
\end{equation}
In turn, this relation allows to check that the linear combination given at 
 the right hand side of \eqref{eq:Pn}
satisfies the  recursion relation  Equation \eqref{eq:defkn}. By uniqueness, both expressions coincide. This completes the proof.
\end{proof}

\vspace{.3cm}

\begin{lemma}
\label{lem:a_infty1}
Each element of $\lbrace k_n[a] \rbrace$ lies in  $\ker(\partial)$, i.e. for all $ n \geq 1$:
    $$   \partial k_n[a] =0,
    $$
    where $\partial $
 is the differential on $Tree[\FMb]$ defined in Section \ref{sec:constructingdeltapsi}. \end{lemma}
\begin{proof}
Let $t[a] $ be a binary decorated tree with $n$ leaves, $a$ a list of $n$ homogeneous elements in $ \FM_\bullet$, and $A$ an inner vertex of $ t$.
    For every pair $(t[a],A)$ there is a dual one, $ \overline{(t[a],A)} =(\bar{t}_A[a],\bar{A})$, constructed as follows. Assume that $A$ is the root of the right subtree of its parent $P_A $ as in the following picture,
    \begin{equation}
        \label{tree:Aright}
\adjustbox{valign =c}{$t{[a]} =$}
\adjustbox{valign =c}{ \scalebox{0.5}{
    \begin{forest}
for tree = {grow' = 90}, nice empty nodes,
[
[ ,label = {[mystyle] \scalebox{2}{$P_A \enspace$}}, edge = dashed 
         [\scalebox{2}{$t_1{[b]}$}, tier = 1 ]
            [,label = {[mystyle] \scalebox{2}{$A \enspace$}},
      [\scalebox{2}{$t_2{[c]}$}, tier =1]
      [\scalebox{2}{$t_3{[e]}$}, tier =1] 
         ]
 ]
 ]
\path[fill=black] (.parent anchor) circle[radius=3pt]
                (!1.child anchor) circle[radius=3pt]
                (!12.child anchor) circle[radius=3pt];
\end{forest}}
},
    \end{equation}
    where $t_1[b], t_2[c], t_3[d]$ are  subtrees of $ t[a]$. Then we define $\overline{ (t[a],A)}=(\bar{t}_A[a],\bar{A})$  as follows:
  
          \begin{equation}
        \label{tree:Aleft}
\adjustbox{valign =c}{$\bar{t}_A[a] =$} 
\adjustbox{valign =c}{\scalebox{0.5}{
    \begin{forest}
for tree = {grow' = 90}, nice empty nodes,
[
[ ,label = {[mystyle] \scalebox{2}{$P_{\bar A} \enspace \enspace $}}, edge = dashed 
            [,label = {[mystyle] \scalebox{2}{$\bar A \enspace \enspace$}},
      [\scalebox{2}{$t_1{[b]}$}, tier =1]
      [\scalebox{2}{$t_2{[c]}$}, tier =1] 
         ]
         [\scalebox{2}{$t_3{[e]}$}, tier = 1 ]
 ]
 ]
\path[fill=black] (.parent anchor) circle[radius=3pt]
                (!1.child anchor) circle[radius=3pt]
                (!11.child anchor) circle[radius=3pt];
\end{forest}
}}
\end{equation}
with the rest of the tree $\overline{t}_A $  being identical to the corresponding part in $t$. We then complete the description of the duality by imposing that it has to be an involution.
Notice that no pair $(t[a],A) $ is dual to itself.

\vspace{0.1cm}
\noindent
In view of the explicit description of $ \partial $ in 
Equation \eqref{eq:differential_on_trees}
and the explicit description of $k_n[a] $ in
Equation \eqref{eq:defkn}, we see that
$\partial k_n[a]$ is given by 
\begin{equation}
\label{eq:deltaasdoublesum}
\partial k_m[a] = \sum_{(t[a],A) } (-1)^{W_A+P(t[a])}  \partial_A (t[a]) \: .
\end{equation}
 Here the sum runs over all pairs $ (t[a],A)$, where $t \in Y_n$ is an ordered decorated tree with $n$ leaves, and $A$ is an inner vertex of $t$.

\vspace{0.2cm}
\noindent
Note that in the above sum, for every pair $(t[a],A)$ there is also a contribution from its dual pair  $\overline{ (t[a],A)}=(\bar{t}_A,\bar A))$. 
We will now show that they cancel against each other. Note first  that
$$  
\partial_A(t[a])=\adjustbox{valign =c}{\scalebox{0.5}{
    \begin{forest}
for tree = {grow' = 90}, nice empty nodes,
[
[ ,label = {[mystyle] \scalebox{2}{$A \mapsto P_{A} \enspace \enspace \quad \quad $}}, edge = dashed,
      [\scalebox{2}{$t_1{[b]}$}, tier =1]
      [\scalebox{2}{$t_2{[c]}$}, tier =1] 
         [\scalebox{2}{$t_3{[e]}$}, tier = 1 ]
 ]
 ]
\path[fill=black] (.parent anchor) circle[radius=3pt]
                (!1.child anchor) circle[radius=3pt];
\end{forest}
}} =\partial_{\bar{A}}(\bar t_A [a]  ) 
$$
where we used the pictorial description in Equations \eqref{tree:Aleft} and \eqref{tree:Aright}.
We are therefore just left with showing that
\begin{equation}\label{eq:oppositesigns} (-1)^{P(t[a])+W_A} = -(-1)^{ P(\bar t_A[a])+W_{\bar A}}.\end{equation}
Recall that the integer $P(t[a])$ 
is obtained by summing up all the degrees of the left subtrees of every vertex in $t$, its root included.
But the degrees of these subtrees coincide for both trees $t[a] $ and $\bar{t}_A[a] $, except for the ones that we now describe. The subtree $ t_2[c]$  only appears in the tree \eqref{tree:Aright} and not in the tree \eqref{tree:Aleft}, while the tree $ \r(t_1[b],t_2[c])$  appears as a subtree of the tree \eqref{tree:Aleft} but not of the tree $ \eqref{tree:Aright}$. The degree of these trees are $|t_2[c]| $ and $|t_1[b]|+|t_2[c]| +1 $, respectively, which yields

$$(-1)^{P(t[a])}= -(-1)^{P(\bar{t}_A[a])} (-1)^{|t_2[c]|}. $$
Likewise, the weight $ W_A$ is the sum of the degrees of all trees on the left of the path from the root to $A$, to which the length of that path is added.
Comparing $W_A$ and $ W_{\bar A}$ as in the trees \eqref{tree:Aleft} and \eqref{tree:Aright}, one easily finds
 $$  (-1)^{W_A}=(-1)^{W_{\bar{A}}} (-1)^{|t_2[c]|}.$$
This then shows that the equality in Equation 
\eqref{eq:oppositesigns}  holds true.

\end{proof}

\begin{lemma}
\label{lem:aboutmuexplicit}
The following relation holds:
\begin{equation}\label{eq:meaningmu}
\begin{array}{ll}&(-1)^{i+jk}\mu_{n-j+1}(\id^{\otimes i}\otimes \mu_{j} \otimes \id^{\otimes k})  (a_1, \dots, a_n)\\[2pt]=&
\sum\limits_{(t_1,t_2) \in Y_{n-j+1} \times Y_{j}}  \, \epsilon_{t_1,t_2}  \,   \psi_{t_1}(a_1, \dots, a_i  ,  
 \psi_{t_2}(a_{i+1}, \dots, a_{i+j}), a_{i+j+1},\dots, a_n) 
\end{array}\end{equation}
where
  \begin{equation}\label{eq:defepsilon} \begin{array}{rcl}\epsilon_{t_1,t_2} &=&  \eta \,  (-1)^{i +\sum_{r=1}^{i} |a_{r}| +P(t_2[a_{i+1}, 
\dots,a_{i+j}])+ P(t_1[a_1, \dots, a_i  ,  
 \psi_{t_2}(a_{i+1}, \dots, a_{i+j}), a_{i+j+1},\dots, a_n])} 

 \end{array} \end{equation}
with $ \eta =(-1)^{\sum_{r=1}^n |a_r|(n-r)}$.
\end{lemma}
\begin{proof}
A direct computation gives
$$\begin{array}{ll}&(-1)^{i+jk}\mu_{n-j+1}(\id^{\otimes i}\otimes \mu_{j} \otimes \id^{\otimes k})  (a_1, \dots, a_n)\\[2pt]=&(-1)^{i+jk +(j-2)(|a_1|+\dots + |a_i|)}\mu_{n-j+1}(a_{1}, \dots, a_i, \mu_{j}(a_{i+1}, \dots, a_{i+j}), a_{i+j+1}, \dots, a_{i+j+k}) \\[2pt]=&
\sum\limits_{(t_1,t_2) \in Y_{n-j+1} \times Y_{j}}  \, \epsilon_{t_1,t_2}  \,   \psi_{t_1}(a_1, \dots, a_i  ,  
 \psi_{t_2}(a_{i+1}, \dots, a_{i+j}), a_{i+j+1},\dots, a_n) 
\end{array}
$$
Here we used $ |\mu_j| = j-2$ and the Koszul sign convention when exchanging the order of elements in $\FMb$  to go from the first to the second line. To go from the second to the third line one makes use of the  definition of $ \mu_\bullet$ given in Equation \eqref{eq:homotopytransfer2}.
Furthermore, 
$\epsilon_{t_1,t_2} $ is given  by the following complicated but explicit expression:
$$ \begin{array}{rcll}\epsilon_{t_1,t_2} &=& & (-1)^{i+jk +(\sum_{r=1}^i|a_i|)(j-2)}\\
& & \times &(-1)^{P(t_2[a_{i+1}, 
\dots,a_{i+j}])+ \sum_{r=i+1}^{i+j} |a_{r}|(i+j-r)} \\
& & \times & (-1)^{P(t_1[a_1, \dots, a_i  ,  
 \psi_{t_2}(a_{i+1}, \dots, a_{i+j}), a_{i+j+1},\dots, a_n])} \\
 & & \times& (-1)^{  \sum_{r=1}^{i} |a_{r}| (n+1-j-r)+ |\psi_{t_2}[a_{i+1}, \dots, a_{i+j}]|(n-i-j)+\sum_{r=i+j+1}^{n} |a_{r}| (n-r)   } .\end{array}$$ 
Using $ |\psi_{t_2}(a_{i+1, \dots, a_{i+j}})| =  \sum_{r=i+1}^{i+j} |a_r| + j-2 $  and $ n-i-j=k$, this gives, by regrouping  terms in an appropriate manner:
 $$
 \begin{array}{rcll} \epsilon_{t_1,t_2} &=& & (-1)^{i+jk }\\
& & \times &(-1)^{P(t_2[a_{i+1}, 
\dots,a_{i+j}])+  
P(t_1[a_1, \dots, a_i  ,  
 \psi_{t_2}(a_{i+1}, \dots, a_{i+j}), a_{i+j+1},\dots, a_n])} \\ & & \times & (-1)^{\sum_{r=1}^{i} |a_{r}| }
\\ & & \times&  (-1)^{\sum_{r=1}^{i} |a_{r}| (n-r)+ \sum_{r=i+1}^{i+j}|a_r|(n-r)+(2-j)k+\sum_{r=i+j+1}^{n} |a_{r}| (n-r)   } \\
 &=& & (-1)^{i}\\
& & \times &(-1)^{P(t_2[a_{i+1}, 
\dots,a_{i+j}])+ P(t_1[a_1, \dots, a_i  ,  
 \psi_{t_2}(a_{i+1}, \dots, a_{i+j}), a_{i+j+1},\dots, a_n])} \\ & & \times & (-1)^{\sum_{r=1}^{i} |a_{r}| }
\\ & & \times&  (-1)^{\sum_{r=1}^{n} |a_{r}| (n-r)} 
 .\end{array} $$
This is the desired expression.
\end{proof}


\begin{lemma}
\label{lem:explicitrelations}
    Let $ t[a_1, \dots,a_n]$ be an ordered decorated tree, $A$ be a vertex of $t$, and $ a_{i+1}, \dots,a_{i+j}$ the decorations that descend from $A$. We have
\begin{equation*}  \begin{array}{rcl}
        W_{A} &=& i + \sum_{r=1}^i|a_r| + b_A \\
        P(t[a]) &=& P(t_{\uparrow A}[ a_{i+1}, \dots, a_{i+j}]) + P(t_{\downarrow A}[a_{1}, \dots, a_{i}, \psi_{t_2}(a_{i+1}, \dots, a_{i+j}),a_{i+j+1}, \dots, a_{n} ])+ b_A.
    \end{array}
\end{equation*}
Here $b_A$ is the number of vertices $B$ of $t[a]$ such that $A$ belongs to the left subtree of $B$.
\end{lemma}
\noindent For the tree in Equation \eqref{eq:AB}, $b_A=1$ since the vertex $B$ is the only one that satisfies the condition above. Furthermore, $ b_B=b_R=0$. 
 \begin{proof}[Proof of Lemma \ref{lem:explicitrelations}]
We leave the first relation to the reader and only prove the second one. 
Let us use the shorthand $a$ for the list $ a_1,\dots,a_n$ and $a_A$ for the list $a_{i+1}, \dots, a_{i+j} $.
Moreover, we denote the vertices of $t$ which are not a leaf by
 ${\mathrm{RVer}}(t)$, i.e.
  $$
  {\mathrm{RVer}}(t) =
  \{ {\mathrm{Root}}(t)\} \, \cup  \,   {\mathrm{InnVer}}(t), $$ and the left subtree of $B \in {\mathrm{RVer}}(t) $ by $\ell(B,t[a]) $. For the trees $t_{\uparrow A} $ and $t_{\downarrow A} $, we use the notations correspondingly. 
A vertex $B \in {\mathrm{RVer}}(t)$ can be either in $ {\mathrm{RVer}}(t_{\uparrow A})$, or in ${\mathrm{RVer}}(t_{\downarrow A}) $. In the later case, $A$ either belongs to the left tree of $B$ or  not. Since by definition 
$P(t[a])= \sum_{B \in {\mathrm{RVer}}(t) } |\ell(B,t[a])|$, we have
 \begin{equation}\label{eq:firstone0}  P(t[a])= \left( \sum_{B \in {\mathrm{RVer}}(t_{\uparrow A}[a])}  + \sum_{\scalebox{0.5}{ $\begin{array}{c}B \in {\mathrm{RVer}}(t_{\downarrow A}[a]) \\ A \notin \ell(B,t[a]) \end{array}$ }} +\sum_{\scalebox{0.5}{ $\begin{array}{c}B \in {\mathrm{RVer}}(t_{\downarrow A}[a]) \\ A \in \ell(B,t[a]) \end{array}$ }}   \right) |\ell(B,t[a])| .\end{equation}
For  $B\in {\mathrm{RVer}}(t_{\uparrow A}) $, the left subtrees $ \ell(B,t[a])$ and $\ell(B,t_{\uparrow A}[a_A] ) $ coincide and
 \begin{equation}\label{eq:firstone1}     \sum_{B \in {\mathrm{RVer}}(t_{\uparrow A})}|\ell(B,t[a])|=\sum_{B \in {\mathrm{RVer}}(t_{\uparrow A})}  |\ell(B,t_{\uparrow A}[a_A])|=P(t_{\uparrow}[a]).\end{equation}
 For a vertex $B \in {\mathrm{RVer}}(t_{\downarrow A}) $ such that $A$ does not belong to  $\ell(B,t[a])$, the left subtrees   
 $ \ell(B,t[a]) $ and $\ell\left(B,t_{\downarrow A}[a_1, \dots,a_i,\psi_{t_{\uparrow}}(a_A), a_{i+j+1}, \dots, a_n] \right) $ coincide, and 
 \begin{equation}\label{eq:firstone2}  \sum_{\scalebox{0.5}{ $\begin{array}{c}B \in {\mathrm{RVer}}(t_{\downarrow A}[a]) \\ A \notin \ell(B,t[a]) \end{array}$ }}   |\ell(B,t[a])|   = \sum_{\scalebox{0.5}{ $\begin{array}{c}B \in {\mathrm{RVer}}(t_{\downarrow A}[a]) \\ A \notin \ell(B,t[a]) \end{array}$ }}  |\ell(B,t_{\downarrow A}[a_1, \dots,a_i,\psi_{t_{\uparrow}}(a_A), a_{i+j+1}, \dots, a_n] )|  . \end{equation}
 If $A$ belongs to the left subtree of $B$ in $ t[a]$, then  $\ell(B,t_{\downarrow A}[a_1, \dots,a_i,\psi_{t_{\uparrow}}(a_A), a_{i+j+1}, \dots, a_n] ) $ is obtained out of $ \ell(B,t[a])$ by replacing the subtree $t_{\uparrow A}(a_A)$  by a new leaf with decoration $ \psi_t[a_{i+1}, \dots, a_{i+j}]$. This substitution reduces the degree by 
one.  Since there are $b_A$ such vertices,  one obtains 
 \begin{equation}\label{eq:firstone3} \sum_{\scalebox{0.5}{ $\begin{array}{c}B \in {\mathrm{RVer}}(t_{\downarrow A}[a]) \\ A \in \ell(B,t[a]) \end{array} $}}   |\ell(B,t[a])|   = \sum_{\scalebox{0.5}{ $\begin{array}{c}B \in {\mathrm{RVer}}(t_{\downarrow A}[a]) \\ A \in \ell(B,t[a]) \end{array} $}}  |\ell(B,t[a_1, \dots,a_i,\psi_{t_{\uparrow}}(a_A), a_{i+j+1}, \dots, a_n] )|  +b_A. \end{equation}
 Equations 
\eqref{eq:firstone1}, \eqref{eq:firstone2}, and \eqref{eq:firstone3} 
give an expression for each one of the three terms that appear in 
\eqref{eq:firstone0}. When combined, this gives the desired result.
 \end{proof}

\begin{proof}[Proof of Theorem \ref{thm:ainfty2}]
We have to check that the collection of the $\CO$-linear maps 
$
(\mu_n)_{n \geq 1}
$ as in Theorem \ref{thm:ainfty2}
satisfy the axioms of an $A_\infty$-algebra given in Equation \eqref{eq:ainfty2}.
For $n=1$, this condition simply means that $ d^2=0$, which is true. For $n=2$, it means that $\mu_2 $ is a chain map, which one checks by a direct calculation.

\vspace{2mm} \noindent 
In order to check it for all $n \geq 3$, we first notice that
\begin{equation}
\label{eq:defmuoutofk}\mu_n (a_1, \dots , a_n )  (-1)^{\sum_{i=1}^n|a_i|(n-i)}(\psi\circ \mathrm{pr})(k_n[a_1, \dots, a_n]),\end{equation}
where $\psi$ is the chosen arborescent operation and $\mathrm{pr} $ is the projection from ordered to symmetrized trees as in Equation \eqref{eq:psitDef}.
To show this, we  use the properties of $k_n[a_1, \dots, a_n] $, i.e Lemmas \ref{lem:knexplicit} and \ref{lem:a_infty1}, and of $\psi$, i.e. Equation \eqref{eq:deltaPsiSquare}.

\vspace{0.1cm}

\noindent
Assume that the axioms Equation \eqref{eq:ainfty2} are satisfied up to order $n-1$. 
First let us verify the relation \eqref{eq:ainfty2}  when applied to an $n$-tuple $ a_1 , \dots , a_n \in \FM_\bullet \oplus \CO$ where for at least one index $l$ we have $a_l=F \in \CO$.
By definition, we have $\mu_k(a_{i_1}, \dots, a_{i_k}) =0$ for all $ k \geq 3$ provided one of the indices $ i_1, \dots, i_k$ is equal to $ l$.
There are only two non-zero contributions in \eqref{eq:ainfty2}.
When $ l $ is neither $1$ nor $n$, the two non-zero contributions  are
\begin{align*}
    (-1)^{l-1}\mu_{n-1}(a_1, \dots,  a_{l-1}, \mu_2(F, a_{l+1}), \dots, a_n )+ (-1)^{l-2}\mu_{n-1}(a_1, \dots,  a_{l-2}, \mu_2(a_{l-1}, F), \dots, a_n )   .
\end{align*}
They add up to zero since 
$ \mu_2(F,b+G)=F (b+G)=\mu_2(b+G,F)$ for every $b \in \FM_\bullet, G \in \CO $ and because $ \mu_{n-1}$ is $ \CO$-linear.
For $l = 1$ and $l=n$, the two non-zero contributions are
$$
\mu_{n-1}(\mu_2(F, a_2), a_3, \dots, a_n )- \mu_2(F, \mu_{n-1}(a_2, \dots, a_n)),
$$
and
$$  (-1)^{n-2}\mu_{n-1}(a_1, \dots, \mu_2(a_{n-1}, F) )+(-1)^{n-1} \mu_2( \mu_{n-1}(a_1, \dots, a_{n-1}), F),
$$
respectively. In both cases,  the two terms add up to zero for the same reasons as  before.


\vspace{0.1cm}
\noindent
It remains to check that Equation \eqref{eq:ainfty2}  holds when each one of the elements $a_1 , \dots , a_n \in \FMb$ is homogeneous.
Since $ \psi$ is an arborescent operation such that  $ \delta_\psi^2=0$, Equation \eqref{eq:psi4} holds for every ordered tree $t$ with $n$ leaves. 
We sum  Equation \eqref{eq:psi4}  over all binary ordered trees $t \in Y_n$, with the prefactor $ (-1)^{P(t[a])}$ defined in Lemma \ref{lem:knexplicit}:
\begin{equation}
\label{eq:sumup}\sum_{t \in Y_n}(-1)^{P(t[a])}  \sum_{\scalebox{0.5}{$\begin{array}{r}A \in \mathrm{Ver}(t)\end{array}$}} \psi_{t_{\downarrow A}} \circ_{A} \psi_ {t_{\uparrow A}} (a_1, \dots, a_n)= 
\sum_{t \in Y_n}(-1)^{P(t[a])} \sum_{\scalebox{0.5}{$\begin{array}{r}A \in \mathrm{InnVer}(t)\\  \end{array}$}}  (-1)^{W_A} \psi_{\partial_A t} (a_1, \dots, a_n). 
   \     \end{equation}
 The right-hand side of Equation \eqref{eq:sumup} vanishes for the following reason:  
\begin{equation*}
\begin{array}{ll}
&\sum_{t \in Y_n}(-1)^{P(t[a])} \sum_{\scalebox{0.5}{$\begin{array}{r}A \in \mathrm{InnVer}(t)   \end{array}$}}  (-1)^{W_A}\psi_{\partial_A t} (a_1, \dots, a_n)    \\ =& \psi \circ {\mathrm{pr}} \left( \sum_{t \in Y_n}\sum_{\scalebox{0.5}{$\begin{array}{r}A \in \mathrm{InnVer}(t)\\ \end{array}$}} (-1)^{P(t[a]) + W_A} \partial_A t[a_1, \dots, a_n] \right)  \\ 
 =&  \psi \circ {\mathrm{pr}} \, \, (\partial (k_n[a_1, \dots, a_n])) = 0 . 
\end{array}
\end{equation*}
 The first equality follows from Equation \eqref{eq:psitDef}, the second one from Lemma \ref{lem:knexplicit} and the 
 definition of $ \partial$ given in Equation \eqref{eq:differential_on_trees}. The expression vanishes due to Lemma \ref{lem:a_infty1}. By the convention given in first line of Equation \eqref{eq:conventions}, Equation \eqref{eq:sumup} now becomes: 
\begin{equation}
\label{eq:TYn}
\sum_{t\in Y_n} 
\sum_{\scalebox{0.5}{$\begin{array}{r}A \in \mathrm{Ver}(t)\end{array}$}} 
     (-1)^{P(t[a])+W_A} \psi_{t_{\downarrow A}}(a_1, \dots, a_i, \psi_{t_{\uparrow A}}(a_{i+1}, \dots, a_{i+j}),a_{i+j+1} \dots, a_n)=0.
\end{equation}
\vspace{-2mm}

\noindent
We have to prove that the latter expression is equivalent to Equation \eqref{eq:ainfty2}.
Consider a triple $ (i,j,k) \in \mathbb N_0 \times \mathbb N \times \mathbb N_0$  such that $i+j+k =n $. We say that a pair $ (t[a],A)$, with $ t \in Y_n$  and $A $ a vertex (root included) or a leaf of $t$,  is of type $ (i,j,k)$ if the decorations of  $t_{\uparrow A} $ are $a_{i+1} , \dots , a_{i+j} $. 
For instance, the following pair $(t,A) $ is of type $ (1,3,1)$: 
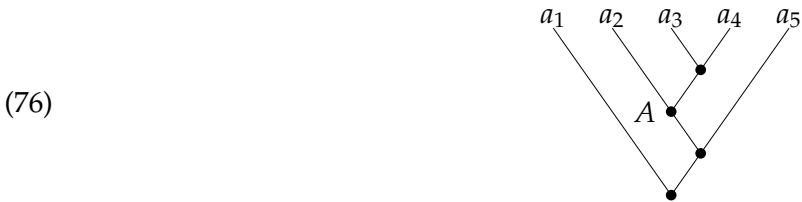
\begin{equation}
\label{eq:tree131}
    \adjustbox{valign=c}{\scalebox{0.5}{
\begin{forest}
for tree = {grow' = 90}, nice empty nodes, for tree={ inner sep= 0 pt, s sep= 0 pt, 
},
[
[\scalebox{2}{$a_1$}, tier =1]
[ 
[, ,edge label={node[left]{\scalebox{2}{$A\,\,$}}},
[\scalebox{2}{$a_2$}, tier = 1]
[
[\scalebox{2}{$a_3$}, tier = 1]
[\scalebox{2}{$a_4$}, tier = 1]
]
]
[\scalebox{2}{$a_5$}, tier =1]
]
]
\path[fill=black] (.parent anchor) circle[radius=4pt]
(!2.child anchor) circle[radius=4pt]
(!21.child anchor) circle[radius=4pt]
(!212.child anchor) circle[radius=4pt];
\end{forest}}}
\end{equation}
since the leaves that descend from $A$ are $(a_2,a_3,a_4) $.
The summation that appears in Equation \eqref{eq:TYn} can be rearranged as:
$$
\sum_{t\in Y_n}
\sum_{\scalebox{0.5}{$\begin{array}{r}A \in \mathrm{Ver}(t)\end{array}$}} = 
\sum_{\scalebox{0.5}{$\begin{array}{c}i,k \geq 0 , j \geq 1 \\
i+j+k=n
\end{array}$}} \sum_{\scalebox{0.5}{$\begin{array}{c}(t,A)   \hbox{ of } \\ \hbox{ type }  (i,j,k)\end{array}$}} .
 $$
Equation \eqref{eq:TYn} then becomes: 
\begin{equation}
\label{eq:TYn2}
\sum_{\scalebox{0.5}{$\begin{array}{c}i,k \geq 0 , j \geq 1 \\
i+j+k=n
\end{array}$}} \sum_{\scalebox{0.5}{$\begin{array}{c}(t,A)   \hbox{ of } \\ \hbox{ type }  (i,j,k)\end{array}$}}
     (-1)^{P(t[a])+W_A} \psi_{t_{\downarrow A}}(a_1, \dots, a_i, \psi_{t_{\uparrow A}}(a_{i+1}, \dots,a_{i+j}),a_{i+j+1}, \dots, a_n)=0.
\end{equation}
\vspace{-2mm}


\noindent
 There is a natural bijection  \begin{equation}\label{eq:treestrees}  \left\{ \hbox{ Pairs $ (t,A)$ of type $ (i,j,k)$}\right\}
\, \simeq  \,  Y_{n-j+1} \times Y_{j} 
 .\end{equation}
 This bijection maps $ (t,A)$ to the pair $(t_{\downarrow A} , t_{\uparrow A}) $. We recall that when $A$ is a leaf,  $t_{\uparrow A} $ is the trivial tree, and when it is the root $t_{\downarrow A} $ is the trivial tree. Its inverse maps $ (t_1,t_2) \in Y_{n-j+1 \times Y_i}  $ to the pair $ (t,A)$ with $ t \in Y_n$ and  $A $ an inner vertex, a root, or a leaf of $t$ such that 
$t_1=t_{\downarrow A}  $ and $ t_2 = t_{\uparrow A}$. For instance, the pair $(t,A) $ of type $(1,3,1) $ described in Equation \eqref{eq:tree131} corresponds to the pair $ (\rtree,\rtree)\in Y_3 \times Y_3$. 
In view of this bijection, the identity \eqref{eq:meaningmu} given in
 Lemma \ref{lem:aboutmuexplicit} becomes
\begin{equation}
\label{eq:step2}
\begin{array}{ll}&  (-1)^{i+jk}\mu_{n-j+1}(\id^{\otimes i}\otimes \mu_{j} \otimes \id^{\otimes k})  (a_1, \dots, a_n) 
\\=&
\sum\limits_{\scalebox{0.5}{$\begin{array}{c}(t,A)   \hbox{ of } \\ \hbox{ type }   (i,j,k)\end{array}$}}\epsilon_{t_{\downarrow A},t_{ \uparrow A}}  \, \psi_{t_{\downarrow A}}(a_1, \dots, a_i, \psi_{t_{\uparrow A}}(a_{i+1}, \dots, a_{i+j}),  a_{i+j+1},\dots, a_n). 
\end{array}
\end{equation}
Now one checks  that $ \epsilon_{t_{\downarrow A},t_{\uparrow A}}= \eta \, (-1)^{P(t[a])+W_A}$. 
in view of the definition of $ \epsilon_{t_{\downarrow A},t_{\uparrow A}} $ given in Equation \eqref{eq:defepsilon} and Lemma \ref{lem:explicitrelations}. 
As a consequence, Equation \eqref{eq:step2} becomes
\begin{equation}
\label{eq:step1}
\begin{array}{ll}  & (-1)^{i+jk}\mu_{n-j+1}(\id^{\otimes i}\otimes \mu_{j} \otimes \id^{\otimes k})  (a_1, \dots, a_n) 
\\=&
\eta \sum\limits_{\scalebox{0.5}{$\begin{array}{c}(t,A)   \hbox{ of } \\ \hbox{ type }   (i,j,k)\end{array}$}} (-1)^{P(t[a])+W_A}  \psi_{t_{\downarrow A}}(a_1, \dots, a_i, \psi_{t_{\uparrow A}}(a_{i+1}, \dots, a_{i+j}), a_{i+j+1},\dots, a_n)
\end{array}\end{equation}
In view of Equation \eqref{eq:TYn2}, this indeed yields  
the $n$-th $A_\infty $-axiom Equation \eqref{eq:ainfty2}.

\vspace{1mm}
\noindent
We are left with the task of checking the axiom defining a $C_\infty$-algebra, i.e.  Equation \eqref{eq:Cinfty2}.
In view of the definition of $\mu_n $ given in Theorem \ref{thm:ainfty2},  it suffices to show that for all $ i \leq n$ and all homogeneous $ a=(a_1, \dots,a_n) \in \FMb^n$:
 $$ \sum_{\sigma \in {\mathrm{Sh}}(i,n)} \epsilon(\sigma,a)  k_n[a_{\sigma}]= 0$$
 lies in the kernel of the projection $ {\mathrm{pr}}$ from ordered decorated trees to symetrized decorated trees. 
  Above, $ a_{\sigma} $ is an abbreviation for $a_{\sigma(1)}, \dots, a_{\sigma(n)}$ and $\epsilon(\sigma,a)$ is the Koszul sign of the permutation $ \sigma$ with the convention that the degree of $a_i $ is  $ |a_i|+1$. 
We therefore have to prove that
 \begin{equation}\label{label:forCinfty}
 \sum\limits_{(\sigma,t) \in {\mathrm{Sh}}(i,n) \times Y_n} \epsilon(\sigma,a)(-1)^{P(t[a_\sigma])}  {\mathrm{pr}}( t[a_{\sigma}])= 0
 \end{equation}

 \vspace{1mm}
 \noindent
 We proceed as follows. We say that a subtree of $t[ a_{\sigma(1)}, \dots, a_{\sigma(n)}] $ is of \emph{$(\leq i)$-type} if all its leaves are in $\{a_1, \dots,a_i \} $ and of \emph{($> i)$-type} if all its leaves are in $\{a_{i+1}, \dots,a_{n} \} $.
 We say that a vertex $A$ in an ordered decorated tree $t[a_{\sigma(1)}, \dots, a_{\sigma(n)} ] $ is of \emph{mixed type} if either its left subtree is of $ (\leq i) $-type and its right subtree is of $ (>i) $-type, or its left subtree is of $(>i)$- type and its right subtree is of $(\leq i)$-type. For every $(\sigma,t) \in {\mathrm{Sh}}(i,n)\times Y_n $,  there exists at least one vertex of mixed type.  We say that two pairs $(t,\sigma), (t',\sigma') \in {\mathrm{Sh}}(i,n) \times Y_n$ with leftmost mixed vertices $A$ and $A'$
 are \emph{dual} if the right subtree of $A$  is the left subtree $\ell(A',t'[a_{\sigma'}])$ of $A'$, the left subtree $\ell(A,t[a_\sigma])$ of $A$  is the right subtree of $A'$, while the rest of the trees $ t[a_\sigma]$ and $t'[{a_{\sigma'}}]$ coincide.

 \vspace{1mm}
 \noindent
 Let us study several properties of dual pairs.
In view of the equivalence relation \eqref{eq:quotientEquiv} whose equivalence classes define the projection $ {\mathrm{pr}}$, we have 
   $$ {\mathrm{pr}}(t[a_\sigma]) =  (-1)^{|\ell(A,t[a_\sigma])||\ell(A',t'[a_{\sigma'}])|} {\mathrm{pr}}(t'[a_{\sigma'}]). $$
 The decorated trees $t[a] $ and $ t'[a_{\sigma'}]$ have left subtrees with exactly the same degrees, except for the left subtrees of the vertices $A$ and $A'$, which are $\ell(A,t[a_\sigma]$ and $\ell(A',t'[a_{\sigma'}]$, respectively. Hence:  
 $$ 
 \begin{array}{rcl} P(t[a_\sigma])+|\ell(A,t[a_\sigma])|&=&P(t'[a_{\sigma'}]) +| \ell(A',t'[a_{\sigma'}])|.  
 \end{array} 
 $$
 By definition of the Koszul sign
  $$ \epsilon(\sigma',a)=\epsilon(\sigma,a)(-1)^{ (\sum_{r \in I}  |a_r|-|I_l|)(\sum_{r \in I'}  |a_r|-|I'|)}  = \epsilon(\sigma) (-1)^{(|\ell(A,t[a_{\sigma}])|+1)(|\ell(A',t'[a_{\sigma'}])|) +1) }.$$
  Here $I,I'\subset \{1, \dots,n\}$ are the indices of the leaves of $ \ell(A,t[a_{\sigma}])$ and $\ell(A',t'[a_{\sigma'}]) $, respectively. 
 The three relations above imply that the contribution of a pair and its dual pair  come with opposite signs in the left hand side Equation \eqref{label:forCinfty}.
 Since any pair in $({\mathrm{Sh}}(i,n) \times Y_n) $ admits one and only one dual pair, Equation \eqref{label:forCinfty} holds. This completes the proof of Theorem \ref{thm:ainfty2}.
 
\end{proof}

\section{Ideals invariant under a finite group action}
\label{sec:finitegroups}

\noindent Consider an algebra  $ \CO$ on which a finite group $G$ acts by algebra automorphisms, and let $ \CI \subset \CO$ be a $G$-invariant ideal.
It is natural to look for \emph{$G$-equivariant Koszul-Tate resolutions} $(S(\CE),\delta) $. This means that we require an action  of $ G$ by DGCA automorphisms of $(S(\CE),\delta) $ whose restriction to elements of degree $0$, which are isomorphic to $ \CO$, coincides with the initial $G$-action.

\vspace{1mm}\noindent 
A pleasant feature of the Tate algorithm presented in Construction \ref{const:tate.algorithm} is that it can be adapted to yield such a  $G$-equivariant Koszul-Tate resolution: it suffices to average at every step (despite the fact that there are infinitely many such steps in general). We prove that the arborescent Koszul-Tate resolution offers the same advantage upon averaging at every step (and there are finitely many such steps in many situations, see Section \ref{sec:complexity}).

\begin{proposition}
Let $G$ be a finite group acting on an algebra $\mathcal O $ over $\mathbb R$ or $\mathbb C $ and $\mathcal I $ a $G $-invariant ideal. There exists a $G$-equivariant arborescent Koszul-Tate resolution.
\end{proposition}
\begin{proof}
The proof goes in two steps. 

\vspace{1mm}
\noindent
{\textbf{Step 1.}} Let us show that a $G$-invariant ideal $\CI $ admits a $G$-equivariant free $\CO $-module resolution. 

\vspace{0.1cm}

\noindent
We denote by $(g,F) \mapsto g\triangleright F$ the action of $g \in G$ on $F \in  \CO $.
Since $G$ acts on $\CO $, it acts on any free $\CO $-module.
Let $\FM_1 $ be a free $\CO $-module and $d \colon \FM_1 \to \CO $ an $\CO$-module morphism whose image is $ \CI$. In general, $d$ is not $G$-equivariant. Consider the tensor product 
$$ \FM_{1}' := \FM_{1} \otimes_{\mathbb K} \mathbb K [G]  $$
with $\mathbb K [G]$ being the free vector space generated by the group $G  $.  There is natural $G$-action on this space by 
 $$   h \triangleright (a \otimes g) := (h \triangleright a) \otimes hg \hbox{ for all $a \in \FM_1, g,h \in G$ } . $$
With respect to this $G$-action, the $\CO$-linear morphism \begin{equation} \label{eq:equivD} d' ( a \otimes g) = g \triangleright d  \left(  g^{-1}  \triangleright a \right) \hbox{ for all $a \in \FM_{1}, g \in G$} \end{equation}
is $G$-equivariant. Its image is still $\CI $. Since $d$ is $G$-equivariant, its kernel is preserved under the $G$-action, and for all $g \in G$ 
$$  a \mapsto g \triangleright a $$
is an invertible endomorphism of ${\mathrm{Ker}}(d) $.
Let $\FM_2 $ be a free $\CO$-module equipped with a surjective $\CO$-linear morphism $\FM_2 \longrightarrow {\mathrm{Ker}}(d) $, then consider  
$$ \FM_{2}' := \FM_{2} \otimes_{\K} \mathbb K [G] . $$ The map $d'$ constructed as in \eqref{eq:equivD} is $G$-equivariant and its image is ${\mathrm{Ker}}(d)$, so that
$$ \xymatrix{\FM_{2}' \ar[r]^{d'} & \FM_{1}' \ar[r]^{d'} & \CO \ar[r] & \CO/\CI}  $$
are the first terms of a $G$-equivariant free $\CO$-module resolution of $\CI$.
The procedure continues by recursion.

\vspace{.5cm}

\noindent
{\textbf{Step 2.}}
Let $(\FM_\bullet,d') $ be an $K$-equivariant free $\CO$-module resolution as obtained out of Step 1. In the construction of its arborescent Koszul-Tate resolution,
all operations described in Section \ref{sec:arborescent} are $G$-equivariant except maybe those that involve the arborescent operations 
 $t \mapsto \psi_t  $, so that $\delta_\psi $ will be $G$-equivariant if the maps $ \psi_t $ can be chosen to be $G $-equivariant.

\vspace{0.1cm}

\noindent
 We prove by recursion on the degree $k$ that such choices exist. Equation \eqref{eq:deltaPsiSquare} gives the recursion relation that an  arborescent operation $ \psi$ needs to satisfy in order to have $ \delta_\psi^2=0$: 
 
 \begin{equation}
\label{eq:psi30}
 \begin{gathered}
     d' \circ \psi_t (a_1, \dots, a_n) 
 \, = \\ \,\left(  \sum_{A \in InnVer(t)} \left(\psi_{t_{\downarrow A}} \circ_{A} \psi_ {t_{\uparrow A}} - (-1)^{W_A}\psi_{\partial_A t}\right) -\sum_{i=1}^n(-1)^{W_i}\psi_t(a_1, \dots, d'a_i, \dots,a_n) + \pi^1\circ d'\right) (a_1\otimes \dots \otimes a_n).
 \end{gathered}
 \end{equation}
 More precisely, while proving Theorem \ref{thm:isKT} (and more precisely Corollary \ref{cor:differentialexists}), we showed that $ \delta_\psi^2=0$ if and only if the $ \psi_t$ are recursively constructed so that they satisfy \eqref{eq:psi30}.
 If one assumes that the assignment $(a_1, \dots, a_n) \mapsto \psi_t[a_1, \dots, a_n] $ is $G$-equivariant for all tree $t$
 and all $ a_1, \dots,a_n \FM$ such that the sum $ |t|+|a_1|+ \dots + |a_n|$ of all degrees is $\leq k $, then the right hand side of the previous equation is $G$-equivariant. Since $d'$ is $G$-equivariant, one may redefine  $ \psi_t(a_1, \dots,a_n)$
and all trees $t$ and homogeneous elements $ a_1, \dots,a_n \in \FM_\bullet$ whose degrees add up to $ k+1$ by
 $$ \tilde{\psi}_t[a_1, \dots,a_n] :=  \frac{1}{|K|} \sum_{g \in G} g^{-1} \triangleright \psi_{t} ( g \triangleright a_1, \dots,  g \triangleright a_n ) $$
 and Equation \eqref{eq:psi3} still holds with this new definition.

\vspace{1mm}
\noindent The construction of $G$-equivariant arborescent operations $t \mapsto \psi_t $ then continues by recursion. 
\end{proof}

\section{The reduced complex of the arborescent Koszul-Tate resolution}
\label{app:reduced}

\vspace{0.1cm}

\noindent
We prove in this Appendix the following fact : any two Koszul-Tate resolutions are homotopy equivalent in a unique up to homotopy manner. Here, "homotopy" is to be taken  in the sense of  $Q$-manifolds, see e.g. Section 3.4.3 in \cite{Linf}.
This result already appeared in \cite{Kazdhan-Felder}, but our description is more precise about the properties of this homotopy equivalence.
We deduce from that result that the reduced complex we introduce below is a well-defined notion.

\vspace{0.1cm}

\noindent
Let us adapt to the present context the definitions in \cite{Linf} about homotopy equivalence of morphisms.
Denote by $(\mathbb K[t,dt],d) $ the complex of polynomial exterior forms on $\mathbb K $. Here, it must be understood that the degree of $t$ is zero, and that the degree of $ dt$ is $-1$, so that the degree of the de Rham differential $d$ is $ -1$. 
For $\CEb, \CEb' $ graded $\CO $-modules, consider a graded commutative algebra morphism 
 $$  \mathcal H \colon S(\CEb') \to S(\CEb) \otimes_\mathbb K  \mathbb K[t,dt]  . $$ 
 Recall that for any vector space  $V$ a linear map   valued in $ V \otimes_\mathbb K \mathbb K[t]  $
 can be equivalently seen as a family $ \psi_t$ of maps valued in $V$ and depending on a parameter $ t \in \mathbb K$.
 Therefore, in view of the decomposition
 $$ S(\CEb) \otimes_\mathbb K \mathbb K[t,dt] = S(\CEb) \otimes_\mathbb K \mathbb K[t] \oplus S(\CEb) \otimes_\mathbb K \mathbb K[t] \, dt ,$$
 $ \mathcal H$ may be seen at a family $ \mathcal H_t$ depending on $t \in \mathbb K$ of graded algebra morphisms from 
 $ S(\CEb)   $ to $ S(\CEb) \oplus S(\CEb) 
 {dt}$, where $dt$ is seen as a degree $ -1$ element.
 Moreover, $\mathcal H_t $ decomposes as a sum
 \begin{equation}\label{eq:Ht} \mathcal H_t = \Phi_t + \Delta_t  \, dt\end{equation}
 where
\begin{enumerate}
\item  $\Phi_t $ is a family depending on $ t\in \mathbb K$ of graded commutative algebra morphisms $S(\CEb')\to S(\CEb) $ of degree $0$.
 \item and 
 $\Delta_t \colon  S( \CEb')\to S(\CEb) $ is a family  depending on $ t\in \mathbb K$
 of $\Phi_t $-derivations of degree $+1 $.
\end{enumerate}
Assume now that $ S(\CEb')$ and $S(\CEb) $ are equipped with degree $-1 $ differentials $\delta$ and $\delta' $ respectively. Then $S(\CEb) \otimes   \mathbb K[t,dt] $ comes equipped with the differential 
$$ \delta \otimes_\mathbb K {\mathrm{id}} +{\mathrm{id}} \otimes_\mathbb K  dt \frac{\partial}{\partial t}   .$$
Here, $  dt \frac{\partial}{\partial t} $ stands for the de Rham differential in one variable
 $$    dt \frac{\partial}{\partial t} P(t) = P'(t)  \,dt \hbox{ and } dt \frac{\partial}{\partial t} (P(t) \, dt )=0  $$
 for any $ P(t) \in \mathbb K[t]$.
If the algebra morphism $\mathcal H $ is a also a chain map, then $\Phi_t $ is a morphism of differential graded commutative algebras and:
 \begin{equation} \label{eq:Phitderivations}
 -\frac{\partial \Phi_t}{ \partial t} = \delta \circ \Delta_t  + \Delta_t \circ \delta'.
 \end{equation}
$\mathcal H $ is called a \emph{homotopy} between the differential graded algebra morphism $\Phi_0 $ and $\Phi_1 $.

 \begin{remark}
 \normalfont
 \label{rmk:aboutgreaterthan2}
Integrating Equation \eqref{eq:Phitderivations} between $0$ and $1$ shows that $\Phi_0 $ and $ \Phi_1$ are homotopic as chain maps with homotopy $\int_0^1 \Delta_t dt $. Notice that for degree reasons, $ \Delta_t$ maps $\CEb' $ to $S^{\geq 1} (\CEb) $ and therefore maps $ S^{\geq 2} (\CEb') $ to $S^{\geq 2} (\CEb) $.
Hence, the same holds for the homotopy $\int_0^1 \Delta_t dt $.
\end{remark}

\begin{definition}
\label{def:hom.equiv}
    Two differential graded commutative algebras $(\CA, \delta)$ and $(\CA', \delta')$ are said  to be homotopy equivalent, if there exist two DGA morphisms $\Phi: (\CA, \delta) \longrightarrow (\CA', \delta')$ and $\Phi': (\CA', \delta') \longrightarrow (\CA, \delta)$ such that the compositions $\Phi'\circ \Phi$ and $\Phi'\circ \Phi$ are homotopic to the identity map.
\end{definition}

\vspace{0.1cm}

\noindent
We start with a lemma.
\begin{lemma}
\label{lem:tensorStillKT}
 Let $(S(\CEb), \delta)$ be a Koszul-Tate resolution of $\CO/\CI$. 
 The homology of $$\left(S(\CEb)    \otimes \mathbb K[t,dt], \delta \otimes id + id \otimesk dt \frac{\partial}{\partial t} \right) $$
 is zero in every degree except in degree $ 0$ where it is $ \CO/\CI$.
\end{lemma}
\begin{proof}
The complex computing this homology has two lines:
$$ \xymatrix{ \cdots\ar[rr]^{\delta \otimesk {\mathrm{id}}} && S(\CEb)_{k+1} \otimesk \mathbb K[t] dt \ar[r]^{\delta \otimesk {\mathrm{id}}}& S(\CEb)_{k} \otimesk \mathbb K[t] dt \ar[r]^{\delta \otimesk {\mathrm{id}}}& S(\CEb)_{k-1} \otimesk \mathbb K[t] dt  \ar[rr]^{\delta \otimesk {\mathrm{id}}}&&\cdots \\    \cdots\ar[rr]^{\delta \otimesk {\mathrm{id}}}&& \ar[r]^{\delta \otimesk {\mathrm{id}}}S(\CEb)_{k+1} \otimesk \mathbb K[t]   \ar[r]^{\delta \otimesk {\mathrm{id}}}\ar[u]^{{\mathrm{id}} \otimesk dt \frac{\partial }{ \partial t}}&S(\CEb)_k \otimesk \mathbb K[t] \ar[r]^{\delta \otimesk {\mathrm{id}}}\ar[u]^{{\mathrm{id}} \otimesk dt \frac{\partial }{ \partial t}}& S(\CEb)_{k-1} \otimesk \mathbb K[t] \ar[u]^{{\mathrm{id}} \otimesk dt \frac{\partial }{ \partial t}}  \ar[rr]^{\delta \otimesk {\mathrm{id}}}&&\cdots\\  }  $$
By assumption, horizontal lines have no homology except in degree $0$ where it is $\CO/\CI \otimesk \mathbb K[t] dt $ and  $\CO/\CI \otimesk \mathbb K[t] $ for the upper and lower lines respectively. The $k$-th column has homology equal to $S(\CEb)\otimesk \mathbb K $ concentrated in the lower line. The result follows by a simple diagram chasing argument.
\end{proof}

\vspace{0.1cm}

\noindent
Lemma \ref{lem:tensorStillKT} allows to prove that Koszul-Tate resolutions of $\CO/\CI $ are terminal objects in the category where objects are symmetric algebras over a free or projective $\mathbb N $-graded $\CO$-module $ \CEb$ equipped with a differential such that $\delta (\mathcal E_{1}) \subset \CI$, and arrows are homotopy classes of differential graded commutative algebras. Let us state it in a less abstract manner:

\begin{lemma}
\label{lem:dga.morphisms}
     Let $(S(\CEb), \delta)$ be a Koszul-Tate resolution of $\CO/\CI$.
     Let $ (S(\CEb'), \delta')$ be  differential graded algebra,
     with $  ({\CE_{i}'})_{i \geq 1} $ a sequence of projective $\CO $-modules.
     If $\delta'(\CE_1') \subset \CI$, then there exists a differential graded commutative algebra morphism $\Phi^{\bullet}: (S(\CEb'), \delta') \longrightarrow (S(\CEb), \delta)$ and any two such morphisms are homotopic.
\end{lemma}
\begin{proof}
By assumption, $ \delta' (\CE_1') = \CI$ and $ \delta(\CE_1) \subset \CI$. Since $ \CE_1'$ is a projective $ \CO$-module, there exists a $\CO $-linear map $f^1$ making the following diagramm commutative
$$
\begin{tikzcd}[row sep=2.5em, column sep=2.5em]
   \CE_1' \arrow[r, "\delta"] \arrow[d, "f_1"] & \CO  \arrow[d, "\id"] \\ \CE_1 \arrow[r, "\delta'"] & \CO   
\end{tikzcd}
$$
We then construct $ \Phi$ by recursion out of a family of degree $0$ $ \CO$-linear maps:
 $$ f^n \colon \CE_n' \longrightarrow S(\CEb)_n  .$$
Assume $f^1, \dots, f^n $  are constructed such that the graded algebra morphism $ \Phi^n$, whose restriction to $ \CE_i'$ is $f^i$ for $ i \geq n$ and $0$ otherwise, is a chain map when applied to elements of degree $0$ to $n$. 
Then $ \Phi^n \circ \delta' -  \delta \circ \Phi^n $ is a $\Phi^n $-derivation whose restriction to $\CE_{n+1} $ is valued in the kernel of $\delta $. Since the homology of $\delta $ is zero by assumption, there exists  $ f^{n+1} \colon \CE_{n+1} \to S(\CEb)_{n+1}$ such that 
 $$   \Phi^n \circ \delta' -  \delta \circ \Phi^n  = \delta \circ f^{n+1}  . $$
 It is routine to check that 
the graded algebra morphism $ \Phi^{n+1}$ whose restriction to $ \CE_i'$ is $f^i$ for $ i \leq n+1$ and $0$ otherwise, is a chain map when applied to elements of degree $0$ to $n+1$. The graded algebra morphism $ \Phi$ whose restriction to $ \CE_i'$ is $f^i$ for $ i \in \mathbb N$ is a chain map. This proves the first part of the Lemma.

\vspace{0.1cm}
\noindent
Let $ \Phi_0,\Phi_1 \colon (S(\CEb'),\delta') \to (S(\CEb),\delta)$
be two such morphisms, coming from sequences $ (f_0^i)_{i \geq 1}  $
and $ (f_ 1^i)_{i \geq 1}$ respectively. 
Let us assume that $ \Phi_0,\Phi_1$ are homotopic up to a certain order $n$, i.e. that we are given a graded algebra 
morphism 
$ \mathcal H^{n} \colon S(\CEb') \to S(\CEb) \otimesk \mathbb K[t,dt]$ which is a chain map when applied to elements of degrees ranging from $0$ to $n$, and such that in the decomposition as in Equation \eqref{eq:Ht}
\begin{equation}
 \label{eq:HtFor_n}
 \mathcal H^{n}_t = \Phi_t^n + \Delta_t^n dt  
 \end{equation}
 we have  $ \Phi_0^n=\Phi_0$ and $\Phi_1^n = \Phi_1 $. For $n$ = 0 we can choose $\mathcal H^0_t = (1-t)\Phi_0 + t\Phi_1$.
 Let us construct $\mathcal H^{n+1} $, making use of the decomposition:
$$   \mathcal H^{n+1}_t = \Phi_t^{n+1} + \Delta_t^{n+1} dt  .$$
The values of $ \Phi_t^{n+1}$ at $t=0 $ and $ t=1$ are imposed to be $ \Phi_0^{n+1}$ and $ \Phi_1^{n+1}$ respectively.  
Consider the failure of $\mathcal H^n $ to be a chain map, when applied to elements of degree $n+1 $. Let  $g^{n+1}$ be the restriction to $ S(\CEb')_{n+1}$ of
 $$ \left( \delta \otimes \id + \id \otimes \mathrm{dt} \frac{\partial}{\partial t} \right) \circ {\mathcal H}^n  -{\mathcal H}^n \circ \delta'  .$$
 By assumption, $g^{n+1}$ is zero on $(S^{\geq 2}\CEb')_{n+1}$. Moreover, its component in  $S(\CEb) \otimes \mathbb K[t]$ is $0$ for $t=0,1$.  The image of $ g^{n+1}$ lies inside the kernel of $\delta \otimes \id + \id \otimes dt \frac{\partial}{\partial t} $.
By lemma \ref{lem:tensorStillKT}, $ g^{n+1}$ is therefore valued in the image of 
$$ \delta \otimes \id + \id \otimes dt \frac{\partial}{\partial t}  \colon (S(\CEb) \otimes \mathbb K[t,dt])_{n+1}  \to (S(\CEb) \otimes \mathbb K[t,dt])_{n}  .$$
Let us view $g^{n+1}$ as a map $\CE'_{n+1} \longrightarrow (S(\CEb) \otimes \mathbb K[t,dt])_{n}$.
Since $ \CE_{n+1}'$ is a projective module, there exists an $ \CO$-linear map $$\hat h^{n+1} \colon
 \CE_{n+1}' \to (S(\CEb) \otimes \mathbb K[t,dt])_{n+1}   $$ 
such that 
$$
g^{n+1} = -\left(\delta \otimes \id + \id \otimes dt \frac{\partial}{\partial t} \right) \circ \hat h^{n+1} .
$$
Let us write $\hat  h^{n+1}$, which is valued in $S(\CEb)\otimesk \mathbb K[t,dt] $, as a decomposition 
 $$ \hat h^{n+1} = \alpha_t^{n+1} + \beta^{n+1}_t dt .$$
Since $\delta \alpha_t^{n+1} = 0$ for $t = 0, 1$, by adding an appropriate boundary term, we modify $\hat h^{n+1}$ in such a way that $\alpha_t^{n+1} = 0$ for $t =0, 1$. 

\vspace{1mm}
\noindent
Define a map $h^{n+1}\colon \CEb' \longrightarrow (S(\CEb) \otimes \mathbb K[t,dt])$
$$
h^{n+1}(a) = \begin{cases}
    \mathcal{H}^n(a),\hbox{ if } a\notin \CE'_{n+1} \\
    \mathcal{H}^n(a) + \hat h^{n+1}(a), \hbox{ if } a\in \CE'_{n+1}.
\end{cases}
$$

\vspace{1mm}
\noindent
Then $\mathcal H^{n+1}$ is defined as the graded commutative algebra morphism that coincides with $h^{n+1} $ on $ \CEb'$. By construction $\mathcal H^{n+1}$ is a chain map when applied to elements of degrees ranging from $0$ to $n+1$.
\end{proof}

\vspace{0.1cm}
\noindent
As an immediate consequence of Lemma \ref{lem:dga.morphisms},
we recover the following result from \cite{Kazdhan-Felder} - with additional clarification about the ideals  $S^{\geq2 } \CEb \oplus \CI   \CEb  $ and $S^{\geq 2} \CEb' \oplus \CI  \CEb' $.

\begin{proposition}
\label{prop:homotopicAlways}
     Two Koszul-Tate resolutions $(S(\CEb), \delta)$ and $(S(\CEb'), \delta')$ of $ \CO/\CI$   are homotopy equivalent in a unique up to homotopy manner:
\vspace{-9mm} 
\begin{equation}\label{eq:existHomotopy}
\begin{tikzcd}[column sep = 4em]
(S(\CEb), \delta) \arrow[r, bend left=20, "\Psi", shift ={(0 ,1mm)}]  \arrow[out=225, in=135, looseness=8, loop, "\mathcal H"] & (S(\CEb'), \delta') \arrow[l, bend left=20, swap, "\Psi'", shift = {(0, -1mm)}]\arrow[out=-45, in=45, looseness=5, loop,swap, "\mathcal H'"]
\end{tikzcd}
\end{equation}
\vspace{-9mm}

\noindent
     where $\Psi, \Psi' $ and each one of the components as in \eqref{eq:Ht} of the homotopies $\mathcal H_t=\Phi_t + \Delta_t \, dt , \mathcal H_t' =\Phi_t' + \Delta_t '\, dt $ preserve the ideals $S^{\geq2 } \CEb \oplus \CI   \CEb  $ and $S^{\geq 2} \CEb' \oplus \CI  \CEb' $.
\end{proposition}
\begin{proof}
Lemma \ref{lem:dga.morphisms}
implies the existence and uniqueness up to homotopy of morphisms of differential graded commutative algebras as in Equation \eqref{eq:existHomotopy}.
Remark \ref{rmk:aboutgreaterthan2} gives the conditions about the ideals  $ S^{\geq 2} \CEb \oplus \CI\CEb  $ and $S^{\geq 2} \CEb' \oplus \CI  \CEb' $. 
\end{proof}

\vspace{0.1cm}

\noindent
Having found such a homotopy we can make a following quotient construction: consider the quotient of a Koszul-Tate resolution $(S(\CEb), \delta)$ of $\CO/\CI $ by $\CI\CEb\oplus \Sdeg(\CEb)$.
This quotient is isomorphic to the tensor product $\CE_\bullet \otimes_{\CO} \CO/\CI$ as an $\CO$-module. 
Since $\delta(\CI\CEb\oplus \Sdeg(\CEb)) \subset \CI\CEb\oplus \Sdeg(\CEb)$,  the quotient comes equipped with an induced differential $ \underline{\delta}$.

\begin{definition}
    For a Koszul-Tate resolution $(S(\CEb), \delta)$ of $\CO/\CI$, the complex of $ \CO$-modules $(\CEb\otimes \CO/\CI, \underline{\delta})$ is called a reduced complex of $\CO/\CI$.
\end{definition}

\vspace{0.1cm}

\noindent
The following is an immediate corollary of Proposition \ref{prop:homotopicAlways}:

\begin{corollary}
\label{cor:reduced}
Two reduced complexes associated to two Koszul-Tate resolutions of $\CO/\CI $ are homotopy equivalent  and there is distinguished class of homotopies between them.
In particular, its homology is canonically attached to the ideal $ \CI \subset \CO$. 
\end{corollary}

\vspace{0.1cm}

\noindent
Consider an ideal $\CJ \subset \CO $ containing $ \CI$. The tensor product 
$(\CE_\bullet \otimes \CO/\CJ, \underline \delta)$ 
of a reduced complex  of $ \CI $ with $ \CO/\CJ$
 is a complex of $ \CO$-modules, on which $ \underline{\delta}$ induces a differential that we denote again by $ \underline{\delta}$.

 \begin{definition}
    For a Koszul-Tate resolution $(S(\CEb), \delta)$ of $\CO/\CI$ and an ideal $\CJ \subset \CO $ containing $ \CI$, the complex of $ \CO$-modules $(\CEb\otimes \CO/\CJ, \underline{\delta})$ is called a reduced complex of $\CO/\CI$ evaluated at $ \CJ$.
\end{definition}

 \noindent
 Corollary \ref{cor:reduced} implies that
the homology of this complex does not depend on the choice of a Koszul-Tate resolution of $ \CO/\CI$. 

\begin{corollary}
\label{cor:reducedWIthJ}
Two reduced complex of $\CO/\CI$ evaluated at $ \CJ$ associated to two Koszul-Tate resolutions of $\CO/\CI $ are homotopy equivalent  and there is a distinguished class of homotopies between them.
In particular, its homology is canonically attached to the ideals $ \CI \subset \CJ \subset \CO$. 
\end{corollary}
\vspace{0cm}
\noindent
Let $\CJ$ be a maximal ideal of $\CO$. Since $\CO/\CJ \simeq \mathbb K $ is a vector space, any reduced complex of $\CO/\CI$ evaluated at $ \CJ$ is a complex of vector spaces over $\mathbb K $.  The homology of the reduced complex of $ \CO/\CI$ evaluated at $ \CJ$ is a vector space as well. We denote by $ b_i(\CO/\CI,\CJ) \in \mathbb N \cup \{+\infty\}$ the dimension of the homology in degree $i \in \mathbb N$. By the above discussion, $b_i(\CO/\CI,\CJ)$ does not depend on the choice of a Koszul-Tate resolution, hence the notation. 
If the free $\CO $-module $\CE_i  $ is indexed by a set $ S_i$, then the chains in degree $i$ of the reduced complex of $ \CO/\CI$ evaluated at $ \CJ$ form a vector space indexed by the same set $S_i$. The $i$-th homology being a vector space of smaller dimension, the next statement follows:

\begin{theorem}
\label{thm:lowernboundonrank}
For any Koszul-Tate resolution $ (S(\CEb),\delta)$ of $ \CO/\CI$, any maximal ideal $ \CJ$ containing $ \CI$, and any $ i \in \mathbb N$, we have 
${\mathrm{rk}}(\CE_i) \geq b_i(\CO/\CI,\CJ)$. 
\end{theorem}

\section{Minimal Koszul-Tate resolutions}
\label{app:minimal}

\vspace{0.1cm}

\noindent
Although the arborescent Koszul-Tate resolutions have the advantages listed in the introduction, they may  not be the "smallest" Koszul-Tate resolution.  For instance, for $\CI \subset \CO$ a complete intersection with $r$ generators, the Koszul $\CO$-module resolution $((\FM_i = \wedge^i \CO^r)_{i \geq 1},d = \sum_{i=1}^{n}\phi_i \iota_{e_i}) $
of section \ref{sec:Koszul} is by construction an $\CO$-module resolution with finitely many generators. One can therefore consider the associated arborescent Koszul-Tate resolution. The latter is of course a huge symmetric algebra: for $ r \geq 2$, it has infinitely many generators.

\vspace{0.2cm}
\noindent
Theorem \ref{thm:lowernboundonrank} gives a lower bound to the possible ranks of the free $\CO$-modules $ \CEb$ entering a Koszul-Tate resolution of $ \CO/\CI$.   
In this appendix, we give a context where this lower bound can be reached.
We consider two distinct cases, which may seem restrictive but quite a few examples exist, see section~\ref{sec:basicExamples}: 
\begin{itemize}
    \item[ ] Case i. $\CO$ is a local Noetherian ring with maximal ideal $\CJ$ and residue field $\mathbb K$.
    \item[ ] Case ii. $\CO$ is the polynomial ring $\mathbb K[x_1, \dots, x_n]$, $\CJ$ is the maximal ideal $\CJ:=\langle x_1, \dots, x_n \rangle$ and  $\CI$ is an ideal with homeogeneous generators.
\end{itemize}
Here is a classical theorem, see e.g. \cite{Eisenbud} for case i and  Theorem 2.12 of \cite{peeva2010graded} for case ii. Let $\CJ$ be a maximal ideal containing $\CI $.

\begin{theorem}
\label{thm:basis.theorem}
    Let $M$ be a finitely generated $\CO$-module and $ \CJ$ a maximal ideal of $\CO $ containing $\CI $. In Case ii., in addition, $M$ is graded with respect to the polynomial degree $\mathrm{pol}(.)$.
    \begin{itemize}
        \item[1.] If $\bar e = \lbrace \bar e_1, \dots , \bar e_j\rbrace $ is a basis of the vector space $M/\CJ M$, then its preimage in $M$ is a minimal set of generators of $M$.
        \item[2.] Every minimal set of generators is obtained through 1. 
        \item[3.] In Case i., if $e = \lbrace e_1, \dots , e_j\rbrace$ and $f = \lbrace f_1, \dots, f_j \rbrace$ are two minimals sets of generators of $M$ and  $e_i = \sum_{k=1}^j A_{ij}f_j$ for some matrix $(A_{ij})$, then the matrix $(A_{ij})$ is invertible. In Case ii., one asks additionally that $e$ and $f$ are homogeneous and that the matrix  $(A_{ij})$  has homogeneous entries.
    \end{itemize}
\end{theorem}

\vspace{0.1cm}

\noindent
Recall that a free or projective resolution $(\FM_{\bullet}, d)$ of $\CO/\CI $ is said to be \emph{minimal at $\CJ $}  if $d(\FM_{i+1}) \subset \CJ \FM_{i}$ for all $i \geq 1 $. 
Theorem \ref{thm:basis.theorem} implies that in  case i. the $  \CO$-module resolution $(\FM_{\bullet}, d)$ obtained as in Construction \ref{const:free.resolution} is minimal at $\CJ $ if and only if at each step
we choose a minimal set of generators of $\mathrm{Ker}(d_i)$. 
Also, in this case, a minimal free $ \CO$-module resolution is unique up to isomorphisms, and any free  $ \CO$-module resolution can be decomposed as a direct sum of a trivial complex and a minimal free  $ \CO$-module resolution.
In case ii., the same conclusion holds provided that the generators of $\FMb$ are chosen to be homogeneous at every step.
By construction, all free $\CO$-modules $\FMb$ are graded with respect to the polynomial degree $\mathrm{pol}(.)$, and  $\mathrm{pol}(d) = 0$.


\vspace{0.1cm}

\noindent
We give the following definition of minimality at $\CJ $ of a Koszul-Tate resolution of $\CO/\CI$:
\begin{definition}
Let $ \CJ$ be a maximal ideal containing $\CI $. A Koszul-Tate resolution $(S(\CE), \delta)$ is said to be minimal at $\CJ $ if $\delta(\CE_{i+1})  \subset \CJ \CE_{i} \oplus \Sdeg (\CE)_i$ for all $i\geq 1$.
\end{definition}

\vspace{0.2cm}
\noindent
The following result is straightforward.
\begin{lemma}
\label{lem;asbi}
Let $ (S(\CEb), \delta)$ be a Koszul-Tate resolution of $ \CO/\CI$ such that $\CE_i $ is a free module of finite rank for all $ i \geq 1$.
Then $ (S(\CEb), \delta)$ is minimal  at $ \CJ$ if and only if $ {\mathrm{rk}}(\CE_i)=b_i(\CO/\CI,\CJ)$.
\end{lemma}
\begin{proof}
 Let $ i\in \mathbb N$.
By Theorem \ref{thm:basis.theorem}, the rank of $ \CE_i$ coincides with the dimension of $ \CE_i/\CJ \CE_i$.
The minimality condition means that the differential of the reduced complex  of $ \CO/\CI$ evaluated at $\CJ$ is zero, so that the dimension of $ \CE_i/\CJ \CE_i$ is $b_i(\CO/\CI,\CJ)$. 
\end{proof}

\vspace{0.1cm}
\noindent
For Koszul-Tate resolution obtained through Tate algorithm, i.e. Construction \ref{const:tate.algorithm}, this notion of minimality coincides with the naive one. 

\begin{proposition}
\label{prop:Tateisminimal}
Assume $\CO,\CI $ and $\CJ $ are as in Cases (i) and (ii) in the beginning of this section. A Koszul-Tate resolution $(S(\CE), \delta)$ obtained out of the Tate algorithm, i.e. Construction \ref{const:tate.algorithm}, is minimal at $\CJ $ if and only if for every $i \geq 1$, in the $i$-th step of the construction, a minimal set of generators of $H_i(S(\mathcal E_1 \oplus \cdots \oplus \mathcal E_i) , \delta_i) $ is chosen.
\end{proposition}
\begin{proof} We will present the proof in the case of local ring, i.e. in case (\emph{i}). For the case  (\emph{ii}), the invertible elements of $\CO$ are simply non-zero constant polynomials, and all families considered have to be made of homogeneous elements.

\vspace{0.1cm}

\noindent
  Recall that in case (\emph{i}), all elements not in the maximal ideal $\CJ $ are invertible.
    Let as assume that the Koszul-Tate resolution $(S(\CE), \delta)$ obtained out of the Tate algorithm is not minimal. This means that $\exists i\in \mathbb N$ such that $\delta_{i+1}(\CE_{i+1}) \not\subset \CJ\CE_{i}\oplus \Sdeg(\CE)_i$.
    Recall that when applying the Tate algorithm at step $i$, one needs to choose  $\chi = \lbrace \chi_1, \dots, \chi_k \rbrace$ a set of generators of $H_{i-1}(S(\mathcal E_1 \oplus \cdots \oplus \mathcal E_{i-1}) , \delta_{i-1})$  and to associate to it a basis  $e_\chi = \lbrace e_{1}, \dots, e_{k}  \rbrace$ of $\CE_{i}$. We then extend $ \delta_i$ to $ S(\mathcal E_1 \oplus \cdots \oplus \mathcal E_{i})$  by imposing $\delta_i (e_i)= \chi_i$ for all $ i=1,\dots, k$. 
    If $\delta_{i+1}(\CE_{i+1}) \not\subset \CJ\CE_{i}\oplus \Sdeg(\CE)_i$, then there exists $f\in \CE_{i+1}$ such that 
    $$
    \delta_{i+1} f = \sum_{j=1}^k a_j e_j\, \, \mathrm{mod}\,\, \Sdeg(\CE)_i
    $$
where at least one of the elements $a_1, \dots, a_k$ does not belong to $ \CJ$,  hence is invertible. 
Without any loss of generality, we may assume  $a_1$ is invertible.
After multiplying the previous equality by $a_1^{-1}$, we obtain
    $$\delta_{i+1} \left(  a^{-1}_1 f \right) = e_{1} + \sum_{j=2}^k \frac{a_j}{a_1} e_j\, \, \mathrm{mod}\,\, \Sdeg (\CE)_i.$$
   By applying $ \delta_{i}$ to the previous equality, one obtains
   $$
     \delta_i(e_1) + \frac{a_2}{a_1}\delta_i(e_2) + \dots + \frac{a_k}{a_1}\delta_i (e_k) = \delta_i(g).
    $$
    for some $g \in S^{\geq 2} S(\mathcal E_1 \oplus \cdots \oplus \mathcal E_{i-1}) $. In particular, $ \delta_i(g)= \delta_{i-1}(g) $ and the previous relation means that in $H_{i-1}(S(\mathcal E_1 \oplus \cdots \oplus \mathcal E_{i-1}) , \delta_{i-1})$ we have 
    $$
     \chi_1 + \frac{\overline{a_2}}{\overline{a_1}}\chi_2 + \dots + \frac{\overline{a_k}}{\overline{a_1}}\chi_k = 0.
    $$
    Therefore, $\chi$ is not a minimal set of generators of $H_{i-1}(S(\mathcal E_1 \oplus \cdots \oplus \mathcal E_{i-1}) , \delta_{i-1})$.

\vspace{0.1cm}

\noindent
Let us prove the converse implication.
    Assume that the Tate algorithm is applied, and that at the $i$-th step one chooses a set $\chi$ which is not a minimal set of generators of $H_{i-1}(S(\mathcal E_1 \oplus \cdots \oplus \mathcal E_{i-1}) , \delta_{i-1})$. Then, after relabeling the generators, we have a relation of the form
    $$
    \chi_1 + \overline{b_2}\chi_2 + \dots + \overline{b_k}\chi_k = 0
    $$
    saying that $\chi$ is redundant. On the level of generators of $\CE_{i}$ this equality reads as
    $$
    \left \{e_{1} + \sum_{2}^k b_j e_j + \dots \right \}\hbox{ is $ \delta$-closed.}
    $$
    Here $\dots$ stands for terms in $\Sdeg_{i+1}(\CE)$ which are governed by a choice of representatives of $H_i(S(\mathcal E_1 \oplus \cdots \oplus \mathcal E_i) , \delta_i) $. Since $\delta \Sdeg(\CE)\cap\CE \subset \CI\CE$, this cycle is not exact in $(S(\mathcal E_1 \oplus \cdots \oplus \mathcal E_{i}) , \delta_{i})$. Hence there exists $f\in \CE_{i+2}$ such that $\delta f = e_{1} + \sum_{j=2}^k b_j e_j + \dots$. Therefore, the Koszul-Tate resolution $(S(\CE), \delta)$ is not a minimal Koszul-Tate resolution.
\end{proof}



\vspace{0.1cm}

\noindent
Let us now explain how a Koszul-Tate resolution of $\CO/\CI $ minimal at  $ \CJ$ can be obtained as a quotient of an arbitrary Koszul-Tate resolution $(S(\mathcal E), \delta) $ of $\CO/\CI $ in cases i. and ii.
Let us first interpret the minimality condition in terms of the \emph{reduced complex} $(\CE_\bullet \otimes \CO/\CI, \underline{\delta}) $ of Appendix \ref{app:reduced}), and more precisely in terms of the complex
$  ( \mathcal E_\bullet \otimes \CO / \CJ , \underline{\delta} )$
as in Corollary \ref{cor:reducedWIthJ}.

\vspace{0.1cm}
\noindent
A direct consequence of Proposition  \ref{prop:Tateisminimal} and Lemma \ref{lem;asbi} is that the lower bounds of Theorem \ref{thm:lowernboundonrank} can be always reached:

\begin{corollary}
    Assume $\CO,\CI $ and $\CJ $ are as in Cases (i) and (ii) in the beginning of this section. There exists a Koszul-Tate resolution $(S(\CE), \delta)$ of $ \CO/\CI$ such that for all $ i \in \mathbb N$, the equality $ {\mathrm{rk}}(\CE_i)=b_i(\CO/\CI,\CJ)$ holds.
\end{corollary}

\vspace{0.1cm}

\noindent
By construction, the evaluation at $\CJ $ of the reduced complex is a complex of $\CO/\CJ \simeq \mathbb K $-modules, i.e. a complex of vector spaces. Now, for complex of vector spaces, a normal form exists \cite{Eisenbud}. To be more precise, there exist sequences $ (C_i)_{i \geq 0}$ and $ (H_i)_{i \geq 1}$ of vector spaces with $ C_0=0$ such that for all $i \geq 1 $, we have
 $$  \CE_i \otimes \CO/\CJ = C_{i} \oplus H_i \oplus C_{i-1}  $$
and the differential is given by $\underline{\delta} (c',h,c) = (c , 0 , 0) $ for any $ i \geq 1$ and any $c' \in C_{i} \, , \, h \in H_i \, , \, c \in C_{i-1} $.
In particular, $H_i $ is isomorphic to the $i$-th homology of the reduced complex evaluated at $\CJ $.
We have for all $i \geq 1 $, in view of Theorem \ref{thm:basis.theorem}, a decomposition of $\CE_i$ as follows:
\begin{equation}
\label{eq:decomponE} 
\begin{array}{rcl}\CE_i  &\, = \,& \CO \otimesk C_{i}  \, \oplus \, \CO \otimesk H_i \, \oplus \, \CO \otimesk C_{i-1}  \\
& \, =\, &   \CC_{i} \oplus \CE'_\bullet \oplus  \CC[-1]_{i} \end{array}
\end{equation}
where we have introduced the following free graded  $\CO$-submodules of $\CE_{\bullet}$:  $\CC[-1]_{\bullet} = \oplus_{i=1}^{\infty}\CO\otimesk C_{i}$, whose $i$-th degree component is $\CO\otimesk C_{i-1}$,  $\CC_{\bullet} = \oplus_{i=1}^{\infty}\CO\otimesk C_{i}$
, whose $i$-th degree component is $\CO\otimesk C_{i}$, and $ \mathcal E_\bullet' = \oplus_{i \geq 0} \CO \otimesk H_i$.

\begin{proposition}
\label{prop:quotientisKT}
Assume $\CO,\CI $ and $\CJ $ are as in Cases (i) and (ii) in the beginning of this section.
The quotient of a Koszul-Tate resolution $ (S (\mathcal E), \delta) $ by the ideal  generated by $\CC[-1]_{\bullet}\oplus \delta \left( \CC[-1]_{\bullet}\right) $ 
is a minimal Koszul-Tate resolution of $\CO/\CI $ isomorphic to $ S(\CE')$.
\end{proposition}

\noindent
First, let us show that $ \CC[-1]_{\bullet}\cap \delta \left(  \CC[-1]_{\bullet} \right) = 0$, so the direct sum is well-defined. If the intersection is not zero, then any $a \in  \CC[-1]_{\bullet}\cap \delta \left( \CC[-1]_{\bullet} \right)$ belongs as well to the kernel of $\delta$. Let $s:\CC[-1]_{\bullet} \longrightarrow \CC_{\bullet}$ and $s^{-1}: \CC_{\bullet} \longrightarrow \CC[-1]_{\bullet}$ be the suspension and the desuspension maps. Let $\delta_s$ be the degree $-1$ derivation  of  $S(\CE_{\bullet})$ whose restriction to $\CE$ maps $(a, b, c)$ to $(c, 0, 0)$. Here we used  decomposition \eqref{eq:decomponE}.

\begin{lemma}
\label{lem:ker=0}
    $\ker \delta|_{\CC[-1]_{\bullet}} = 0$.
\end{lemma}
\begin{proof}
    From the definition of $\delta_s$ it follows that the difference $\mathcal D\coloneq (\delta - \delta_s)|_{\CC[-1]_i}$ is valued in $\CJ  \CE_{i-1}  \oplus  S^{\geq 2}(\CE)_{i-1}$.
Let $a\in \ker \delta|_{\CC[-1]_i}$, then $\mathrm{p}_{\CC_{i-1}} \circ \delta(a) = \delta_s (a) + \mathrm{p}_{\CC_{i-1}} \circ \mathcal D (a) = (s + \mathrm{p}_{\CC_{i-1}}\circ \mathcal D)(a) = 0$, where $\mathrm{p}_{\CC_{i-1}}$ is the projection on $\CC_{i-1}$. In Case i. the map $s + \mathrm{p}_{\CC_{i-1}}\circ \mathcal D: \CC[-1]_i \longrightarrow \CC_{i-1}$ has an invertible determinant, therefore it is invertible, so we deduce $a$ = 0. In Case ii. we decompose $a$ into a finite sum of its homogeneous components with respect to $\mathrm{pol}(.)$: $sa = \sum_{i\in B} a^{(i)}$ for some ordered set $B$, such that $\mathrm{pol}(a^{(i)}) < \mathrm{pol}(a^{(j)})$ if $i < j$. Note that $s$ does not change the polynomial degree. Let $a^{(0)}$ be the first component of this decomposition. Then $(sa)^{(0)} = -\mathrm{p}_{\CC_{i-1}}\circ \mathcal D(a^{(0)}) = Xa^{(0)}$, where $X$ is a matrix with components belonging to $\CJ$. Therefore $a^{(0)} = 0$. By repeating the argument, we establish $a = 0$.
\end{proof}
\noindent
From Lemma \ref{lem:ker=0} it follows that $ \CC[-1]_{\bullet} \simeq \delta \left( \CC[-1]_{\bullet} \right)$. Therefore, $\delta \left( \CC[-1]_{\bullet} \right)$ is a free $\CO$-module. Note that $\delta \left( \CC[-1]_{\bullet} \right) \cap \CEb' = 0$, since $\mathrm{p}_{\CC_{i-1}} \circ \delta$ is invertible as established in Lemma  \ref{lem:ker=0}.
\begin{lemma}
    $S(\CEb) = S(\delta (\CC[-1]_{\bullet})\oplus \CEb' \oplus{\mathcal C[-1]}_{\bullet} )$.
\end{lemma}
\begin{proof}
    It is enough to establish that the generators of $\CEb$ can be expressed through the combinations of the generators $\delta (\CC[-1]_{\bullet})\oplus \CEb' \oplus{\mathcal C[-1]}_{\bullet}$. Let $e_1, \dots, e_n$ be a basis of $\CE_i$ and $f_1, \dots, f_m$ be a basis of $\delta (\CC[-1]_{\bullet})\oplus \CEb' \oplus{\mathcal C[-1]}_{\bullet}$. Then it is clear from the previous discussion that $m = n$ and $(f_1, \dots, f_n) = A(e_1, \dots, e_n) + R(e)$, where $A$ is an invertible matrix and $R(e)$ is an $n$-tuple where each component is at least quadratic in $e_i$. By induction in the homological degree we establish the statement of the lemma.
\end{proof}
\vspace{1mm}
\noindent
Let us introduce a map $\mathfrak j \colon S^{\geq 1}(\mathcal C_{\bullet} \oplus {\mathcal C}[-1]_{\bullet}) \longrightarrow S(\CE)$, a graded algebra morphism that extends a map $j\colon \mathcal C_{\bullet} \oplus {\mathcal C}[-1]_{\bullet} \longrightarrow \delta\left (\CC[-1]_{\bullet} \right) \oplus \CEb' \oplus{\mathcal C[-1]}_{\bullet}\,\,\,$, $(a, b) \mapsto ( \delta \circ s^{-1}(a), 0,  b)$. Denote by $\mathcal R$ the image of $\mathfrak j$.

\begin{corollary}
\label{cor:iso.red}
The graded algebra morphism $\mathfrak j $ is injective, and therefore it is a graded algebra isomorphism
$\mathcal R \simeq  S^{\geq 1}(\mathcal C_{\bullet} \oplus {\mathcal C[-1]}_{\bullet}) $. The quotient $ S(\CE)/\mathcal R$ is isomorphic to  $S(\CE')$.
\end{corollary}

\noindent
By construction, $\delta $ restricts to $ \mathcal R$, and we still denote by $ \delta$ this restriction.

\begin{lemma}
The complex $(\mathcal R, \delta)$ is acyclic. 
\end{lemma}
\begin{proof}
Let us equip $ S^{\geq 1}(\mathcal C_{\bullet} \oplus {\mathcal C[-1]}_{\bullet}) $ with the differential $d_s$ given by $d_s(a, b) = (sb, 0)$ for all $ a \in \mathcal C_{\bullet}, b \in {\mathcal C[-1]}_{\bullet}$, then extended to a graded derivation. By construction, $\mathfrak j$ is a chain map. 
Consider the degree $ +1$ $\CO$-linear derivation ${h} $ on $S^{\geq 1}(\mathcal C_{\bullet} \oplus {\mathcal C[-1]}_{\bullet})$
  which is defined on $\mathcal C_{\bullet} \oplus {\mathcal C[-1]}_{\bullet}$ by
  $$ h (a,b) = (0, s^{-1}a) $$ 
  for all $a \in \CC[-1]_{\bullet}    $ and $ b \in {\mathcal C}_{\bullet}$.
  We have $ h \circ \delta_s + \delta_s \circ h= {\mathrm{id}} $ on $\mathcal C_{\bullet} \oplus \widehat{\mathcal C}_{\bullet} $.
 Hence for any $a \in S^{k}(\mathcal C_{\bullet} \oplus \widehat{\mathcal C}_{\bullet})$ we have $(\delta_s \circ h + h \circ \delta_s)a = ka$. The acyclicity of $(S^{\geq 1}(\mathcal C_{\bullet} \oplus \widehat{\mathcal C}_{\bullet}), \delta_s)$ 
 follows.

\end{proof}

\begin{proof}[Proof of Proposition \ref{prop:quotientisKT}]
By Corollary \ref{cor:iso.red},
 there is a natural graded algebra isomorphism
  \begin{equation}\label{eq:desirediso} S(\CEb) \simeq S(\CEb') \otimes\left( \CO \oplus \mathcal R \right) ,\end{equation}
under which the differential $\delta $ decomposes as follows, 
  $$
\xymatrix{  
   S(\CEb')_{i-3} \otimes \mathcal R_{j+1} & & & \\
    & S(\CEb')_{i-2} \otimes \mathcal R_{j+1}& & \\
     & & S(\CEb')_{i-1} \otimes \mathcal R_j& S(\CEb')_i \otimes \mathcal R_j \ar[d]^{{\mathrm{id}} \otimes \delta} 
    \ar[l]^{ \, \delta' \otimes {\mathrm{id}}
    } \ar[ull]\ar[uulll]
    \\
      & & & S(\CEb')_i \otimes \mathcal R_{j-1}\\}  
$$
with the understanding that $ \mathcal R_{0} = \CO  $. Above, the horizontal and vertical components of the differential  $\delta$ are $ \delta'\otimes {\mathrm{id}}$ and ${\mathrm{id}} \otimes \delta $ respectively, where $ \delta'$ is the differential induced by $ \delta$ on  $ S(\CEb')=S(\CEb)/\mathcal R$.
\vspace{0.1cm}
\noindent
Since the homology of all vertical lines is zero in every degree except in degree $0$ where it is $\CO $, the homology of $(S(\CEb),\delta) $ coincides with the homology of
$(S(\CEb'), \delta')$.
Hence $(S(\CE'), \delta') $ is again a Koszul-Tate resolution of $ \CO/\CI$. Last, the reduced complex of $(S(\CEb'),\delta') $ is by construction equipped with a differential valued in $\CJ \CEb' $, so that the reduced complex tensored with $\CO/\CJ$ has a trivial differential. Since its computes a complex whose homology in degree $i$ is $ b_i(\CO/\CI,\CJ)$, this completes the proof. 
\end{proof}

\bibliographystyle{alpha}
\bibliography{bibliography}
  \end{document}